\documentclass[10pt, reqno,amsmath,amsthm,amssymb,amscd]{amsart}
\usepackage{amsmath}
\usepackage{mathrsfs,amssymb, amscd,amsmath,amsthm}
\usepackage[enableskew,vcentermath]{youngtab}
\usepackage{multicol}\multicolsep=0pt
\usepackage{tikz}

\newcommand{\clr}{rgb:black,1;blue,4;red,1}
\newcommand{\wdot}{ node[circle, draw, fill=white, thick, inner sep=0pt, minimum width=4pt]{}}
\newcommand{\bdot}{ node[circle, draw, fill=\clr, thick, inner sep=0pt, minimum width=4pt]{}}

\newcommand{\ob}[1]{\mathsf{#1}}

\newcommand{\down}{\downarrow}

\newcommand{\AOB}{\mathcal{AOB}}

\newcommand{\OB}{\mathcal{OB}}

\newcommand{\AK}{\mathcal{AK}}
\newcommand{\CK}{\mathcal{CK}}
\newcommand{\CB}{\mathcal{CB}}
\newcommand{\AB}{\mathcal{AB}}

\newcommand{\lcap}{
\begin{tikzpicture}[baseline = 3pt, scale=0.5, color=\clr]
        \draw[-,thick] (1,0) to[out=up, in=right] (0.53,0.5) to[out=left, in=right] (0.47,0.5);
        \draw[-,thick] (0.49,0.5) to[out=left,in=up] (0,0);
\end{tikzpicture}
}
\newcommand{\lcup}{
\begin{tikzpicture}[baseline = 6pt, scale=0.5, color=\clr]
        \draw[-,thick] (1,1) to[out=down, in=right] (0.53,0.5) to[out=left, in=right] (0.47,0.5);
        \draw[-,thick] (0.49,0.5) to[out=left,in=down] (0,1);
\end{tikzpicture}
}

 \providecommand{\og}{``}
\providecommand{\fg}{''} \providecommand{\smfandname}{and}

\usepackage{amssymb}
\usepackage{mathrsfs,amssymb}
\usepackage{multicol}\multicolsep=0pt
\usepackage{pstricks,pst-node}
\usepackage[enableskew,vcentermath]{youngtab}

\usepackage[sort]{cite}
\usepackage{xcolor,graphicx}

\def\crulefill{\leavevmode\leaders\hrule height 1pt\hfill\kern 0pt}
\long\def\QUERY#1{%
\leavevmode\newline%
\noindent$\star\star\star$\thinspace\textsf{Comment/Query}\crulefill\newline%
   \space #1\newline\hbox to 120mm{\crulefill}$\star\star\star$\newline}
\newtheorem{Theorem}{Theorem}[section]
\newtheorem{Lemma}[Theorem]{Lemma}
\newtheorem{Cor}[Theorem]{Corollary}
\newtheorem{Prop}[Theorem]{Proposition}

 \theoremstyle{definition}

\newtheorem{Defn}[Theorem]{Definition}

\newtheorem{rem}[Theorem]{Remark}

\newtheorem{Assumption}[Theorem]{Assumption}

\numberwithin{equation}{section}
\theoremstyle{definition}

\definecolor{white}{HTML}{FFFFFF}
\definecolor{darkblue}{HTML}{111199}
\definecolor{darkgreen}{HTML}{336633}
\definecolor{darkred}{HTML}{993333}

\definecolor{darkpurple}{HTML}{995599}

\makeatletter
\def\enumerate{\begingroup\ifnum\@enumdepth>3\@toodeep\else
      \advance\@enumdepth\@ne
      \edef\@enumctr{enum\romannumeral\the\@enumdepth}%
      \topsep\z@\parskip\z@
      \list{\csname label\@enumctr\endcsname}
        {\@nmbrlisttrue\let\@listctr\@enumctr
         \parsep\z@\itemsep\z@\topsep\z@
         \setcounter{\@enumctr}{0}
         \def\makelabel##1{\hss\llap{\rm ##1}}
       }\fi}

\makeatother

\let\bar=\overline
\let\epsilon=\varepsilon
\def\({\big(}
\def\){\big)}

\def\0{\underline{0}}

 \DeclareMathOperator{\Ker}{Ker}

\DeclareMathOperator{\End}{End}

\def\s{\mathfrak s}

\def\t{\mathfrak t}
\def\u{\mathfrak u}
\def\v{\mathfrak v}

\def\Hom{\text{Hom}}

{\catcode`\|=\active
  \gdef\set#1{\mathinner{\lbrace\,{\mathcode`\|"8000%
                                   \let|\midvert #1}\,\rbrace}}
  \gdef\seT#1{\mathinner{\Big\lbrace\,{\mathcode`\|"8000%
                                   \let|\midverT #1}\,\Big\rbrace}}
}
\def\midvert{\egroup\mid\bgroup}
\def\midverT{\egroup\,\Big|\,\bgroup}

\def\Set[#1]#2|#3|{\Big\{\ #2\ \Big| \
           \vcenter{\hsize #1mm\centering #3}\Big\}}

\def\Hom{{\rm Hom}}

\def\Set{{\rm Set}}

\def\s{\mathfrak s}%
\def\t{\mathfrak t}%
\def\v{\mathfrak v}%
\def\Hom{\text{Hom}}%
\def\textsf#1{{\textit{#1}}}%

\begin{document}
\title{Representations of  weakly triangular categories}
\author{Mengmeng Gao, Hebing Rui, Linliang Song}
\address{M.G.  School of Mathematical Science, Tongji University,  Shanghai, 200092, China}\email{1810414@tongji.edu.cn}
\address{H.R.  School of Mathematical Science, Tongji University,  Shanghai, 200092, China}\email{hbrui@tongji.edu.cn}
\address{L.S.  School of Mathematical Science, Tongji University,  Shanghai, 200092, China}\email{llsong@tongji.edu.cn}

\thanks{Both M. Gao and H. Rui are supported  partially by NSFC (grant No.  11571108).  L. Song is supported  partially by NSFC (grant No. 12071346). }
\sloppy \maketitle

\begin{abstract} A new class of locally unital and locally finite dimensional algebras $A$ over an arbitrary  algebraically closed field is discovered. Each of them  admits an  upper finite weakly triangular  decomposition, a generalization of  an upper finite triangular decomposition. Any  locally unital algebra which admits an upper finite Cartan decomposition is Morita equivalent to some special  locally unital algebra $A$  which admits an  upper finite weakly triangular  decomposition.
It is established that the category $A$-lfdmod of locally finite dimensional left $A$-modules is  an upper finite fully stratified category in the sense of Brundan-Stroppel.  Moreover, $A$ is semisimple if and only if its  centralizer subalgebras associated to certain idempotent elements are semisimple. Furthermore, certain  endofunctors  are defined and   give  categorical actions of some   Lie algebras  on the subcategory of $A$-lfdmod
 consisting  of  all objects  which have a finite standard filtration. In the case $A$ is the locally unital algebra associated to one of  cyclotomic oriented Brauer categories,
 cyclotomic Brauer categories  and   cyclotomic Kauffman categories,  $A$ admits an upper finite  weakly triangular decomposition.  This leads to categorifications of  representations of the classical limits of coideal algebras, which  come from all integrable highest weight modules of $\mathfrak {sl}_\infty$ or $\hat {\mathfrak{sl}}_e$. Finally, we study representations of $A$ associated to  either
  cyclotomic Brauer categories  or  cyclotomic Kauffman categories in details, including    explicit criteria on the semisimplicity of $A$ over an arbitrary field, and on  $A$-lfdmod being upper finite highest weight category in the sense of Brundan-Stroppel, and  on Morita equivalence between $A$ and  direct sum of infinitely many (degenerate) cyclotomic Hecke algebras.
 \end{abstract}

\tableofcontents

\section{Introduction}

Throughout the paper,   $\Bbbk$ is  an arbitrary algebraically closed field. All algebras, categories and functors  will be assumed to be  $\Bbbk$-linear.

Motivated by  attempting  to categorify representations of the classical limits of coideal algebras, which  come from all integrable highest weight modules of $\mathfrak {sl}_\infty$ or $\hat {\mathfrak{sl}}_e$  by using representations of cyclotomic Brauer categories $\CB(\mathbf u)$ and cyclotomic Kauffman categories $\CK(\mathbf u)$~\cite{RS3, GRS}, we investigate a locally unital algebra $A$  which admits an upper finite weakly triangular  decomposition in this paper. This work is also influenced by Brundan-Stroppel's work in \cite{BS}.  A notion of  an  upper finite fully stratified category was introduced in ~\cite{BS}  in the study of representations of locally unital  and locally finite dimensional algebras. If   a locally unital algebra is an upper finite based stratified algebra, then  the category  $A$-lfdmod  is  an  upper finite fully stratified category ~\cite{BS}. Conversely, an upper finite fully stratified category with tilting-rigid property can be realized as $A$-lfdmod for some upper finite based stratified algebra $A$.
Introduced by Brundan and Stroppel, an upper finite Cartan decomposition of a locally unital algebra is a generalization of an upper finite triangular  decomposition \cite{SS, BS}. It is important that a locally unital algebra with an upper finite Cartan decomposition is an  upper finite based stratified algebra \cite{BS} although it is not easy to verify that an algebra has an upper finite Cartan decomposition in certain cases. The notion of  upper finite weakly triangular  decomposition given in this paper is another  generalization of an upper finite triangular  decomposition.
While it is not clear to us whether an algebra $A$ coming from   $\CB(\mathbf u)$, $\CK(\mathbf u)$ or
 cyclotomic oriented  Brauer categories $\OB(\mathbf u)$~\cite{BCNR} admits an upper finite Cartan decomposition, it is rather easy to verify that such $A$  does admit an upper finite weakly triangular  decomposition (see Propositions~\ref{COBWA},~\ref{CBWA},~\ref{CKWA}). Our first main result is that  $A$-lfdmod  is  an  upper finite fully stratified category if
  $A$ is
  a locally unital algebra  which admits an  upper finite weakly triangular  decomposition (see  Theorem~\ref{newstratification}).  The second main result is that a locally unital algebra $A$ which admits an upper finite weakly triangular decomposition is semisimple if and only if    its  centralizer subalgebras associated to certain idempotent elements are semisimple (see Theorem~\ref{ss}).
  As  applications, we study  the locally unital algebras  associated to $\CB(\mathbf u)$, $\CK(\mathbf u)$ or
  $\OB(\mathbf u)$. When $A$ comes from either $\CB(\mathbf u)$ or $\CK(\mathbf u)$, we classify  irreducible $A$-modules and show that $A$-lfdmod is an upper finite highest weight category in the sense of Brundan-Stroppel if all the (degenerate) cyclotomic Hecke algebras are semisimple.
  Furthermore,  an explicit  sufficient condition is given  for $A$  being  Morita equivalent to direct sum of infinitely many (degenerate) cyclotomic Hecke algebras. Finally we give an explicit criterion for $A$ being semisimple over  an arbitrary field.
  The last main result is on categorifications of representations of the classical limits of coideal algebras, which  come from all integrable highest weight modules of $\mathfrak {sl}_\infty$ or $\hat {\mathfrak{sl}}_e$  by using representations of $\CB(\mathbf u)$ and  $\CK(\mathbf u)$ (see Theorems~\ref{gsharpisobrauercyc}, \ref{gsharpisoKaucyc}).

Now, we discuss the contents of this paper in details.   For any small finite dimensional $\Bbbk$-linear category $\mathcal A$, one often replaces $\mathcal A$ with its  associated locally unital algebra \begin{equation}\label {Ba1} A=\bigoplus_{ a,  b\in\text{ob }\mathcal A}\text{Hom}_{\mathcal A}( a, b),\end{equation}
whose representations are equivalent to those for $\mathcal A$.
Brundan and Stroppel considered $A$ which admits an  upper finite Cartan decomposition\cite[Definition~5.23]{BS}. In this case, $A$ has  three subspaces $A^\flat$, $A^\circ$ and  $A^\sharp$ satisfying certain axioms. The key point is that   $A^\circ$ is a locally unital subalgebra of $A$
 and  $A^\flat$ (resp., $A^\sharp$) is projective right (resp., left) $A^\circ$-module.
Many $\Bbbk$-linear categories come into their theory.  Archetypal examples include the Brauer category~\cite{LZ, RS4},  the oriented Brauer category~\cite{BCNR},   the oriented skein category \cite{Br} and  those triangular categories in \cite{SS}.
All locally unital algebras associated to these categories admit upper finite triangular decompositions.

On the other hand, when  $q\in \Bbbk^\times $ is not a root of unity, the category of right  modules for  the Birman-Murakami-Wenzl algebra~\cite{BW} shares similar properties to those for the Brauer algebra~\cite{B} over $\mathbb C$.  See \cite{CVM, RuiS3} on  the classification of blocks for  these algebras.  Since the  Brauer category (resp., Kauffman category~\cite{VT}) is  the category  version of Brauer  (resp., Birman-Murakami-Wenzl) algebras, it is natural to expect that $A$-lfdmod  is also an upper finite fully stratified category where $A$ is associated to the Kauffman category. In fact, it is the case. However, it is not clear to us whether $A$ admits an upper finite Cartan decomposition although $A$  has  three  similar subspaces.

We introduce the notion of  an  upper finite weakly triangular decomposition for a locally unital algebra $A$ in Definition~\ref{WT}.
 In this case, we have three subspaces   $A^-, A^\circ$ and $A^+$ and we do not assume that $A^\circ $ is a subalgebra! Any locally unital algebra  which admits an upper finite triangular decomposition has an  upper finite weakly triangular decomposition. However, the converse is not true in general since $A^\circ $ may not be a subalgebra!

 Influenced by Brundan-Stroppel's work in  \cite{BS}, we consider   a class of quotient algebras   $A_{\preceq  \ob a}$  of $A$ such that each  $A_{\preceq \ob a}$ still admits an upper finite weakly triangular decomposition and  $A_{\preceq  \ob a}\bar 1_{\ob a}$ (resp., $\bar 1_\ob a A_{\preceq  \ob a}$) is a projective right (resp., left) $\bar A_{\ob a}$-module in general (see Definition~\ref{bara1} and Lemma~\ref{exactj}).   This leads us to construct  exact tensor functors and exact hom functors  from  $\bar A_{\ob a}$-fdmod to $A_{\preceq \ob a}$-lfdmod and hence to $A$-lfdmod, where $\bar A_{\ob a}$-fdmod is the category of all finite dimensional left $\bar A_{\ob a}$-modules. Such functors are  called
 standardization functors and costandardization  functors in the sense of \cite{LW}.  Consequently, we are able to prove that $A$-lfdmod  is an upper finite fully stratified category if $A$ admits an upper finite weakly triangular decomposition. This result  has its own interests and important  applications elsewhere.
  For example,  we show in \cite{GRS1} that  the representations of  $\OB(\mathbf u)$  with suitably chosen parameters can be used to give  tensor product categorifications (in the general sense of Losev-Webster \cite{LW})    of integrable  lowest weight and integrable  highest weight representations for the Lie algebra $\mathfrak {sl}_\infty$ or $\hat{\mathfrak {sl}}_p$ introduced by Webster in \cite{W}. When $\Bbbk=\mathbb C$, such a result was  expected by Brundan et. al in \cite[page~77]{BCNR}.

After proving  that  $A$-lfdmod  is an upper finite fully stratified category, we  consider its subcategory  $A\text{-mod}^\Delta$  consisting  of all  objects with  a finite standard filtration.
We focus on the locally unital algebras $A$ which come from   certain quotients of
  two kinds of strict monoidal categories, say  $\mathcal C_1$ and $ \mathcal C_2 $ in  section~5.
  Under some mild assumptions, we are able to prove that
 there is a short exact sequence of exact functors  from the category $\bar A^\circ$-fdmod to the category $A$-lfdmod in \eqref{keyses}, where $\bar A^\circ$ is the direct sum of all possible $\bar A_{\ob a}$'s. Consequently, there is a categorical $\mathfrak g^\sharp$-action on  $A\text{-mod}^\Delta$ if
 we assume that there is a   categorical $\mathfrak g$-action on the category $\bar A^\circ$-pmod of finite generated projective modules for some  Kac-Moody algebra $\mathfrak g$, where
 $\mathfrak g^\sharp$ is a suitable Lie subalgebra of $\mathfrak g$. In other words,  $A\text{-mod}^\Delta$ can be used to categorify certain representations of $\mathfrak g^\sharp$.
In particular,  when $\mathcal C_1$ (resp.,  $\mathcal C_2$ ) is the   affine Brauer category (resp.,  affine  Kauffman category),  $A$ comes from   $\CB(\mathbf u)$ (resp., $\CK(\mathbf u)$).
  In this case,
  $\mathfrak g$ is the direct sum of certain $\mathfrak{sl}_\infty$ or $\widehat{\mathfrak{sl}}_e$ for some positive integer $e$.   Such a result   can be regarded as  categorifications related to  cyclotomic Birman-Murakami-Wenzl  algebras~\cite{Olden} and cyclotomic Nazarov-Wenzl algebras~\cite{AMR}.
 When  $\mathfrak g^\sharp= \mathfrak g$, the category
  $A$-lfdmod  is  Morita equivalent   to the direct sum of  categories of left  modules for infinitely many  (degenerate) cyclotomic Hecke algebras. However, the most interesting case is $\mathfrak g^\sharp \neq \mathfrak{g}$. In this case,     ($\mathfrak{g}, \mathfrak g^\sharp$)    is (the infinite rank limit of) the quantum symmetric pair (in the sense of Letzter\cite{Le} in the finite type or Kolb \cite{Ko} in the Kac-Moody setting) at the limit $q$ to 1.
 Using the $\mathfrak g^\sharp $-action on  $A\text{-mod}^\Delta$ as above, it is possible to   connect deeper representations (decomposition numbers, etc.) of cyclotomic Birman-Murakami-Wenzl  algebras  and cyclotomic Nazarov-Wenzl algebras over the complex field $\mathbb C$ to the Kazhdan-Lusztig theory ($\imath$-canonical basis, etc.) of certain $ \mathfrak g^\sharp $-modules studied in \cite{B1, BWa,BWan,BWang,BWW}. This is a project in progress.

    Thanks to the results on the semisimplicity of  cyclotomic Nazarov-Wenzl algebras and cyclotomic Birman-Murakami-Wenzl algebras in \cite{RuiS, RuiS1, RuiS2}, we are able to give an explicit criterion  on the semisimplicity of $A$ associated to $\CB(\mathbf u)$ and $\CK(\mathbf u)$. In the non-semisimple case, blocks of $A$ are determined. In particular, when the (degenerate)  cyclotomic Hecke algebras are semisimple and $\text{char}~ \Bbbk\neq 2$, we use results  on the classification of blocks for cyclotomic Nazarov-Wenzl algebras and cyclotomic Birman-Murakami-Wenzl algebras  in \cite{RuiS4, RuiS3} to classify blocks of $A$ completely. In this case,   $A$-lfdmod is an upper finite highest weight
    category in the sense of \cite{BS}.

The paper is  organized  as follows.  In section~2,   we introduce  the notion of an upper finite weakly triangular decomposition for a  locally unital $\Bbbk$-algebra.
As an example, we prove that $\OB(\mathbf u)$ admits this structure.
 In section~3, we give  a systematic understanding of  $A$-lfdmod  if $A$ admits  an  upper finite weakly triangular decomposition. In particular, we prove that  $A$-lfdmod
is an upper finite fully stratified category in the sense of  Brundan-Stroppel. We also give a criterion on the semisimplicity of $A$.
  Weakly cellular structures of   centralizer subalgebras of $A$ will be given in section~4 under the assumption~\eqref{ssgm}.
    Using truncation functors, we establish an explicit relationship  between  $A$-lfdmod and  representations  of  all  centralizer subalgebras of $A$. In section~5,  we investigate properties of certain endofunctors on  $A\text{-lfdmod}$. This gives a  categorical $\mathfrak g^\sharp$-action on  $A\text{-mod}^\Delta$. In section~6, we apply general theory in sections~2-5 to study   representations of $\CB(\mathbf u)$ and $\CK(\mathbf u)$ and categorify those  representations of the classical limits of coideal algebras, which come from all integrable highest weight modules of  $\mathfrak {sl}_\infty$ or $\hat {\mathfrak{sl}}_e$.

\section{  Upper finite weakly triangular categories }

 In this paper,  $\mathcal A$ is a  small finite dimensional $\Bbbk$-linear category  whose object set is $J$. Then  $$\dim_{\Bbbk} \Hom_{\mathcal A}( a, b)<\infty,$$   for all $a,  b\in J$.
  Assume that there is an equivalence relation $\sim$ on $J$ such that each equivalence class contains finitely many   objects. Define  \begin{equation}
\label{pos} I=J/\sim ,\end{equation}  the set of all  equivalence classes of $J$.  For any  $\ob a\in I$, let
  \begin{equation} 1_{\ob a}=\sum_{b\in \ob a} 1_b,\end{equation}
where  $1_b$ is the identity morphism in $\Hom_{\mathcal A}( b, b)$. When each equivalence class contains a unique object,  $I$ can be identified with $J$. In this case,  we replace each equivalence  class by the object in it.

 The  $\Bbbk$-algebra $A$ associated  to  $\mathcal A$ is given in \eqref{Ba1}, where the multiplication  is induced by the composition of morphisms in $\mathcal A$.
 For any subspace $B\subseteq A$ and any $a, b\in J$,  let $B_{a, b}=1_a B 1_b$.
  Similarly,  let    $B_{\ob a, \ob b}=\bigoplus_{c\in \ob a, d\in \ob b} 1_{c} B 1_{d}$ for any   $\ob a, \ob b\in I$.  Then $B_{\ob a, \ob b}=1_{\ob a} B 1_{\ob b}$.
 When $b=a$ (resp., $\ob a=\ob b$), $B_{a, b}$ (resp., $B_{\ob a, \ob b}$) is denoted  by $B_a$ (resp., $B_{\ob a}$). Then $A_{a, b} =\Hom_{\mathcal A}( b, a)$ and \begin{equation}\label {Bal2} A= \bigoplus_{ a,  b \in J} A_{a, b}.\end{equation}
 So,  $A$ is  a locally unital and locally finite dimensional  $\Bbbk$-algebra, and  the set  $\{1_{ a} \mid  a \in J\}$  serves as the system of  mutually orthogonal idempotents of $A$. Conversely,  any locally unital and locally finite dimensional algebra can be obtained in this way~\cite[Remark~2.1]{BRUNDAN}.

\subsection{Upper finite weakly triangular decompositions} Let $A$ be given in \eqref{Bal2}.
 \begin{Defn}\label{WT} The data $(I, A^-, A^\circ, A^+)$ is called a
 \textsf{weakly triangular decomposition}  of  $A$   if:
\begin{itemize} \item[(W1)]  $(I, \preceq)$ is a poset, where $I$ is given in \eqref{pos}.
\item[(W2)] $A^\diamond$ are subspaces (not subalgebras!) of $A$, where $\diamond\in \{\pm,\circ \}$. Moreover,
 $A^\pm =\bigoplus_{   b, c \in J}  A_{b,c}^\pm $ and  $ A^\circ =\bigoplus_{\ob a\in I}\bigoplus_{b, c\in \ob a}   A_{b, c}^\circ $.
\item [(W3)] $ A^-_{\ob a,  \ob b}=0$ and   $A^+_{\ob b,    \ob a}=0$ unless $    \ob a\preceq   \ob b$.  Furthermore,  $ A^-_{\ob a }=A^+_{\ob a}=\bigoplus_{c\in  \ob a}\Bbbk 1_{c}$.
    \item[(W4)]  $A^-\otimes_{\mathbb K }  A^\circ \otimes_{\mathbb K}  A^+\cong A$
as  $\Bbbk$-spaces where $\mathbb K=\oplus_{  b\in J}  \Bbbk 1_{ b}$.   The required isomorphism  is given by the multiplication on $A$.\end{itemize}\end{Defn}
Following \cite{BS}, $(I, \preceq) $ is upper finite if $\{\ob b\in I\mid \ob a\preceq \ob b\}$ is finite for all $\ob a\in I$.
We call $(I, A^-, A^\circ, A^+)$  an upper finite weakly triangular decomposition if $(I, \preceq)$ is upper finite. In this case,  the category $\mathcal A$ is called  an upper finite  weakly triangular category and
  the algebra $A$ is said to  admit  an upper finite weakly triangular decomposition. Furthermore,
if $I=J$,  $A^-$,   $A^\circ$, $A^+$, $A^-A^\circ$ and $A^\circ A^+$ are locally unital subalgebras of $A$, then
    $A$ admits a    \textsf{triangular decomposition} in the sense of  \cite[Definition~5.24]{BS}. The algebra $A$ associated to a triangular category in \cite[Definition~4.2]{SS}  admits an upper finite  triangular decomposition. Furthermore, $A^\circ$ is assumed to be semisimple in~\cite{SS}.
     Any triangular decomposition is a weakly triangular decomposition and the converse is not true in general since $A^\circ$ may  not be  a subalgebra. Similar reason shows that a weakly triangular decomposition may not be a Cartan decomposition in the sense of Brundan-Stroppel~\cite{BS}.
Archetypal examples on  upper finite  weakly triangular categories are triangular categories in \cite{SS} together with   cyclotomic Brauer categories~\cite{RS3}, cyclotomic Kauffman categories~\cite{GRS} and cyclotomic oriented Brauer  categories~\cite{BCNR}. However, it is not clear to us   whether $A$ admits either  an upper finite triangular decomposition  or an upper finite Cartan decomposition  if $A$ comes from one of  cyclotomic Brauer categories, cyclotomic Kauffman categories and cyclotomic oriented Brauer  categories.

\subsection{Cyclotomic oriented Brauer categories} \label{COBC} Fix  $\mathbf u=(u_1, u_2, \ldots, u_m)\in \Bbbk^{m}$ where $m$ is any fixed positive integer.  Let  $\OB(\mathbf u)$ be the  cyclotomic oriented Brauer category associated to the polynomial $f(t)=\prod_{i=1}^m(t-u_i)$
 (denoted by $\OB^f(\delta_1, \delta_2,\ldots, \delta_{m})$ in  \cite{BCNR}). The objects in
 $\OB(\mathbf u)$ are given by \begin{equation}\label{oobj} J=\langle \uparrow, \down\rangle, \end{equation}
 the set of all finite sequences of the symbols $\uparrow, \down$, including the empty word $\emptyset$.
  In order to explain that  $\OB(\mathbf u)$ is  an upper finite weakly triangular category,
  we need a  basis of any  morphism space in $\OB(\mathbf u)$~\cite{BCNR}.

   We begin by recalling  an   $(a, b)$-Brauer diagram \cite[Definition~ 2.1]{LZ} for any $a, b\in \mathbb N$. It is a diagram on which $a+b$ points are placed on  two parallel horizontal lines, and $a$  points on the lower line and $b$ points on the upper line, and each point joins precisely to one other point.
  If two points at the upper (resp., lower) line join each other, then this strand  is called a cup (resp.,
cap). Otherwise, it is called a vertical strand.

Any $(a, b)$-Brauer diagram can also  be considered as a partitioning of the set $\{1,2,\ldots, a+b\}$ (hence there is no $(a, b)$-Brauer diagram if $a+b$ is odd). Two $(a, b)$-Brauer diagrams are said to be equivalent if they give the same partitioning of
$\{1,2,\ldots, a+b\}$.

 An oriented Brauer diagram is   obtained by adding consistent orientation to each strand in  a    Brauer diagram as above.
 Given any oriented Brauer diagram $d$, let $a$ (resp., $b$) be the element in $J$ which is indicated
 from the orientation of the endpoints of the lower line (resp., upper line) and $d$ is called of type $a\rightarrow b$.
 For example, the following diagram is of type  $\uparrow\uparrow\down \down \rightarrow  \down\uparrow$:

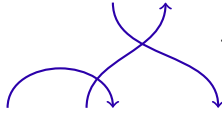
 \begin{center}\begin{tikzpicture}[baseline = 25pt, scale=0.35, color=\clr]
        \draw[->,thick] (5,0) to[out=up,in=down] (8,4);
         \draw[<-,thick] (10,0) to[out=up,in=down] (6,4);
         \draw[->,thick] (2,0) to[out=up,in=left] (4,1.5) to[out=right,in=up] (6,0);
           \end{tikzpicture}.
\end{center}
Two oriented Brauer diagrams of type $a\rightarrow b$ are said to be  equivalent if their underlying Brauer diagrams are equivalent.

A dotted  oriented Brauer diagram of type $a\rightarrow b$ is a oriented Brauer diagram of type $a\rightarrow b$  such that
there are finitely many $\bullet$'s (called dots) on each strand.
A normally ordered dotted oriented Brauer diagram  of type $a\rightarrow b$ is  a dotted oriented Brauer diagram  of type $a\rightarrow b$ such that:
 \begin{itemize}\item  whenever a dot appears on a strand, it is on the outward-pointing boundary,
\item there are at most $ m-1$ dots on each strand.\end{itemize}
     Suppose $m=4$.  The  right one is a normally ordered dotted oriented Brauer diagram of  type  $\uparrow\uparrow\down \down \rightarrow \down\uparrow$  and the left one is not:
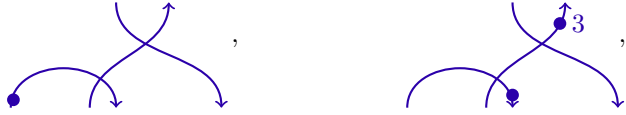
\begin{center}
     \begin{tikzpicture}[baseline = 25pt, scale=0.35, color=\clr]
        \draw[->,thick] (5,0) to[out=up,in=down] (8,4);
         \draw[<-,thick] (10,0) to[out=up,in=down] (6,4);
         \draw  (2.1,0.3) \bdot;
         \draw[->,thick] (2,0) to[out=up,in=left] (4,1.5) to[out=right,in=up] (6,0);
           \end{tikzpicture}~,
    \qquad\qquad\qquad
     \begin{tikzpicture}[baseline = 25pt, scale=0.35, color=\clr]
        \draw[->,thick] (5,0) to[out=up,in=down] (8,4);
         \draw[<-,thick] (10,0) to[out=up,in=down] (6,4);
         \draw[->,thick] (2,0) to[out=up,in=left] (4,1.5) to[out=right,in=up] (6,0);
          \draw  (6,0.48) \bdot;
           \draw  (7.8,3.2) \bdot;
            \draw (8.5,3.2) node{$3$};
           \end{tikzpicture},
\end{center}
 where $ \bullet~ 3$ represents that   there are  three  $ \bullet$'s on the outward-pointing boundary of this strand.
 Two  normally ordered dotted oriented Brauer diagrams are said to be equivalent if the underlying oriented Brauer diagrams are equivalent and there are the same number of dots on their corresponding strands.
 Let $ \mathbb{OB}_{a,b}$ be the set of all equivalence classes of   normally ordered dotted  oriented Brauer diagrams  of type $ a\rightarrow b$.
Thanks to  \cite[Theorem~1.5]{BCNR},
   $\Hom_{\OB(\mathbf u)} ( a,   b)$ has  basis given by $ \mathbb{OB}_{a,b}$ and  two equivalent diagrams represent the same morphism in  $\OB(\mathbf u)$. So,  each equivalence class can be identified    with any element in it.  Let $\mathbb{OB}=\bigcup_{a,b\in J}\mathbb{OB}_{a,b}$.

 \begin{Prop}\label{COBWA} The cyclotomic oriented Brauer  category $\OB(\mathbf u)$ is  an upper finite weakly triangular category.
 \end{Prop}
\begin{proof} Let $A$ be the algebra associated to $\OB(\mathbf u)$. Suppose    $ a= a_1a_2\cdots  a_h$, where $a_i\in \{\uparrow, \downarrow\}$, $1\le i\le h$.
Define
 $\ell_\down( a)=|\{i:  a_i=\down\}|$ and $ \ell_\uparrow( a)=|\{i:  a_i=\uparrow\}| $,  where $|D|$ is the cardinality of  a set $D$.
 When $h=0$, i.e.,  $a$ is the empty word,  $\ell_\down( a)= \ell_\uparrow( a)=0$.
For any $a, b\in J$,   write  $ a\sim  b$ if $(\ell_\uparrow( a),\ell_\down(a))=( \ell_\uparrow( b),\ell_\down(b))$.
Then $\sim$ is an equivalence relation on $J$. As sets, $\mathbb N^2  \cong J/\sim$. For the data in  (W1), define
 \begin{equation}\label{ocbap} I=\mathbb N^2 \text{ and   $(r,s)\preceq (r_1,s_1)$ if $r=r_1+k$ and $s=s_1+k$ for some $k\in\mathbb N$}.\end{equation}
  Obviously,   $\(I, \preceq)$ is upper finite. For the data in (W2), define
\begin{itemize}\item [(a)] $A^-$:  the $\Bbbk$-space with basis consisting  of  elements in $ \mathbb{OB}$ on which there are  neither  caps  nor crossings among vertical strands and there are no dots on vertical strands.
\item [(b)] $A^\circ:$ the $\Bbbk$-space with basis consisting of elements in $ \mathbb{OB}$  on which there are neither  cups nor caps.
 \item [(c)] $A^+$ : the $\Bbbk$-space with basis consisting  of elements in $ \mathbb{OB}$  on which there are neither cups  nor crossings among vertical strands and  there are no dots on vertical strands.
\end{itemize}
 Obviously, $\{1_ a\}=\mathbb {OB}\cap A^-_{  a}=\mathbb {OB}\cap A^+_{ a}$, where $1_a$ is defined in an obvious way. For example,  if $a=\uparrow\uparrow\downarrow$, then  $$1_a=\begin{tikzpicture}[baseline = 15pt, scale=0.35, color=\clr]
          \draw[->,thick] (1.5,1.1) to[out=up, in=down] (1.5,2.2);
            \draw[->,thick] (2.5,1.1) to[out=up, in=down] (2.5,2.2);
         \draw[<-,thick](3.5,1.1)to[out= up,in=down](3.5,2.2);
           \end{tikzpicture}.$$
By the definitions of $A^\pm$  in (a) and (c), $\mathbb {OB}\cap A_{a, b}^-=\mathbb {OB}\cap A_{b, a}^+=\emptyset$ if $(\ell_\uparrow( a),\ell_\down(a))\npreceq( \ell_\uparrow( b),\ell_\down(b))$. Moreover, $A_{b, c}^{\pm}=0$ if $b\sim c$ and $b\neq c$. So,   (W3) follows.

   For any $(g,h,k)\in (\mathbb {OB}\cap A^-1_b, \mathbb {OB}\cap A^\circ_{b,c},\mathbb{OB}\cap 1_cA^+ )$, the composition $g\circ h\circ k$ is not  normally ordered in general, i.e., the dots on the vertical strands may not be on  the outward-pointing boundary. For example: $$g\circ h \circ k=\begin{tikzpicture}[baseline = 25pt, scale=0.35, color=\clr]
        \draw[<-,thick] (2,0) to[out=up,in=down] (0,5);
       \draw[<-,thick] (0,0) to[out=up,in=left] (2,1.5) to[out=right,in=up] (4,0);
        \draw[->,thick] (2,5) to[out=down,in=left] (3,4) to[out=right,in=down] (4,5);
       \draw (0.2,0.8) \bdot;
        \draw  (3.8,4.4) \bdot;
        \draw  (1.2,2.1) \bdot;
           \end{tikzpicture}$$ if $(g, h, k)=\left(\ \ \begin{tikzpicture}[baseline = -0.5mm]
 \draw[-,thick,darkblue] (0,0) to[out=down,in=left] (0.28,-0.28) to[out=right,in=down] (0.56,0);
 \node at (0.56,0) {$\color{darkblue}\scriptstyle\bullet$};
  \draw[-,thick,darkblue] (0,0) to (0,.2);
  \draw[<-,thick,darkblue] (-0.1,-0.28) to (-0.1,.2);
   \draw[->,thick,darkblue] (0.56,0) to (0.56,.2);
 \end{tikzpicture},~\begin{tikzpicture}[baseline = -0.5mm]
 \node at (0,0) {$\color{darkblue}\scriptstyle\bullet$};
  \draw[<-,thick,darkblue] (0,-0.28) to (0,.2);
 \end{tikzpicture},~\begin{tikzpicture}[baseline = 2.5mm]
 \draw[<-,thick,darkblue] (0.28,0) to[out=90,in=-90] (-0.28,.6);
 \draw[<-,thick,darkblue] (-0.28,0) to[out=90,in=-90] (0.28,.6);
 \node at (-0.26,0.15) {$\color{darkblue}\scriptstyle\bullet$};
  \draw[-,thick,darkblue] (0.28,.6) to[out=up,in=left] (0.56,0.85) to[out=right,in=up] (0.56+0.28,0.6);
  \draw[-,thick,darkblue] (0.56+0.28,0.6) to (0.56+0.28,0);
\end{tikzpicture}\  \ \right)$.
Thanks to \cite[Theorem~1.5]{BCNR},   $\OB(\mathbf u)$ is a quotient category of the affine  oriented Brauer category $\AOB$,    a $\Bbbk$-linear strict monoidal category.
By the defining relations of $\AOB$ in \cite[(1.10)]{BCNR}, one has
\begin{equation}\label{movedot1}
 \begin{tikzpicture}[baseline = 7.5pt, scale=0.5, color=\clr]
            \draw[-,thick] (0,0) to[out=up, in=down] (1,2);
            \draw[->,thick] (0,2) to[out=up, in=down] (0,2.2);
            \draw[-,thick] (1,0) to[out=up, in=down] (0,2);
            \draw[->,thick] (1,2) to[out=up, in=down] (1,2.2);
             \draw(1,1.7)\bdot;
        \end{tikzpicture}
        ~=~~  \begin{tikzpicture}[baseline = 7.5pt, scale=0.5, color=\clr]
            \draw[->,thick] (0,0) to[out=up, in=down] (1,2);\draw[-,thick] (0,0) to[out=up, in=down] (0,-0.2);
             \draw[->,thick] (1,0) to[out=up, in=down] (0,2);\draw[-,thick] (1,0) to[out=up, in=down] (1,-0.2);
                        \draw(0,0.1)\bdot;
        \end{tikzpicture} +~   \begin{tikzpicture}[baseline = 7.5pt, scale=0.5, color=\clr]
            \draw[->,thick] (0,0) to[out=up, in=down] (0,2);
            \draw[->,thick] (1,0) to[out=up, in=down] (1,2);
                   \end{tikzpicture}~,\quad \quad  \begin{tikzpicture}[baseline = 7.5pt, scale=0.5, color=\clr]
            \draw[<-,thick] (0,0) to[out=up, in=down] (1,2);
            \draw[->,thick] (0,2) to[out=up, in=down] (0,2.2);
            \draw[-,thick] (1,0) to[out=up, in=down] (0,2);
            \draw[-,thick] (1,2) to[out=up, in=down] (1,2.2);
             \draw(0,1.7)\bdot;
        \end{tikzpicture}
               ~=~\begin{tikzpicture}[baseline = 7.5pt, scale=0.5, color=\clr]
            \draw[<-,thick] (0,0) to[out=up, in=down] (1,2);
             \draw[->,thick] (1,0) to[out=up, in=down] (0,2);
                        \draw(0.9,0.4)\bdot;
        \end{tikzpicture}~+~
       \begin{tikzpicture}[baseline = 10pt, scale=0.5, color=\clr]
         \draw[->,thick] (2,2) to[out=down,in=right] (1.5,1.5) to[out=left,in=down] (1,2);
            \draw[->,thick] (2,0) to[out=up, in=right] (1.5,0.5) to[out=left,in=up] (1,0);
        \end{tikzpicture}.
    \end{equation}
and \eqref{movedot1} explain how to move dots pass crossings in any diagram~\cite[page~81]{BCNR}.  Consequently,  dots  can be slid freely up to a linear combination of some elements in $\mathbb {OB}$ with fewer dots.
This shows that  the multiplication map from $A^-\otimes_{\mathbb K }  A^\circ \otimes_{\mathbb K}  A^+$ to $A$ sends  any basis element $g\otimes h\otimes k$,  $(g,h,k)\in(\mathbb {OB}\cap A^-1_b, \mathbb {OB}\cap A^\circ_{b,c},\mathbb{OB}\cap 1_cA^+ )$,  to a unique corresponding  basis element in $\mathbb {OB}$ (up to  a  linear combination of some elements in $\mathbb {OB}$  with fewer dots). Furthermore, the multiplication map is obviously surjective.
So, it is the required  isomorphism and (W4) holds.
\end{proof}

If we allow  $m=1$, then $(I, A^-, A^\circ, A^+)$ given in the proof of Proposition~\ref{COBWA} is the same as that for the oriented Brauer category in \cite{Re}. In fact, it is an upper finite triangular decomposition in the sense of \cite{BS}.

\subsection{Quotient algebras}
To conclude this section, we give some elementary observations about the structure of $A$, where $A$  always admits an upper finite weakly triangular decomposition.

\begin{Defn}\label{defniogfxhy} For any $ a\in J$ and $\ob a\in I$, define
  \begin{itemize} \item $Y( a)=\bigcup_{   b\in J} Y(  b,  a)$, where $Y(b, a)$ is a basis of  $A^-_{b, a}$.
   \item  $ X(   a)=\bigcup_{ b\in J} X(   a,    b)$,   where $X(a, b)$ is a basis of $A^+_{a, b}$.
   \item  $ H(  \ob a)=\bigcup_{b, c\in\ob  a} H(b, c)$, where $H(b, c)$ is a basis of  $A^\circ_{b, c}$.
    \end{itemize}\end{Defn}
  Since we are assuming that $\mathcal A$ is   finite dimensional, all   $Y(b, a)$, $X(b, a)$ and $H(a, b)$ are finite sets.
  The partial order $\preceq$ on $I$ induces a partial order on $J$ such that  $a\prec b$ if $a\in \ob a$ and $b\in \ob b$ with $\ob a\prec \ob b$.
 Thanks to (W3), $X(a,a)=Y(a,a)=\{1_a\}$ for any $a\in J$, and  $Y(  b,  a)=X(  a,  b)=\emptyset$ unless  $  b \preceq  a$.
The following  result follows from Definition~\ref{WT}, immediately.

\begin{Lemma}\label{cellbasis}
For all $  a,  c\in J$,   $  A_{a, c}$ has basis
 $   \{ yhx\mid (y, h, x)\in \bigcup_{  d\succeq  c,   b\succeq a, b\sim d} Y(  a,   b)\times H(  b,d)\times X(d,  c)\}$.
\end{Lemma}

 \begin{Defn}\label{ideals}  For any $\ob a\in I$,  and
  $\diamond\in \{\succ, \not\preceq, \nprec\}$,  let $A^{\diamond \ob a}$ be  the two-sided ideal of $A$ generated by $I^{\diamond \ob a}=\{1_\ob b\mid \ob b\diamond \ob a\}$.    \end{Defn}
A subset  $K\subseteq I$ is said to be an  upper set (or a coideal) of $I$ if $\{ \ob b\mid  \ob a\preceq  \ob b\}\subseteq K$ for any $  \ob a\in K$.   Obviously,  $I^{\succ\ob a}$,   $I^{\npreceq\ob a}$ and  $I^{\nprec\ob a}$ are upper sets.

\begin{Lemma}\label{W5}
   For any $h_1,h_2\in A_{\ob a}^\circ $, there is an  $h_3\in  A_{\ob a}^\circ $ such that
$h_1h_2\equiv h_3 \pmod {A^{\succ \ob a}}$.

\end{Lemma}
 \begin{proof}The result follows  immediately from (W3) and Lemma~\ref{cellbasis}.
\end{proof}

\begin{Lemma}\label{linearcom} Suppose  $\ob a\in I$. \begin{itemize}\item[(1)] The two-sided ideal  $A1_{\ob a}A$ is the same as the two-sided ideal generated by all $1_b, b\in \ob a$.
\item [(2)] Any  $g\in A1_{\ob a}A$ is  a linear combination of elements  $yhx$, where
 $(y, h, x)\in Y( c)\times  H( c,d)\times X( d)$ such that $ c,d\in\ob b$ with $\ob  b\succeq \ob a$.\end{itemize}
\end{Lemma}
\begin{proof} (1) is obviously since $1_{\ob a}=\sum_{b\in \ob a} 1_b$ and all   $1_b$'s are pairwise orthogonal idempotents.
We  prove (2) by induction on $\preceq$. Thanks to
 Lemma~\ref{cellbasis} and (W3), we can assume  $g=y_1h_1x_1y_2h_2x_2$, where
  $(y_1, h_1, x_1)\in Y( b_1)\times  H( b_1,c_1)\times  X(c_1, a)$ and  $(y_2, h_2, x_2)\in Y( a, b_2)\times H( b_2,c_2)\times X( c_2)$ with $a\in\ob a$ and $b_i, c_i\in \ob b_i$ for some $\ob b_i\in I$ with $\ob b_i\succeq\ob a$, $i=1,2$. If  $ \ob b_i\succ \ob  a$ for some $i=1,2$, then  $g\in A 1_{\ob b_i}A$, and the result follows  from induction assumption. So, it is enough to consider the case $\ob b_1=\ob b_2=\ob a$. Thanks to  (W3),
   $x_1=1_{a} =y_2$ and hence  $g=y_1h_1h_2x_2$.
By Lemma~\ref{W5}, $g=y_1h_3x_2+h_4$ for some $h_3\in A^\circ_{\ob a} $ and $h_4\in A1_{\ob c}A$ with $\ob c\succ  \ob a$.  So, the result follows from induction assumption again.
\end{proof}

 \begin{Lemma}\label{twoside}  For any upper set $I^\uparrow$ of $I$,  let  $A^\uparrow$ be the   two-sided ideal
of $A$ generated by $\{1_{ \ob a}\mid  \ob a\in I^\uparrow\}$.  Then  $S_{a,c}$ is a $\Bbbk$-basis of  $ A^\uparrow_{a,c}$ for all $a,c\in J$, where $$S_{a, c}:=\{yhx\mid (y,h,x)\in \bigcup_{ b,d\in\ob a, \ob a\in I^\uparrow}Y(   a, b)\times H(b,d)\times X(d,c)\}.$$\end{Lemma}

 \begin{proof}By Lemma~\ref{cellbasis}, it is enough to prove that $A^\uparrow$ is the $\Bbbk$-space  spanned by $S$,  where  $S=\dot{\bigcup}_{a,c\in J}S_{a,c}$.
For any $yhx\in S$, $yhx=y1_{ \ob a}hx$ for some $ \ob a\in I^\uparrow$ and  hence the subspace spanned by $S$ is contained in
$A^\uparrow$.
On the other hand,
any element in $A^\uparrow$ is a linear combination of elements of $A1_{\ob b}A$ for
some $ \ob b\in I^\uparrow$. By Lemma~\ref{linearcom}(2), any element in  $A1_{ \ob b} A$
is a linear combination of some  $y_1h_1x_1$  in $ A1_\ob c A$
with $\ob c\succeq \ob b$, where
 $(y_1, h_1, x_1)\in Y( c)\times  H( c,d)\times X( d)$ such that $ c,d\in\ob c$.  Since $I^\uparrow$ is an upper set, $\ob c\in  I^\uparrow$. So,   $A^\uparrow$ is contained in the $\Bbbk$-space spanned by $S$.  \end{proof}
For any $x\in A$, let  $\bar x$ be its image in any quotient algebra of $A$.
The following definition is motivated by \cite{BS}.
\begin{Defn}\label{bara1}For any $\ob a\in I$,  define  $A_{\preceq \ob a}= A/A^{\not\preceq \ob a}$, $A_{\prec \ob a}= A/A^{\not\prec \ob a}$ and  $\bar A_{\ob a}=\bar 1_{ \ob a } A_{\preceq \ob a}\bar 1_{ \ob a}$.  \end{Defn}

\begin{Prop}\label{bara} Suppose $\ob a\in I$. We have
\begin{itemize}
\item [(1)]  $\{\bar h\mid h\in H( \ob a)\}$ is a $\Bbbk$-basis of  $\bar A_{ \ob a}$.
\item[(2)] Both $ A_{\preceq  \ob a}$  and $ A_{\prec   \ob a}$ admit    upper finite weakly triangular decompositions.

    \item[(3)] $A_{\prec \ob a}\cong A_{\preceq\ob a}/Q$, where $Q$ is the two-sided ideal of $A_{\preceq\ob a}$ generated by $\bar 1_\ob a$.
   \item [(4)] $\pi_{\ob a}  (A_{\ob a})=\bar A_{\ob a}$ where $ \pi_{\ob a}: A\twoheadrightarrow  A_{\preceq \ob a}$ is the
 canonical epimorphism.
  \end{itemize}\end{Prop}
 \begin{proof}By Lemmas~\ref{cellbasis},~\ref{twoside}, we immediately have (1). Moreover, for any $e,c\in J$,
 $ 1_eA_{\diamond  \ob a}1_c$ has basis
 $$\{\bar y\bar h\bar x\mid (y,h,x)\in \bigcup_{b,d\in \ob b, \ob b\diamond \ob a} Y(  e, b)\times H(  b,d)\times X(d,c)\},$$
 where $\diamond\in \{\preceq,\prec\}$. So, (2) follows.
  (3)  follows immediately from  (2) and Lemma~\ref{twoside}. (4) follows from (1),  Lemma~\ref{cellbasis} and the basis of $ A_{\preceq  \ob a}$ given above.
   \end{proof}

\section{Upper finite fully stratified categories}
For any locally unital  $\Bbbk$-algebra $A=\bigoplus_{a,b\in J}1_a A1_b$,
 a left $A$-module means a module $M$ such that $M=\bigoplus_{a\in J}1_a M$. It is locally finite dimensional if $\dim 1_aM<\infty$ for all $a\in J$. Let $A$-mod ( resp., $A\text{-lfdmod}$, resp., $A\text{-pmod}$) be the category of all (resp., locally finite dimensional, resp., finitely generated projective) left  $A$-modules.
When $A$ is finite dimensional,    $A\text{-fdmod}$ is  the category of all finite dimensional left $A$-modules. If  $A$ is locally finite dimensional, then any finitely generated module is locally finite dimensional (e.g. \cite[\S~2.2]{BRUNDAN}).

Throughout this section, $A$   always admits  an upper finite weakly triangular decomposition    in the sense of Definition~\ref{WT}.
In particular, $A$ is locally finite dimensional and hence   $A\text{-pmod}\subset A\text{-lfdmod}$.
\subsection{Standardization and costandardization functors} Recall $A_{\preceq\ob  a}$ and $A_{\prec\ob  a}$ in Definition~\ref{bara1} and $A^{\not\preceq \ob a}$, $A^{\not\prec \ob a}$ in Definition~\ref{ideals}.
Obviously, $A^{\not\preceq \ob a}$ and   $A^{\not\prec \ob a}$  are   idempotent ideals. So,
$ A_{\preceq\ob  a}\text{-lfdmod}$ and $ A_{\prec\ob  a}\text{-lfdmod}$ are  Serre subcategories  of $A\text{-lfdmod}$.
Since  $\bar A_\ob a=\bar 1_\ob a A_{\preceq\ob  a}\bar 1_\ob a$ and $\bar 1_\ob a$ is an idempotent, by   \cite[Theorem ~II.4.3]{Re},
the exact idempotent truncation functor
$$  j^{ \ob a}: A_{\preceq  \ob a}\text{-lfdmod}\rightarrow   \bar A_{\ob a} \text{-fdmod}
$$ satisfying  $j^{ \ob a} V  = \bar 1_{\ob a} V$
 for any $V\in A_{\preceq \ob a}\text{-lfdmod}$,
is the quotient functor by the Serre subcategory $\mathcal C$ of $ A_{\preceq\ob  a}\text{-lfdmod}$ consisting of modules $V$ with $\bar 1 _\ob a V=0$.
So $\mathcal C= B$-lfdmod, where $B=A_{\preceq  \ob a}/Q$ and  $Q$ is the two-sided ideal of $A_{\preceq\ob a}$ generated by $\bar 1_\ob a$.
By Proposition~\ref{bara}(3), $B\cong A_{\prec \ob a}$ and hence  $\mathcal C=  A_{\prec\ob  a}\text{-lfdmod}$, i.e.,
   $\bar A_{ \ob a}\text{-fdmod}$ is the Serre quotient category $ A_{\preceq  \ob a}\text{-lfdmod}/A_{\prec\ob  a }\text{-lfdmod}$.
 It's known that $j^{ \ob a}$ has the left (resp., right) adjoint   $j^{ \ob a}_!$ (resp., $ j^{ \ob a}_*$), where
       \begin{equation}\label{adj123} j^{ \ob a}_!:= A_{\preceq \ob a}\bar 1_{ \ob a}\otimes_{\bar A_{\ob a}} ?,~~ \text{  $ j^{ \ob a}_*:=\bigoplus_{ \ob b\in I}\Hom_{\bar A_{\ob a}}(\bar 1_{ \ob a} A_{\preceq \ob a}\bar 1_{\ob b},?)$}.\end{equation}
    They give two functors  $\Delta, \nabla:  \bigoplus_{\ob a \in I} \bar A_{\ob a}\text{-fdmod}\rightarrow  A\text{-lfdmod}$ such that
\begin{equation}\label{exa}   \Delta=\bigoplus_{\ob a\in I}j^{\ob a}_!, ~~ \text{   $ \nabla=\bigoplus_{\ob a\in I}j^{\ob a}_*$.}\end{equation}
 Following \cite{LW}, $j^{\ob a}_!$'s  and $ j^{ \ob a}_*$'s are called the
   \textsf{standardization functors and costandardization functors}, respectively.

   Let $\{L_{ \ob a}(\lambda)\mid \lambda\in \bar \Lambda_{ \ob a}\}$ be the complete    set of pairwise  inequivalent irreducible  $\bar A_{ \ob a}$-modules. For  each $\lambda\in \bar\Lambda_{\ob a}$,  let $P_{ \ob a}(\lambda)$ (resp., $I_{\ob a} (\lambda)$) be the  projective cover (resp., injective hull) of $L_{ \ob a} (\lambda)$.

\begin{Defn}\label{stanpro}
For any $\lambda\in \bar\Lambda_\ob a$,
define $\Delta(\lambda)=\Delta(P_{\ob a} (\lambda))$, $\bar\Delta(\lambda)= \Delta(L_{\ob a} (\lambda))$, $\nabla(\lambda)=\nabla(I_{\ob a}(\lambda))$ and $\bar \nabla(\lambda)= \nabla(L_{\ob a}(\lambda))$, and call them the standard, proper standard, costandard and proper costandard  modules, respectively. In other words,
$$\begin{aligned}
\Delta(\lambda)&= j^\ob a_!(P_\ob a(\lambda)), \quad\quad \bar\Delta(\lambda)= j^\ob a_!(L_\ob a(\lambda)),\\
\nabla(\lambda)&=j^\ob a_*( I_\ob a(\lambda)), \quad \quad \ \bar\nabla(\lambda)=j^\ob a_*(L_\ob a(\lambda)).
\end{aligned}$$ \end{Defn}

\begin{Lemma}\label{exactj}  Suppose $\ob a\in I$.
\begin{itemize}
\item [(1)] $A_{\preceq \ob a}\bar 1_{\ob a}$ (resp., $\bar 1_{\ob a}A_{\preceq \ob a}$) is a  projective  right (resp., left)  $\bar A_{\ob a}$-module.
Furthermore, it is a free right (resp., left)  $\bar A_{\ob a}$-module if $ \ob a $ contains a unique element in $J$.
\item [(2)] Both $j^\ob a_!$ and $ j^\ob a_*$ are exact and so are  $\Delta$ and $\nabla$.
\item[(3)] For any $b\in \ob a$, $A_{\preceq \ob a}\bar 1_b\cong j^\ob a_!( \bar A_\ob a\bar 1_b)$ as $A$-modules.
\end{itemize}
\end{Lemma}
\begin{proof}By Proposition~\ref{bara}(1), $\bar A_\ob a$ has basis $\{\bar h\mid h\in H(b,c), b,c\in \ob a\}$.
 Consider the right $\bar A_{\ob a}$-homomorphism
$$  \phi:
\bigoplus_{b\in \ob a, y\in Y(b)} \bar 1_b \bar A_\ob a \rightarrow A_{\preceq \ob a}\bar 1_{\ob a}
$$ such that the image of  $\bar 1_b\bar  h$ in  the $y$th copy of  $\bar 1_b \bar A_\ob a$ is $\bar y \bar h \in A_{\preceq \ob a}\bar 1_{\ob a}$ for all $h\in H(b,c)$. By Proposition~\ref{bara}(1)--(2), $\phi$ sends a basis of $\bigoplus_{b\in \ob a, y\in Y(b)} \bar 1_b \bar A_\ob a$ to a basis of $A_{\preceq \ob a}\bar 1_{\ob a}$. So, $\phi$ is a linear   isomorphism.
Now $A_{\preceq \ob a}\bar 1_{\ob a}$ is projective since   $\bar 1_b \bar A_\ob a$ is a projective right $\bar A_{\ob a}$-module. Similarly, $\bar 1_\ob a  A_{\preceq \ob a}$ is a projective left $\bar A_\ob a $-module.  If $ \ob a$ contains only one element $b$, then $1_\ob a=1_b$. So, $\bar 1_b\bar A_\ob a=\bar A_\ob a$, proving  the last statement in (1).
  By (1),   $\bar 1_{\ob a}  A_{\preceq \ob a}\bar 1_{\ob b}$ is a  finitely generated and projective  left $\bar A_{\ob a}$-module. So,  (2) follows.

  Consider the  isomorphism of $A$-modules
  $$\psi: A_{\preceq \ob a}\bar 1_\ob a\otimes _{\bar A_\ob a} \bar A_\ob a \rightarrow  A_{\preceq \ob a}\bar 1_\ob a, \quad \bar y\otimes \bar h \mapsto \bar y\bar h$$
   for any $h\in H(c,b)$, $y\in Y(c)$,    $b,c\in\ob a$. Recall that $ 1_\ob a=\sum_{b\in \ob a}1_ b$. We have $ A_{\preceq \ob a}\bar 1_\ob a=\bigoplus _{b\in \ob a}A_{\preceq \ob a}\bar 1_b$ and
  $$A_{\preceq \ob a}\bar 1_\ob a\otimes _{\bar A_\ob a} \bar A_\ob a=\bigoplus_{b\in\ob a}A_{\preceq \ob a}\bar 1_\ob a\otimes _{\bar A_\ob a} \bar A_\ob a1_b=\bigoplus_{b\in\ob a}  j^\ob a_!( \bar A_\ob a\bar 1_b)
  .$$
  Moreover, the restriction of $\psi$ to  $j^\ob a_!( \bar A_\ob a\bar 1_b)$ gives a homomorphism of $A$-modules $\psi_b$ from $j^\ob a_!( \bar A_\ob a\bar 1_b)$ to $A_{\preceq \ob a}\bar 1_b$.  Furthermore, $\psi_b$ is an isomorphism since it sends generating elements of
  $j^\ob a_!( \bar A_\ob a\bar 1_b)$ to the basis $\{\bar y\bar h\mid  h\in H(c,b), y\in Y(c), c\in\ob a\}$ of $A_{\preceq \ob a}\bar 1_b$.
\end{proof}

\begin{Defn}\label{rho}  Let  $\rho : \bar \Lambda\rightarrow I$ be the function
 such that $ \rho(\lambda)=\ob a$  for any  $\lambda\in \bar \Lambda_{\ob a}$,  where $\bar \Lambda=\bigcup_{\ob a\in I} \bar \Lambda_{\ob a}$. Following \cite{BS}, for  any given sign function $\varepsilon: I \rightarrow \{\pm\}$, define
 \begin{equation}\label{sgn}\Delta_\varepsilon(\lambda)=\left\{
                       \begin{array}{ll}
                         \Delta(\lambda), & \hbox{if $\varepsilon(\rho(\lambda))=+$,} \\
                         \bar\Delta(\lambda), & \hbox{if $\varepsilon(\rho(\lambda))=-$.}
                       \end{array}
                     \right.
 \end{equation}
\end{Defn}

An $A$-module $V$  has a finite $\Delta$-flag if it has a finite filtration  such that its  sections are  isomorphic to $\Delta (\lambda)$ for various  $\lambda\in \bar \Lambda$. Similarly we have the notion of finite $\Delta_\varepsilon$-flag.
The following result is motivated by  \cite[Theorem~5.14]{BS}.
\begin{Theorem}\label{striateddndn}
 Suppose  $\varepsilon: I \rightarrow \{\pm\}$ is a sign function.
\begin{itemize}
\item [(1)] For any   $\lambda\in\bar  \Lambda$, there is a projective $A$-module $P$ such that $P $ has a finite $\Delta_\varepsilon$-flag with top section $\Delta_\varepsilon(\lambda)$,  and  other sections are of forms $\Delta_\varepsilon(\mu)$ for  $\mu\in\bar \Lambda$ with  $\rho(\mu)\succeq \rho (\lambda)$.
\item [(2)] For any   $\lambda\in\bar  \Lambda$, $\Delta(\lambda)$
has a unique irreducible quotient denoted by $L(\lambda)$.
\item  [(3)]  $\{L(\mu)\mid \mu\in \bar \Lambda\}$ is a complete set of pairwise inequivalent
irreducible $A$-modules.
\end{itemize}
\end{Theorem}
\begin{proof}Suppose $\lambda\in \bar \Lambda_{\ob a}$ for some $\ob a\in I$.
We claim that the projective $A$-module  $A 1_{\ob a}$ is the required $P$ in (1).
In fact, since  $I$ is  upper finite, there are $\ob b_i\in I$, $1\le i\le n$ such that
 $ \ob a=\ob b_1$, $X(  b_i,  a) \neq \emptyset$ for some $b_i\in \ob b_i$ and $a\in \ob a$ and  $ \ob b_i\preceq  \ob b_j\Rightarrow i\leq j$.
Let $P_i$ be the subspace of $A 1_{\ob a}$  spanned by
$$\{yhx\mid (y,h,x)\in\bigcup_{b_j, c_j\in \ob b_j, a\in \ob a}  Y( b_j)\times  H( b_j,c_j)\times X( c_j,  a), i+1\leq j\leq n\}. $$
Thanks to  Lemma~\ref{linearcom}(2),
$A 1_{\ob a}=P_0\supseteq P_1\supseteq \cdots \supseteq P_n=0$
is a filtration of $A$-modules such that  $P_{i-1}/P_{i}$
has  basis given by $\{yhx+P_i \mid (y, h, x)\in Y(  b_i)\times  H(   b_i,c_i)\times  X( c_i,
d), b_i,c_i\in \ob b_i, d\in \ob a\}$.
Define the $\Bbbk$-linear isomorphism
\begin{equation} \label{kkk1} \varphi: \bigoplus_{ d\in\ob a, c_i\in \ob b_i}\bigoplus_{x\in X( c_i,d)}A_{\preceq \ob b_i}\bar 1_{c_i}\rightarrow P_{i-1}/P_i\end{equation}
sending the basis vector $\bar y\bar h\bar 1_{c_i}$  in the  $x$th copy of $A_{\preceq \ob b_i}\bar 1_{ c_i}$ to
$yhx+P_i$, $(y, h)\in Y(  b_i)\times H(   b_i,c_i)$ for any $b_i\in \ob b_i$. Then
$\varphi$ is  an $A$-homomorphism. To see this, take
 $(y, h)\in Y(  c,    b_i)\times  H(   b_i,c_i)$ and any element  $u\in 1_{  d} A 1_{   c}$ for any $  c,   d\in J$.
 By Lemma~ \ref{linearcom}(2), we  have
  $$ uyh= \sum c_{p,q}y_ph_q+ z$$
for  some  $c_{p,q}\in \Bbbk$, $y_p\in Y(  d,   b'_i), h_q\in H(   b'_i,c_i)$, $b'_i\in\ob b_i$ and  $z \in A1_\ob e A$ with $\ob  e\succ \ob  b_i$. Since
 $z$ acts  on $x+ P_{i}$ as zero,
$u\varphi(yh\bar 1_{c_i})= \sum c_{p,q}y_ph_qx+P_i $.
 On the other hand, any $ \bar z$ is zero in $A_{\preceq \ob b_i}$. So,
 $$\varphi (uyh\bar 1_{c_i})=\varphi( \sum c_{p,q}y_ph_q) =\sum c_{p,q}y_ph_qx+P_i=u\varphi(yh\bar 1_{c _i}).$$
 This proves that $\varphi$ is an $A$-isomorphism.
 Thanks to Lemma~\ref{exactj}(3),
  $$ A_{\preceq \ob b_i}\bar 1_{c_i} \cong
 j^{\ob b_i}_!(\bar A_{\ob b_i}\bar 1_{c_i})\cong  \bigoplus _{\mu\in\bar \Lambda_{\ob b_i}}\Delta(\mu)^{\oplus n(\mu)}
 $$
 if $ \bar A_{\ob b_i}\bar 1_{c_i}\cong \bigoplus _{\mu\in\bar \Lambda_{\ob b_i}}P_{\ob b_i}(\mu)^{\oplus n(\mu)}$.
 So, $A 1_{\ob a}$ has a finite $\Delta$-flag such that each section is  of form $\Delta(\mu)$ with $\rho(\mu)\succeq \ob a$.
 If $i=1$, then $\ob b_1=\ob a$. In this case for $c_1,d\in \ob a$, $X(c_1,d)=\{1_d\}$ if $c_1=d$ and $\emptyset$ otherwise. So,
  $$\bigoplus_{ d\in\ob a, c_1\in \ob a}\bigoplus_{x\in X( c_1,d)}A_{\preceq \ob a}\bar 1_{c_1}=\bigoplus_{ d\in\ob a}A_{\preceq \ob a}\bar 1_{d}=A_{\preceq \ob a}\bar1_\ob a  \cong A_{\preceq \ob a}\bar 1_{\ob a}\otimes _{\bar A_{\ob a}}\bar A_{\ob a}=
 j^{\ob a}_!(\bar A_{\ob a})\cong  \bigoplus _{\mu\in\bar \Lambda_{\ob a}}\Delta(\mu)^{\oplus\text{dim}L_{\ob a}(\mu)}. $$
 So,  we can choose $\Delta(\lambda)$ as the top section and $A 1_{\ob a}$ is the required $P$ in (1)
 if  $\varepsilon(\rho(\lambda))=+$ for any $\lambda\in \bar \Lambda$.
 By Lemma~\ref{exactj}(2),  $\Delta$   is exact. So,  for any $\mu\in \bar \Lambda_{\ob b}$,  $\Delta(\mu)$     has a finite $\bar\Delta$-flag such that  each section is of form $\bar\Delta(\nu)$ for some   $\nu\in\bar \Lambda_{\ob b}$.  Therefore,  $A 1_{\ob a}$ is still the required $P$ in (1)  for any $\epsilon$.

Using the adjoint triple in \eqref{adj123} and  \cite[Lemma~2.22]{BS},
we have that  $\Delta(\lambda)$ is an indecomposable projective module in $A_{\preceq \ob a}\text{-lfdmod}$ with the simple head $L(\lambda)$. This proves (2). Moreover, $L(\lambda)$
 is the unique (up to isomorphism) irreducible $A_{\preceq \ob a}$-module such that \begin{equation}\label{kkk1234} \bar 1_{\ob a} L(\lambda)\cong L_{\ob a}(\lambda).\end{equation} So, $\{L(\lambda)\mid \lambda\in\bar\Lambda\}$'s are  pairwise inequivalent. Suppose that $L$ is any irreducible $A$-module. Then there exists $\ob a\in I$ such that $1_{\ob a}L\neq0$, and hence  $L$ is a quotient of $A 1_{\ob a} $. By (1) together with arguments on induction for the length of the $\Delta$-flag of $A 1_{\ob a}$, there is a  $\mu\in \bar \Lambda$ with $\rho(\mu)\succeq \ob a$ such that $\Hom_A(\Delta(\mu),L)\neq0$, forcing $L=L(\mu)$,  the  simple head of $\Delta(\mu)$.  \end{proof}

\subsection{Stratified categories and highest weight categories}
 Let $A$-mod$^{\Delta}$ be the full subcategory of $A$-mod
consisting of all modules with a finite $\Delta$-flag. Since $\Delta(\lambda)\in A\text{-lfdmod}$ for any $\lambda\in\bar \Lambda$, $A$-mod$^{\Delta}$ is a subcategory
of  $A\text{-lfdmod}$.

\begin{Theorem}\label{ext1}The category $A\text{-lfdmod}$ is an upper finite  fully stratified category in the sense of \cite[Definition~3.36]{BS}. The   corresponding stratification    $\rho$ is in Definition~\ref{rho}, and  the standard, proper standard, costandard and proper costandard objects are  defined in Definition~\ref{stanpro}.
\end{Theorem}
\begin{proof}
By Theorem ~\ref{striateddndn}(3), the function $\rho$ in Definition~\ref{stanpro} is a stratification of $A\text{-lfdmod}$ in the sense of \cite[Definition~3.1]{BS}.
For any $\ob a\in I$,  using   Theorem ~\ref{striateddndn}(3) again for the upper finite weakly triangular structures of $A_{\preceq\ob a}$ and $A_{\prec\ob a}$ (see Proposition~\ref{bara}(2)), we see that  the Serre subcategory $A\text{-lfdmod}_{\preceq \ob a}$ (resp., $A\text{-lfdmod}_{\prec \ob a}$) generated by $\{L(\lambda)\mid \rho(\lambda)\preceq \ob a\}$ (resp., $\{L(\lambda)\mid \rho(\lambda)\prec \ob a\}$) is  $A_{\preceq\ob a}$-lfdmod (resp., $A_{\prec\ob a}$-lfdmod). So, the adjoint triple   $(j^{\ob a}_!,j^{\ob a} , j^{\ob a}_*)$ agree with the  adjoint triple between $A\text{-lfdmod}_{\preceq \ob a}$ and  its Serre quotient $A\text{-lfdmod}_{\preceq \ob a}/A\text{-lfdmod}_{\prec \ob a} $. Hence,    the standard, proper standard, costandard and proper costandard objects defined in Definition~\ref{stanpro} are the required objects.
Thanks to Theorem ~\ref{striateddndn}(1), $A\text{-lfdmod}$ satisfies the property  ($\widehat{P\Delta_\varepsilon}$) in \cite{BS} for any $\varepsilon$ as follows:

($\widehat{P\Delta_\varepsilon}$):  For each $\lambda\in \bar \Lambda$, there exists a projective object $P_\lambda$ admitting a finite  $\Delta_\varepsilon$-flag
with $\Delta_\varepsilon(\lambda)$ at the top and other sections $\Delta_\varepsilon(\mu)$ for $\mu\in\bar\Lambda$ with $\rho(\mu)\succeq \rho(\lambda)$.

In our case,  $P_\lambda$  is  $A 1_{\ob a}$ if $\rho(\lambda)=\ob a$. By \cite[Definition~3.36]{BS}, $A\text{-lfdmod}$ is an upper finite fully stratified category.
\end{proof}

 \begin{rem} Brundan and Stroppel introduced  the notion of upper finite based stratified algebra $B$~\cite[Definition~5.17]{BS} and proved that $B$-lfdmod is an upper finite fully stratified category~\cite[Theorem~5.22]{BS}. They also proved  that any locally unital algebra   which admits an  upper finite Cartan decomposition is an upper finite based stratified algebra~\cite[Theorem~5.30]{BS}. 
   In this paper, we give another sufficient condition for $A$-lfdmod being  an upper finite fully stratified category for a locally unital algebra. By \cite[Definition~5.17]{BS}, any upper finite based stratified algebra is Morita equivalent to $A$ which admits an upper finite  weakly triangular decomposition such that
 $\bar A_\ob a$ is a basic algebra and the set of  primitive idempotents  is $\{\bar 1_ b\mid b\in\ob a\}$ for all $\ob a\in I$ (i.e., $\bar \Lambda_\ob a=\ob a$).
\end{rem}

 The following result follows from Proposition~\ref{COBWA} and Theorem~\ref{ext1}, immediately. This is the first step to give tensor product categorifications  of integrable lowest weight  and integrable highest weight representations for $\mathfrak {sl}_\infty$ or $\hat{\mathfrak {sl}}_p$ introduced by Webster in \cite{W}. Details will be given in \cite{GRS1}.
 \begin{Theorem}\label{COBMW1} Let $A$ be the $\Bbbk$-algebra  associated to the cyclotomic oriented Brauer category $\OB(\mathbf u)$ in  Proposition~\ref{COBWA}. Then
  $A$-lfdmod is an upper finite fully  stratified category in the sense of \cite[Definition~3.36]{BS} with respect to the stratification $\rho: \bar\Lambda\rightarrow I$ in Definition~\ref{rho}, where $I$ is given in \eqref{ocbap}.\end{Theorem}
Suppose  $L$ is a simple $A$-module and $V\in A$-mod. Define $$[V:L]=\text{sup}|\{i\mid V_{i+1}/V_i\cong L\}|$$ the supremum being taken over all filtrations by submodules $0=V_0\subset \cdots \subset V_n=V$.
The number $[V:L]$ is called the composition multiplicity  of  $L$ in $V$ although  $V$ may not have a composition series.
Using the adjoint triple  $(j^{\ob a}_!,j^{\ob a} , j^{\ob a}_*)$  yields the following result (see also \cite[Lemma~3.3]{BS}).

\begin{Lemma}\label{barcom}
For any $\lambda, \mu\in \bar \Lambda$, and $\lambda\neq \mu$,
\begin{enumerate}
\item[(1)] $[\bar\Delta(\mu):L(\mu)] =1$ and  $[\bar\Delta(\mu):L(\lambda)] = 0$  unless   $\rho(\mu)\succ \rho(\lambda) $,
\item [(2)]$[\bar\nabla(\mu):L(\mu)] =1$ and $[\bar\nabla(\mu):L(\lambda)] =0$ unless $\rho(\mu)\succ \rho(\lambda)$.
\end{enumerate}
\end{Lemma}

 For any $V\in A$-{\rm mod}$^{\Delta}$, let $(V:\Delta(\lambda))$  be    the   multiplicity
of $\Delta(\lambda)$ in a $\Delta$-flag of $V$.  By  \cite[Lemma 3.49]{BS} and Theorem~\ref{ext1}, $\dim\text{Ext}^i_{A}(\Delta(\nu), \bar\nabla(\mu))=\delta_{i,0}\delta_{\nu,\mu},~i\geq0, \nu,\mu\in\bar \Lambda$. So,  \begin{equation}\label{abc1}(V:\Delta(\lambda))=\dim \Hom_A(V, \bar\nabla(\lambda)), \end{equation}
which  is independent of a flag.

\begin{Prop}\label{dektss1} For any $\lambda\in\bar \Lambda$, let $P(\lambda)$ be the projective cover of $L(\lambda)$.
We have:
 \begin{itemize} \item[(1)] $P(\lambda)\in A$-{\rm mod}$^{\Delta}$  with top section $\Delta(\lambda)$.
 \item [(2)] $(P(\lambda): \Delta(\mu))=[\bar\nabla(\mu): L(\lambda)]$, which is  non-zero   for  $\mu\neq \lambda$  only if  $\rho(\mu)\succ\rho(\lambda)$.
 In particular, $(P(\lambda):\Delta(\lambda))=1$.
  \end{itemize} \end{Prop}
  \begin{proof} Since $P(\lambda)$ is projective,  $\text{Ext}^1_A(P(\lambda),\bar\nabla(\mu))=0$ for all $\mu\in\bar\Lambda$. By \cite[Theorem~3.39]{BS} and  Theorem ~\ref{ext1}, for any finitely generated $A$-module $M$,  $M \in A$-{\rm mod}$^{\Delta}$ if  and only if $\text{Ext}^1_A(M,\bar\nabla(\mu))=0$ for all $\mu\in\bar\Lambda$. So, $P(\lambda) \in A$-{\rm mod}$^{\Delta}$.
 Since $P(\lambda)$ is the projective cover of $L(\lambda)$, by Theorem ~\ref{striateddndn}(2), $\Delta(\lambda)$ must appear as the top section. This proves (1). Finally, (2) follows immediately from \eqref{abc1} and Lemma~\ref{barcom}(2).
   \end{proof}

For any $\ob a\in I$, let $\bar A_\ob a=\bigoplus_{\omega\in \mathbf B_\ob a} \bar A_{\ob a,\omega}$ be the block decomposition of $\bar A_\ob a$.
Define $\mathbf B=\bigcup_{\ob a\in I}\mathbf B_\ob a$. The partial order on $I$ induces a partial order on $\mathbf B$
such that $\omega\prec \omega'$ if $ \omega\in\mathbf B_\ob a$, $\omega'\in \mathbf B_\ob b$ and $\ob a\prec \ob b$.
Define  $\tilde \rho:\bar \Lambda\rightarrow \mathbf B$ such that for $\lambda\in\bar\Lambda_\ob a$, $\tilde \rho(\lambda)=\omega$ if
$L_{\ob a}(\lambda)\in \bar A_{\ob a,\omega}$-fdmod. It is clear that $\tilde \rho $ is a new stratification  of $A\text{-lfdmod}$ in the sense of \cite[Definition~3.1]{BS}.

 For any $\omega\in \mathbf B_\ob a$, let $A_{\preceq \ob a}\text{-lfdmod}_{\preceq\omega}$ be the Serre subcategory of $A_{\preceq \ob a}\text{-lfdmod}$
with irreducible objects $\{L(\lambda)\mid \tilde\rho(\lambda)\preceq \omega\}$. Similarly we have $A_{\preceq \ob a}\text{-lfdmod}_{\prec\omega}$.
Note that $A_{\preceq \ob a}\text{-lfdmod}_{\prec\omega}=A_{\prec \ob a}\text{-lfdmod}$ by Theorem ~\ref{striateddndn}(3).
Let $j^{\ob a,\omega}$ be the restriction of $j^\ob a$ to $A_{\preceq \ob a}\text{-lfdmod}_{\preceq\omega}$.
Since  $j^{\ob a,\omega}(M)\in \bar A_{\ob a,\omega}$-fdmod for any $M\in A_{\preceq \ob a}\text{-lfdmod}_{\preceq\omega}$, $j^{\ob a,\omega}$ is actually a functor from  $A_{\preceq \ob a}\text{-lfdmod}_{\preceq\omega}$ to $\bar A_{\ob a,\omega}$-fdmod.
Moreover,  $j^{\ob a,\omega}$ induces an equivalence of categories between $A_{\preceq \ob a}\text{-lfdmod}_{\preceq\omega}/A_{\preceq \ob a}\text{-lfdmod}_{\prec\omega}$ and $\bar A_{\ob a,\omega}$-fdmod.

Let $j^{\ob a,\omega}_!$ and $j^{\ob a,\omega}_*$ be the restriction of  $j^{\ob a }_!$ and $j^{\ob a }_*$ to $\bar A_{\ob a,\omega}$-fdmod, respectively.
 Then $$\Delta=\bigoplus_{\omega\in \mathbf B_\ob a,\ob a\in I}j^{\ob a,\omega}_!, \quad \nabla=\bigoplus_{\omega\in \mathbf B_\ob a,\ob a\in I}j^{\ob a,\omega}_*.$$
Moreover, $j^{\ob a,\omega}_!$ and $j^{\ob a,\omega}_*$ are actually functors from $\bar A_{\ob a,\omega}$-fdmod to $A_{\preceq \ob a}\text{-lfdmod}_{\preceq\omega}$.
In fact, for any $M\in \bar A_{\ob a,\omega}$-fdmod, $j^{\ob a,\omega}_!(M)$ has a $\bar\Delta$-flag, and  $\bar\Delta(\mu)$ appears as a section if  $[M:L_\ob a(\mu)]\neq 0$. In this case,  $\tilde\rho(\mu)=\omega$. By Lemma~\ref{barcom} and Lemma~\ref{exactj}(2), we see that $j^{\ob a,\omega}_!(M)\in A_{\preceq \ob a}\text{-lfdmod}_{\preceq\omega}$.
The result on  $j^{\ob a,\omega}_*$ follows from similar arguments.
Furthermore, since  $(j^{\ob a}_!,j^{\ob a} , j^{\ob a}_*)$ is an  adjoint triple, so is
  $(j^{\ob a,\omega}_!,j^{\ob a,\omega} , j^{\ob a,\omega}_*)$.

\begin{Theorem}\label{newstratification}
The category $A\text{-lfdmod}$ is an upper finite  fully stratified category in the sense of \cite[Definition~3.36]{BS} with respect to the stratification    $\tilde\rho$.
In particular,   $A\text{-lfdmod}$ is an upper finite  highest weight  category in the sense of \cite[Definition~3.36]{BS} if $ \bar A_\ob a$ is semisimple for all $\ob a\in I$.
\end{Theorem}
\begin{proof}
By Proposition~\ref{dektss1}, for any $\lambda$ the projective module $P(\lambda)$  is the required $P_\lambda$ in  ($\widehat{P\Delta_\varepsilon}$) for $\tilde \rho$ and any sign function $ \varepsilon: \mathbf B\rightarrow \{\pm\}$. So, the first statement holds.
If $ \bar A_\ob a$ is semisimple for all $\ob a\in I$, then the stratum $ \bar A_{\ob a,\omega}$-fdmod is simple for any $\omega$ and hence $A\text{-lfdmod}$ is an upper finite highest weight category.
\end{proof}

\subsection{Duality} In this subsection, we assume that there is an anti-involution $\sigma_A: A\rightarrow A$  such that
\begin{equation} \label{ssgm}\sigma_A(1_{  a})=1_{  a}\end{equation}  for any $  a\in J$.
Then $\sigma_A$ stabilizes both $A_{\ob a}$ and $A^{ \npreceq\ob a}$, and   results in an  anti-involution $\sigma_{A_{\preceq \ob a}}$ on $
A_{\preceq\ob  a}$. Restricting $\sigma_{A_{\preceq \ob a}}$ to $\bar A_{\ob a}$ yields an anti-involution
 $\sigma_{\bar A_{\ob a}}$ on $ \bar A_{\ob a} $.
 For any $V\in  A\text{-lfdmod}$, let
 \begin{equation}
 V^\circledast=\bigoplus_{a\in J}\Hom_{\Bbbk}(1_{a}V,\Bbbk).
 \end{equation}
 Then  $V^\circledast$ is an $A$-module such that for any $x\in  A$ and $f\in V^\circledast$,
$(xf)(v)=f(\sigma_A(x)v)$.
This induces an exact contravariant duality functor $\circledast$  on $A\text{-lfdmod}$.
Similarly, we have the contravariant duality functor on  $\bigoplus_{\ob a\in I} \bar A_{\ob a}\text{-fdmod}$.

 \begin{Lemma}\label{key11} As functors from $\bigoplus_{\ob a\in I} \bar A_{\ob a}\text{-fdmod}$ to $A\text{-lfdmod}$,
 $\circledast\circ \Delta\cong \nabla\circ \circledast$.
\end{Lemma}
\begin{proof} The required natural isomorphism is  $\phi:  \circledast\circ \Delta\rightarrow  \nabla\circ \circledast$ such   that the  $A$-isomorphism
$\phi_V: \Delta(V)^\circledast\rightarrow \nabla(V^\circledast)$   sends  $
\alpha$ to $ \bar \alpha$ for any finite dimensional $\bar A_{\ob a}$-module $V$,
where $$\bar\alpha(f)(v)=\alpha(\sigma_{A_{\preceq\ob  a}}(f)\otimes v)$$ for all $v\in V$ and  $f\in \bar 1_{\ob a}A_{\preceq \ob a} \bar 1_{\ob b}$.
The  inverse of $\phi_V$ is the homomorphism
$ \nabla(V^\circledast)\rightarrow\Delta(V)^\circledast$, sending $ \beta$ to $ \tilde \beta$,
where $\tilde \beta(f\otimes v)=\beta(\sigma_{A_{\preceq \ob a}}(f))(v)$.
\end{proof}
\begin{Assumption}\label{pdual}   For all  $\lambda\in \bar \Lambda$,  $M^\circledast\cong M$  as $\bar A_{\rho(\lambda)}$-modules, where $ M\in\{ P_{\rho(\lambda) }(\lambda), L_{\rho(\lambda)}(\lambda)\}$.
\end{Assumption}
\begin{Lemma}\label{isodual1}Keep the Assumption~\ref{pdual}. Suppose $\lambda, \mu\in \bar \Lambda$. We have
    \begin{itemize} \item [(1)]  $\Delta(\lambda)^\circledast\cong \nabla(\lambda)$, $\bar\Delta(\lambda)^\circledast\cong \bar \nabla(\lambda)$ and $L(\lambda)^\circledast\cong L(\lambda)$. \item [(2)]
$[\Delta(\lambda):L(\mu)]=[\nabla(\lambda):L(\mu)]$ and $[\bar\Delta(\lambda): L(\mu)]=[\bar \nabla(\lambda) : L(\mu)]$.\end{itemize}
\end{Lemma}
\begin{proof} The first and the second isomorphisms in (1) follow from  Lemma~\ref{key11} and Assumption~\ref{pdual}.
Suppose $\ob a=\rho(\lambda)$. Thanks to \eqref{ssgm}, we have $$1_{\ob a} L(\lambda)^\circledast=(1_{\ob a}L(\lambda))^\circledast=L_{\ob a}(\lambda)^\circledast\cong L_{\ob a}(\lambda)=1_{\ob a} L(\lambda).$$
 Since  $ L(\lambda)$ is the unique irreducible module (up to isomorphism) in $A_{\preceq \ob a}$-lfdmod such that $1_{\ob a} L(\lambda)=L_{
 \ob a}(\lambda)$, and $\circledast$ is an involution, $L(\lambda)^\circledast$ is irreducible, forcing     $L(\lambda)\cong L(\lambda)^\circledast$.
 Finally, (2) follows from (1) since $\circledast$ is exact.\end{proof}

\subsection{Semisimplicity} We are going to give a criterion on the semi-simplicity of $A$ over $\Bbbk$.
 \begin{Lemma}\label{delps}
 Suppose   $\bigoplus_{\ob a\in I}A_\ob a$ is semisimple over $\Bbbk$. If $\lambda, \mu \in \bar\Lambda$ and $\mu\neq \lambda$ then  $(P(\lambda): \Delta(\mu))=0$.
 \end{Lemma}
\begin{proof}If  $(P(\lambda): \Delta(\mu))\neq 0$, by Proposition~\ref{dektss1}, there is a short exact sequence
\begin{equation}\label{sss1234}  0\rightarrow N \rightarrow P(\lambda)\rightarrow \Delta(\lambda)\rightarrow 0\end{equation}
such that $N\in A\text{-mod}^\Delta $ and $(N:\Delta(\mu))\neq0$. Let $\ob a=\rho(\lambda)$. Then $1_\ob aL(\lambda)=\bar 1_\ob a L(\lambda)=L_\ob a(\lambda)\neq 0$.
 Applying the idempotent  $1_\ob a$ on \eqref{sss1234}  yields
\begin{equation}\label{esksjxs} 0\rightarrow 1_\ob aN \rightarrow 1_\ob aP(\lambda)\rightarrow 1_\ob a\Delta(\lambda)\rightarrow 0.
\end{equation}
 By the general result on the exact functor defined by an idempotent (e.g., \cite[Lemma~2.22]{BS}), $1_\ob a P(\lambda)$ is the projective cover of the irreducible $A_{\ob a}$-module $1_\ob aL(\lambda)$.  Since we are assuming that $A_{\ob a}$ is semisimple,   $1_\ob a P(\lambda)=1_\ob a L(\lambda)$, forcing   $1_\ob a  \Delta(\lambda)=1_\ob a P(\lambda)$ and   $1_\ob a N=0$.  In particular, $ 1_\ob a\Delta(\mu)=1_\ob a\bar\Delta(\mu)=0$.  Note that  Assumption~\ref{pdual} holds automatically when  $\bigoplus_{\ob a\in I}A_\ob a$ is semisimple (hence $\bar A_\ob a$ is semisimple for all $\ob a\in I$ by Proposition ~\ref{bara}(4)).
By Lemma~\ref{isodual1}(2) and Proposition~\ref{dektss1}(2),
 $$[1_{\ob a} \bar\Delta(\mu): 1_{\ob a} L(\lambda)]=[\bar\Delta(\mu):L(\lambda)]=[\bar\nabla(\mu):L(\lambda)]=(P(\lambda):\Delta(\mu))\neq0$$
 whenever $1_{\ob a} L(\lambda)\neq 0$. This is a contradiction since $ 1_\ob a\bar\Delta(\mu)=0$.
\end{proof}

   \begin{Theorem}\label{ss} The $\Bbbk$-algebra $A$ is semisimple if and only if $\bigoplus_{\ob a\in I}  A_{\ob a} $ is semisimple.
 \end{Theorem}
 \begin{proof} Let $B=\bigoplus_{\ob a\in I}  A_{\ob a} $.  For any $M\in  A_{\ob a}\text{-mod}$ and any $\ob a\in I$, $$1_{\ob a}  A\otimes _{B} M=1_{\ob a} A1_{\ob a} \otimes _{B} M\cong M$$ as $ A_{\ob a}$-modules. If  $A$ is semisimple, then $A\otimes _{B} M$ is a direct sum of simple $A$-modules, and hence $1_{\ob a}  A\otimes _{B} M$ is a direct sum of simple $A_{\ob a}$-modules. So,  $A_{\ob a}$ is semisimple for any $\ob a\in I$, forcing $B$ to be semisimple.

Conversely, suppose $B$ is semisimple.
 Thanks to Proposition~\ref{dektss1} and Lemma~\ref{delps}, $P(\lambda)=\Delta(\lambda)$ and  the exact functor $\Delta$ sends  projective $(\bigoplus_{\ob a\in I} \bar A_{\ob a})$ -modules $P_{\rho(\lambda)}(\lambda)$'s  to  projective $A$-modules $P(\lambda)$'s. Since
$\{P_{\ob a}(\lambda)\mid \lambda\in \bar\Lambda_{\ob a},\ob a\in I\}$ gives a complete set of representatives of all indecomposable objects in  $(\bigoplus_{\ob a\in I} \bar A_{\ob a})\text{-pmod} $, thanks
  to the general  result on locally unital  algebras in  \cite[Corollary~2.5]{BRUNDAN}, $\Delta$ is an equivalence  between  $(\bigoplus_{\ob a\in I} \bar A_{\ob a})$-mod and  $A$-mod.
  Since  $\bigoplus_{\ob a\in I} \bar A_{\ob a}$ (a quotient of $B$ by Proposition ~\ref{bara}(4)) is semisimple, so is $A$.
  \end{proof}

\section{Weakly cellular bases } In this section, we consider a $\Bbbk$-algebra $A$   which  admits an upper finite weakly triangular decomposition in the sense of Definition~\ref{WT} such that $I=J$. Furthermore, on each $A$, there is  an anti-involution
$\sigma_A$ satisfying  \eqref{ssgm}. We will prove that any $\Bbbk$-algebra  $A_{  a}$ is a weakly cellular algebra if   $\bar A_{  b}$ is a cellular algebra for any $  b\in J$.
 We start by  recalling  the definition of a cellular algebra  in \cite{GL}.

 A cellular algebra $H$ is  an  associative $\Bbbk$-algebra which  has    basis   $\{  c_{\s,\t}^\lambda\mid \s,\t\in T(\lambda),\forall \lambda\in \Upsilon\}$ such that \begin{itemize} \item[(C1)]  $(\Upsilon,\unlhd) $ is a poset,
  \item[(C2)]  $T(\lambda)$ is a finite set for any $\lambda\in \Upsilon$,
  \item [(C3)] $ \sigma (  c^\lambda_{\s,\t})=  c^\lambda_{\t,\s}$, where $\sigma$  is a $\Bbbk$-linear anti-involution on $H$,

\item[(C4)]
$  c^\lambda_{\s,\t} x\equiv\sum_{\u} r_{\t, \u}(x)    c^\lambda_{\s,\u} \pmod {H^{\rhd \lambda}}$ for any $x\in H$, where \begin{enumerate}\item  $r_{\t, \u}(x)$'s $\in \Bbbk$ and are independent of $\s$,\item  $ H^{\rhd \lambda}$ is  spanned by
 $\{  c_{\s,\t}^\mu\mid \s,\t\in T(\mu), \mu \in \Upsilon, \mu \rhd \lambda\}$.\end{enumerate}
 \end{itemize}

 Thanks to (C4),    $ H^{\rhd \lambda}$ is
a   two-sided ideal of $H$.  Goodman  keeps (C1), (C2), (C4) and
replaces (C3)  by \eqref{sg1} as follows:
\begin{equation} \label{sg1} \sigma (  c^\lambda_{\s,\t})\equiv   c^\lambda_{\t,\s} \pmod{H^{\rhd \lambda}}.\end{equation}
 He calls the corresponding algebra  a weakly cellular algebra~\cite{G05}.
It has been pointed in \cite{G05} that  all results  for the representation theory of cellular algebras in \cite{GL} are still true
for weakly cellular algebras.

  Thanks to \cite{GL}, for any $\lambda\in \Upsilon$, there is  a cell module, say $S(\lambda)$  with respect to the (weakly)  cellular basis as above.
It is known that there is an invariant form on   $S(\lambda)$.  Let  $\bar \Upsilon=\{\lambda\in \Upsilon\mid S(\lambda)/\text{rad}\neq 0\}$  where $\text{rad}$ is the radical of this invariant form.
Then  the non-zero quotient module $S(\lambda)/\text{rad}$, denoted by  $D(\lambda)$,  is absolutely irreducible.
Later on, we will say that  $S(\lambda)$ has the simple head $D(\lambda)$ if $D(\lambda)\neq 0$.
 Following  \cite{GL},
 such $D (\lambda)$'s consist of  a complete set of pairwise inequivalent  simple  $H$-modules.
For any $\lambda, \mu\in \Upsilon$, say $\lambda$ and $\mu$ are \textsf{cell-link} for a (weakly) cellular algebra $H$  if $\lambda\in \bar\Upsilon $ such that $[S(\mu): D(\lambda)]\neq 0$~\cite[Page 14]{GL}.
We will use cell-link to study  blocks of the $\Bbbk$-algebras associated to
cyclotomic Brauer categories and  cyclotomic Kauffman categories at the end of this paper.

  \begin{Lemma}\label{sysanti}For all $ a, c \in J$,
$| X(  c,   a)|=|Y(  a,   c)|$.
\end{Lemma}
\begin{proof}We can assume  $ c\succeq  a$ without loss of any generality. Otherwise,  $ X( c,  a)=Y( a,  c)=\emptyset$. Since we are assuming that there is an
anti-involution
$\sigma_A$ satisfying  \eqref{ssgm},
 \begin{equation}\label{antihsh} \sigma_A( A_{a ,c})=A_{c, a}\end{equation} for  all  $ a, c \in J$. In particular,  $\dim  A_{c, a}=
\dim A_{a,  c}$. By  Lemma~\ref{cellbasis},
 \begin{equation}\label{dimequl}
| Y( a,  c)|| H( c)|+\sum_{ b\succ c} |Y(  a,   b)| |H(  b)|| X(  b,  c)|=
|H(  c)||X(  c,   a)|+\sum_{  b\succ  c} | Y(  c,   b)|| H(  b)|| X(  b,  a)|.
\end{equation}
If $  c$ is maximal,  then $|Y(  a,   c)|| H(  c)|=|H(  c)|| X(  c,  a)|$,
forcing $|Y(  a,   c)|= |X(  c,  a)|$. Otherwise, by
 induction assumption, $| Y(  a,   b)|= | X(  b,  a)|$ and $|Y(  c,   b)|= |X(  b,  c)|$ for all $  b\succ  c$.
By \eqref{dimequl},  we still have  $|Y(  a,   c)|= | X(  c,  a)|$.
\end{proof}

\begin{Theorem}\label{cellstructureofa0} For any $  b\in J$, suppose  $\bar A_{  b}$ is a cellular algebra  with respect to  the anti-involution $\sigma_{\bar A_{  b}}$ induced by $\sigma_{A}$. Then  $A_{  a}$ is a weakly cellular algebra for any $  a\in J$.
\end{Theorem}

\begin{proof} For any $  a\in J$, suppose that $\bar A_{  a}$ has cellular basis
\begin{equation}\label{ceboaa}
\{\bar h_{\s,\t}^\lambda\mid \s,\t\in T_{  a}(\lambda),\forall \lambda\in \Lambda_  a\},
\end{equation}
 where $\bar h^\lambda_{\s,\t}=\pi_  a(h^\lambda_{\s,\t})$,  $h^\lambda_{\s,\t}\in A_  a$ and $\pi_  a$ is given in Proposition~\ref{bara}(4).
    We use $T_{  a} $ and $\Lambda_{  a}$ here to emphasis that both of them depend on $  a$.
 Thanks to Lemma~\ref{cellbasis} and Proposition~\ref{bara}(1),(4),  we can choose all $h^\lambda_{\s,\t}\in  A^\circ_  a$ such that
\begin{equation}\label{cellbasisiss}\{h^\lambda_{\s,\t}\mid \s,\t\in T_  a(\lambda), \lambda\in \Lambda_  a\} \text{ is a  basis of }A^\circ_  a.
\end{equation}
Moreover, by Lemma~\ref{cellbasis} and  Lemma~\ref{linearcom}(2),  $  A_  a\cap A^{\succ  b}$  has  basis given by
\begin{equation}\label{cella} \{   y_1h_{\s,\t}^\lambda x_1 \mid (x_1, y_1)\in X(  c,  a)\times  Y(  a,  c),\forall \lambda\in \Lambda_  c,  \s,\t\in T_  c(\lambda),   c \succ   b \}.\end{equation}
Replacing  $   c \succ   b $ by  $  c \succeq   b$ in \eqref{cella} yields a basis of  $ A_  a\cap A^{\succeq   b}$.

We claim that we still get a basis of $  A_  a\cap A^{\succ  b}$   if  $y_1$  in \eqref{cella} is  replaced  by $\sigma_A(x_2), x_2\in X(  c,   a)$.
If $  b$ is maximal, then $  A_  a\cap A^{\succ  b}=0$, and  there is nothing to prove.
In general,   by  Lemma~\ref{linearcom}(2) and  \eqref{antihsh}, for $\lambda\in \Lambda_  c $ and $y_1\in Y(  a,  c) $,
$$ y_1h_{\s,\t}^\lambda \equiv \sum_{x_2\in  X(  c,   a), \mu\in \Lambda_  c, \s',\t'\in T_  c(\mu)} r_{x_2, \s', \t'} \sigma_A(x_2)h_{\s',\t'}^\mu  \pmod { A^{\succ   c}}, $$
for some   $r_{x_2, \s', \t'}\in \Bbbk$. Hence,
$$ y_1h_{\s,\t}^\lambda x_1 \equiv \sum_{x_2\in  X(  c,   a), \mu\in \Lambda_  c, \s',\t'\in T_  c(\mu)} r_{x_2, \s', \t'} \sigma_A(x_2)h_{\s',\t'}^\mu x_1 \pmod {A_{  a}\cap  A^{\succ   c}}.$$
 By Lemma~\ref{sysanti}, the cardinalities of two sets are equal.
 So,
our claim follows.  Similarly, we obtain another basis of  $A_{  a}\cap A^{\succeq   b} $. Define \begin{equation}\label{gam1}  \Gamma_  a=\{  b\in J|   b \succeq   a, X(  b,   a)\neq \emptyset\}.\end{equation}
Since  $A_{  a}=A_{  a}\cap  A^{\succeq   a}$,
 $A_  a$ has basis given by  \begin{equation}\label{wcs} S_{  a} =\bigcup_{  b \in \Gamma_  a} \bigcup_{\lambda\in \Lambda_{  b}} \{ \sigma_A(x_2)h_{\s,\t}^\lambda x_1 \mid x_1,x_2\in X(  b,  a), \s,\t\in T_  b(\lambda)\}.\end{equation} To see that $S_{  a}$ is a weakly cellular basis, we need to verify (C1),(C2), (C4) and  \eqref{sg1}.
Suppose
\begin{equation}\label{jdhujdsdji}
\Lambda_{\succeq   a}^\sharp=\bigcup_{  b\in \Gamma_  a}\Lambda_  b.
\end{equation}
  For any $\lambda, \mu\in \Lambda_{\succeq   a}^\sharp$,  write $\lambda \unlhd \mu$ if either   $ \rho(\lambda) \prec \rho(\mu)$ or $\rho(\lambda)=\rho(\mu)$ and $\lambda \unlhd \mu$ in $ \Lambda_{\rho(\lambda)}$.
Then $( \Lambda_{\succeq   a}^\sharp , \unlhd)$ is the required poset in (C1) for $A_{  a}$. Obviously, (C2) holds.

Suppose $\lambda\in \Lambda_  b\subseteq\Lambda_{\succeq  a}^\sharp$. Then   $  b\succeq   a$ and $X(  b,   a)\neq\emptyset$. By  Lemma~\ref{cellbasis},  $\Ker\pi_  b|_{A_  b}=A_{  b} \cap A^{\succ   b}$. Since
  $\overline{\sigma_A(h_{\s,\t}^\lambda )}= \sigma_{\bar A_{  b}} (\bar h_{\s,\t}^\lambda )=\bar h_{\t, \s}^\lambda$
  we have $ \sigma_A(h_{\s,\t}^\lambda )\equiv h_{\t, \s}^\lambda \pmod{A_{  b} \cap A^{\succ   b}}$. So,

$$\sigma_A(\sigma_A(x_2)h_{\s,\t}^\lambda x_1 )\equiv  \sigma_A(x_1)h_{\t,\s}^\lambda x_2   \pmod{A_{  a} \cap A^{\succ   b}}.$$
Thanks to \eqref{cella}, $A_{  a} \cap A^{\succ   b} \subset   A_  a^{\rhd \lambda}$,
where  $ A_  a^{\rhd \lambda}$ is the subspace spanned by all $\sigma_A(x_2')h_{\s',\t'}^\mu x_1'$ with $\mu \rhd\lambda$ in $\Lambda_{\succeq  a}^\sharp$. This proves  \eqref{sg1} for $S_{  a}$.  Finally, we verify (C4). More explicitly, we want to  prove
\begin{equation} \label{cel123} \sigma_A(x_2)h_{\s,\t}^\lambda x_1 u\equiv \sum_{\v\in  T_  b(\lambda), x_3\in X(  b,  a) }  r_{\v,x_3} \sigma_A(x_2) h_{\s,\v}^\lambda x_3 \pmod { A_  a^{\rhd \lambda}},\end{equation}
 for any $u\in A_{  a}$, where  $r_{\v,x_3}$'s are independent of $x_2$ and $\s$.
 In fact, thanks to Lemma~\ref{linearcom}(2) and (W3), there is a $z\in A^{\succ  b}$ such that
 $$  x_1u= 1_{  b}x_1u= \sum  r_{ h', x_3}   h'   x_3+z,$$
where the summation is   over $(   h',  x_3)\in    H(  b)\times X(  b,  a)$.
Since  $A_  a A^{\succ   b}A_  a\subset A_  a\cap A^{\succ   b}$, we have  $\sigma_A(x_2)h_{\s,\t}^\lambda z\in A_{  a}\cap A^{\succ   b}\subset A_  a^{\rhd \lambda}$.
So,
$$\begin{aligned} \sigma_A(x_2)h_{\s,\t}^\lambda x_1u&\equiv\sigma_A(x_2)h_{\s,\t}^\lambda \sum r_{  h',  x_3}   h'   x_3
&\equiv
\sum_{\v\in  T_  b(\lambda), x_3\in X(  b,  a) }  r_{\v,x_3} \sigma_A(x_2)h_{\s,\v}^\lambda x_3 ~(\text{mod } A_  a^{\rhd \lambda}),
\end{aligned}$$
where the second $ \equiv$ follows  from Lemma~\ref{W5} and (C4) for $\bar A_  b$. Obviously,  the scalars  $r_{\v,x_3}$'s  are independent of $x_2$ and $\s$.
\end{proof}

In the remaining part of this section, we  keep conditions in Theorem~\ref{cellstructureofa0}.
So, $\bar A_  a$ (resp., $A_  a$)  has a cellular basis (resp., weakly cellular basis) as above.
For any $ \lambda\in \Lambda_{  a}$, let  $S_{  a}(\lambda)$ be  the corresponding left  cell module  of $\bar A_  a$.  Without loss of any generality, we can assume that  $S_{  a}(\lambda)$  has  basis
  \begin{equation}\label{basisofcell}
  \{\bar  h^\lambda_{\s,\t}+\bar A_{  a}^{\rhd \lambda}  \mid\s\in T_  a(\lambda)\}
  \end{equation} where $\t$ is any fixed element in  $ T_  a(\lambda)$.
   Let $\bar \Lambda_{  a}=\{\lambda\in \Lambda_{  a} \mid S_{  a}(\lambda)/\text{rad}\neq 0\} $, where  $\text{rad}$ is the radical of the invariant form induced by the corresponding cellular basis. Then  $\bar \Lambda_{  a}$ parameterizes all pairwise  inequivalent simple $\bar A_{  a}$-modules. Let  $L_{  a} (\lambda)=S_{  a}(\lambda)/\text{rad} $,  $\lambda\in \bar \Lambda_{  a}$.
 Similarly, for each $\lambda\in \Lambda_{  b}\subseteq \Lambda_{\succeq   a}^\sharp$ (see \eqref{jdhujdsdji}), the left cell module
 $\tilde S_{  a}(\lambda)$ has basis
 \begin{equation}\label{cellbasis1}
 \{ \sigma_A(x_2)h_{\s,\t}^\lambda x_1+ A_  a^{\rhd \lambda} \mid  x_2\in X(  b,  a), \s\in T_  b(\lambda)\},
 \end{equation}
 where $ x_1$ is any fixed element in $   X(  b,  a)$ and $\t$ is in \eqref{basisofcell}.
  Suppose simple $A_{  a} $-modules are parameterized by  $\tilde\Lambda_{\succeq   a}^\sharp \subseteq \Lambda_{\succeq   a}^\sharp $ in the sense that $\tilde S_{  a}(\lambda)$ has the  simple head  $\tilde L_a(\lambda)$ for any $\lambda\in \tilde\Lambda_{\succeq   a}^\sharp$.

   For any $\lambda\in \Lambda_  a$, let
\begin{equation}\label{tildelta} \tilde \Delta(\lambda)= \Delta( S_  a(\lambda)).\end{equation}
Thanks to \eqref{basisofcell} and Lemma~\ref{exactj}(1), $\tilde \Delta(\lambda)$ has basis
$\bigcup_{  b\preceq   a} V_{  b}$, where
\begin{equation}\label{tildeltabasis} V_{  b}= \{y\otimes_{\bar A_  a}(\bar  h^\lambda_{\s,\t}+\bar A_{  a}^{\rhd \lambda})\mid y\in Y(  b,  a), \s\in T_  a(\lambda) \}.\end{equation}

Let  $\rho: \bigcup_{ a \in J} \Lambda_{  a} \rightarrow J$ such that  $\rho(\lambda)=  a $ if $\lambda\in \Lambda_{  a}$, where $\Lambda= \bigcup_{ a \in J} \Lambda_{  a}$.
  Then  $\rho|_{\bar \Lambda}$ is  $\rho$ in Definition~\ref{rho}.
\begin{Prop}\label{xeijdiec1} Suppose  $(\lambda, \mu)\in\bar \Lambda\times \Lambda$.

\begin{itemize}
\item [(1)] $1_{  a}\tilde \Delta(\mu)=0$ if  $\rho(\mu)\not\in \Gamma_{  a}$; and $1_{  a}\tilde \Delta(\mu) \cong \tilde S_{  a}(\mu) $  if   $\rho(\mu)\in \Gamma_{  a}$ (see \eqref{gam1}).
\item [(2)]  $[\tilde \Delta( \mu): L(\lambda)]=\sum_{\nu\in\bar\Lambda_{\rho(\mu)}} [S_{\rho(\mu) } (\mu): L_{\rho(\mu)} (\nu)](P(\lambda):\Delta(\nu))$ if Assumption~\ref{pdual} holds.
\item [(3)] $1_{  a } L(\lambda)=\tilde L_a(\lambda)$ if $ \lambda\in\tilde\Lambda_{\succeq  a}^\sharp\cap \bar \Lambda$, and  $1_{  a } L(\lambda)=0$ if $ \lambda\not\in\tilde\Lambda_{\succeq  a}^\sharp\cap \bar \Lambda$.

\end{itemize}
\end{Prop}
\begin{proof}The first part of (1) is obvious.
Suppose
 $\rho(\mu)=b\in \Gamma_  a$. Thanks to \eqref{tildeltabasis},   $V_{  a}   $  is a basis of  $1_{  a}\tilde \Delta( \mu) $.
  By \eqref{cellbasis1},
$\{ yh_{\s,\t}^\mu x + A_  a^{\rhd \mu} \mid y\in Y(  a,  b), \s\in T_  b(\mu)\}
 $ is a basis of  $\tilde S_{  a}(\mu)$
 for some fixed $( x, \t)  \in X(b,  a)\times T_{b}(\mu)$.
Let  $\phi: \tilde S_{  a}(\mu)\rightarrow 1_{  a}\tilde \Delta( \mu) $ be the $\Bbbk$-linear isomorphism  such that
$$\phi(yh_{\s,\t}^\mu x + A_  a^{\rhd \mu})=y\otimes_{\bar A_ b}(\bar  h^\mu_{\s,\t}+\bar A_{  b}^{\rhd \mu}).$$
Thanks to  Lemma~\ref{linearcom}(2) and (W3), it is routine to check that $\phi$ is an $A_{  a}$-homomorphism, proving (1).

 Since $\Delta$ is an exact functor,   $\tilde \Delta( \mu)$ has a finite  $\bar\Delta$-flag such that each section is of form $\bar\Delta(\nu)$ for some $\nu\in\bar \Lambda_{\rho(\mu)}$. Furthermore, $(\tilde \Delta( \mu): \bar\Delta(\nu))=  [S_{\rho(\mu)}(\mu): L_{\rho(\mu)}(\nu)]$ by Lemma~\ref{exactj}(2).  So,
$$\begin{aligned}~\quad
 [\tilde \Delta( \mu): L(\lambda)] &=\sum_{\nu\in\bar \Lambda_{\rho(\mu)}} (\tilde \Delta( \mu): \bar\Delta(\nu))[\bar\Delta(\nu):L(\lambda)]
=\sum_{\nu\in\bar \Lambda_{\rho(\mu)}} [S_{\rho(\mu)}(\mu): L_{\rho(\mu)}(\nu)][\bar\Delta(\nu):L(\lambda)]\\
&=\sum_{\nu\in\bar\Lambda_{\rho(\mu)}} [S_{\rho(\mu)}(\mu): L_{\rho(\mu)}(\nu)](P(\lambda):\Delta(\nu)), \text{by  Lemma~\ref{isodual1}(2), Proposition~\ref{dektss1}(2)}.\end{aligned}$$

Suppose $\lambda\in \bar\Lambda_  b$. If $1_  aL(\lambda)\neq 0$, by Lemma~\ref{barcom},  $1_{  a}\bar\Delta(\lambda)\neq 0$ and    $  b\succeq  a$.
Since $1_{  a}$ is an idempotent,  $\{1_{  a}L(\lambda)\neq0\mid \lambda\in \bigcup_{  b\succeq  a} \bar\Lambda_  b\}$ consists of a complete set of all pairwise inequivalent  simple $ A_  a$-modules.
So, $$\{1_{  a}L(\lambda)\neq0\mid \lambda\in \bigcup_{  b\succeq  a} \bar\Lambda_  b\}=\{ \tilde L_a(\lambda)\mid \lambda\in \tilde \Lambda_{\succeq  a}^\sharp \}.$$ In order to prove (3),
it suffices to verify  that  $1_  a L(\lambda)\cong \tilde L_a(\lambda)$  if $1_  a L(\lambda)\neq 0$.
If the result were false, then
there is $\nu\in \bar\Lambda_  c$ with $  c \succeq  a$ such that   $0\neq 1_  a L(\nu )\cong \tilde L_a(\lambda)$,  for some $\lambda\in\Lambda_  b$ with $  b\succeq   a$ and $\lambda\neq \nu$. Since $L_  c(\nu)$ is a simple head of $S_  c(\nu)$, there is  an epimorphism  $\tilde\Delta( \nu)\twoheadrightarrow L(\nu)$.
Applying $1_  a$ yields an epimorphism  $\tilde S_  a(\nu )\cong 1_  a\tilde\Delta( \nu)\twoheadrightarrow 1_  aL(\nu)$. So,
$[\tilde S_  a(\nu ):\tilde L_a(\lambda) ]\neq 0$ and hence $\lambda  \unlhd \nu$ and $\rho(\nu) \succeq  b$.
On the other hand, since $1_{  a}$ is an idempotent,
 $$[\tilde\Delta(\lambda): L(\nu)]=[1_  a\tilde\Delta(\lambda): 1_  aL(\nu)]=[\tilde S_  a(\lambda): 1_  aL(\nu)]=[\tilde S_  a(\lambda): \tilde L_a(\lambda)]\neq 0.$$
  By (2), $\rho(\nu)\preceq   b$, forcing   $\rho(\nu)=  b$.
   So, $[S_  b(\lambda): L_  b(\nu)]=[1_  b\tilde \Delta(\lambda): 1_  b L(\nu)]=[\tilde \Delta(\lambda): L(\nu)]\neq 0$. This proves
     $\nu\unlhd \lambda$. Therefore,   $\lambda=\nu$,   a contradiction, and (3) follows.
\end{proof}
\begin{Cor}\label{ijxxexeu} Keep the Assumption~\ref{pdual}.  Suppose  $\lambda\in\bar\Lambda_  a$. Then   $P(\lambda)$ has a finite $\tilde\Delta$-flag  such that $\mu\in \bigcup_{  b\succeq  a}  \Lambda_  b$ if $\tilde \Delta( \mu)$ appears as a section. Furthermore, the  multiplicity of $\tilde \Delta( \mu)$ in this flag is  $[\tilde \Delta( \mu):L(\lambda)]$.
\end{Cor}
\begin{proof}By Proposition ~\ref{dektss1}(1),  $P(\lambda)$ has a finite $\Delta$-flag. Thanks to \cite[Lemmas~2.9, 2.10]{GL}, any finitely generated projective module for a cellular algebra has a cell filtration. So,  $P(\lambda)$ has a finite $\tilde \Delta$-flag.
Let  $(P(\lambda):\tilde \Delta( \mu))$ be the
 multiplicity of $\tilde \Delta( \mu)$ in this flag. Suppose $\mu\in \Lambda_{  b}$. Thanks to Lemma~\ref{exactj}(2), Proposition ~\ref{dektss1} and Lemma~\ref{isodual1}(2) (the place that we need the Assumption~\ref{pdual}), we have
$$\begin{aligned} (P(\lambda):\tilde \Delta(\mu)) &  =\sum_{\nu\in \bar \Lambda_{  b}  } (P(\lambda): \Delta(\nu)) (\Delta(\nu):\tilde \Delta(\mu))
 =\sum_{\nu\in \bar\Lambda_{  b}} [\bar \Delta(\nu): L(\lambda)](P_{  b}(\nu): S_{  b}(\mu)),\\
  [\tilde \Delta(\mu): L(\lambda)] &
 = \sum_{\nu\in \bar\Lambda_{  b}}  [\bar \Delta(\nu): L(\lambda)](\tilde \Delta(\mu): \bar \Delta(\nu))=\sum_{\nu\in \bar\Lambda_{  b}} [\bar \Delta(\nu): L(\lambda)] [S_{  b }(\mu): L_{  b }(\nu)].\\
 \end{aligned}
 $$
 By \cite[Lemmas~2.18,2.19]{Ma},
 $(P_{  b}(\nu): S_{  b}(\mu))=[S_{  b }(\mu): L_{  b }(\nu)]$, and hence $(P(\lambda):\tilde \Delta(\mu))=[\tilde \Delta(\mu): L(\lambda)]$.
 Thanks to   Propositions~\ref{xeijdiec1}(2),~\ref{dektss1}(2),    $\mu\in \bigcup_{  b\succeq  a}  \Lambda_  b$
 if $(P(\lambda):\tilde \Delta(\mu))\neq 0$.
  \end{proof}
\begin{Prop}\label{blockp} Keep the Assumption~\ref{pdual}. For any $\lambda, \mu\in \bar \Lambda$, $L(\lambda)$ and $L(\mu)$ are in the same
block  (and write $L(\lambda)\sim L(\mu)$)   if  there is a sequence $\lambda=\lambda_1,\lambda_2,\ldots,\lambda_n=\mu$ such that each pair
  $\lambda_{i}$ and $\lambda_{i+1}$ are cell-link in a weakly cellular algebra  $A_{  a}$ for some $  a\in J$.
\end{Prop}
\begin{proof} By assumption and Proposition~\ref{xeijdiec1}(1)(3), there is a sequence $\lambda_1, \lambda_2,\ldots,\lambda_n$  such that $\lambda_1=\lambda$, $\lambda_n=\mu$ and either $[1_{  a}\tilde \Delta(\lambda_i): 1_{  a} L(\lambda_{i+1})]\neq 0$ or $[1_{  a}\tilde \Delta(\lambda_{i+1}): 1_{  a}L(\lambda_{i})]\neq 0$ for some $  a$ which depends on  $i$.
  So, either $[\tilde \Delta(\lambda_i): L(\lambda_{i+1})]\neq 0$ or $[\tilde \Delta(\lambda_{i+1}):  L(\lambda_{i})]\neq 0$.

It remains to prove  that   $\tilde \Delta(\nu)$ is in a single block for any $\nu\in\Lambda_a$ and any $a\in J$. Suppose  $\nu'\in\bar\Lambda_a$ such that
$[S_{  a }(\nu): L_{  a }(\nu')]\neq 0$. By \cite[Lemmas~2.18,2.19]{Ma},  $S_a(\nu)$ appears as a section  in a cell filtration of $P_a(\nu')$. Thanks to Lemma~\ref{exactj}(2), $\tilde \Delta(\nu)$ appears in a $\tilde \Delta$-flag
of $\Delta(\nu')$. Since $\Delta(\nu')$ is indecomposable (see Theorem~\ref{striateddndn}(2)), $\tilde \Delta(\nu)$ has to be in a single block.
\end{proof}

\section{ Endofunctors and categorical actions }\label{indkkk}
Adapting the usual string calculus for strict monoidal categories, we have that the composition of morphisms  in a strict monoidal category  is given by vertical   stacking and the tensor product of morphisms  is given by   horizontal concatenation.
In this section, we consider two kinds of $\Bbbk$-linear strict   monoidal categories $\mathcal C_1$ and $\mathcal C_2$
generated by a single generating object \begin{tikzpicture}[baseline = 1.5mm]
	\draw[-,thick,darkblue] (0.18,0) to (0.18,.4);
\end{tikzpicture}.  So the set of objects is $\mathbb N$ and   $m$ represents $\begin{tikzpicture}[baseline = 1.5mm]
	\draw[-,thick,darkblue] (0.18,0) to (0.18,.4);
\end{tikzpicture}^{\otimes m} $.
Furthermore, the generating morphisms of   $\mathcal C_1$ contain   at least  three      morphisms  $\lcap:2\rightarrow0$, $\lcup:0\rightarrow2$, and
\begin{tikzpicture}[baseline=-.5mm]
\draw[-,thick,darkblue] (0,-.3) to (0,.3);
      \node at (0,0) {$\color{darkblue}\scriptstyle\bullet$};
\end{tikzpicture}$:1\rightarrow1$ satisfying at least two  relations:

\begin{equation}\label{relation 1}
\begin{tikzpicture}[baseline = -0.5mm]
\draw[-,thick,darkblue] (0,0) to (0,0.3);
\draw[-,thick,darkblue] (0,0) to[out=down,in=left] (0.25,-0.25) to[out=right,in=down] (0.5,0);
\draw[-,thick,darkblue] (0.5,0) to[out=up,in=left] (0.75,0.25) to[out=right,in=up] (1,0);
\draw[-,thick,darkblue] (0,0) to (0,0.3);
\draw[-,thick,darkblue] (1,0) to (1,-0.3);
  \end{tikzpicture}
  ~=~
  \begin{tikzpicture}[baseline = -0.5mm]
  \draw[-,thick,darkblue] (0,-0.25) to (0,0.25);
   \node at (0.,-0.6) {$\text {I}$};
  \end{tikzpicture}
  ~=~\begin{tikzpicture}[baseline = -0.5mm]
\draw[-,thick,darkblue] (0.5,0) to[out=down,in=left] (0.75,-0.25) to[out=right,in=down] (1,0);
\draw[-,thick,darkblue] (0,0) to[out=up,in=left] (0.25,0.25) to[out=right,in=up] (0.5,0);
\draw[-,thick,darkblue] (1,0) to (1,0.3);
\draw[-,thick,darkblue] (0,0) to (0,-0.3);
  \end{tikzpicture}~,\qquad
\begin{tikzpicture}[baseline = -0.5mm]
  \draw[-,thick,darkblue] (0,0) to[out=up,in=left] (0.28,0.28) to[out=right,in=up] (0.56,0);
  \draw[-,thick,darkblue] (0,0) to (0,-.2);
 \node at (0,0) {$\color{darkblue}\scriptstyle\bullet$};
 \draw[-,thick,darkblue] (0.56,0) to (0.56,-.2);
  \end{tikzpicture}
 \begin{tikzpicture}[baseline = -1mm]
	 \node at (0.,0) {$\text {=}$};
 \node at (0.,-0.6) {$\text {II}$};
\end{tikzpicture}-\begin{tikzpicture}[baseline = -0.5mm]
\draw[-,thick,darkblue] (0,0) to[out=up,in=left] (0.28,0.28) to[out=right,in=up] (0.56,0);
 \draw[-,thick,darkblue] (0.56,0) to (0.56,-.2);
 \draw[-,thick,darkblue] (0,0) to (0,-.2);
 \node at (0.56,0) {$\color{darkblue}\scriptstyle\bullet$};
  \end{tikzpicture}~,
  \end{equation}
  where $\begin{tikzpicture}[baseline = 10pt, scale=0.5, color=\clr]
                \draw[-,thick] (0,0.5)to[out=up,in=down](0,1.5);
    \end{tikzpicture}: 1\rightarrow 1$ is the identity morphism. Archetypal example is
the affine  Brauer category in \cite{RS3}.
For  $\mathcal C_2$, there are at least four  generating   morphisms  $\lcap:2\rightarrow0$, $\lcup:0\rightarrow2$,
\begin{tikzpicture}[baseline=-.5mm]
\draw[-,thick,darkblue] (0,-.3) to (0,.3);
      \node at (0,0) {$\color{darkblue}\scriptstyle\bullet$};
\end{tikzpicture}, \begin{tikzpicture}[baseline=-.5mm]
\draw[-,thick,darkblue] (0,-.3) to (0,.3);
      \node at (0,0) {$\color{darkblue}\scriptstyle\circ$};
\end{tikzpicture}$:1\rightarrow1$ satisfying the relation  \eqref{relation 1}I and at least two more relations  as  follows:

\begin{equation}\label{relation 3}
 \begin{tikzpicture}[baseline = -0.5mm]
\draw[-,thick,darkblue] (0,0) to (0,.5);
      \node at (0,0.3) {$\color{darkblue}\scriptstyle\bullet$};
      \draw[-,thick,darkblue] (0,0.2) to (0,-.2);
      \node at (0,0) {$\color{darkblue}\scriptstyle\circ$};
  \end{tikzpicture}
  ~=~
  \begin{tikzpicture}[baseline = -0.5mm]
  \draw[-,thick,darkblue] (0,-0.2) to (0,0.5); \node at (0.,-0.6) {$\text {I}$};
  \end{tikzpicture}
  ~=~\begin{tikzpicture}[baseline = -0.5mm]
  \draw[-,thick,darkblue] (0,0) to (0,.5);
      \node at (0,0.3) {$\color{darkblue}\scriptstyle\circ$};
      \draw[-,thick,darkblue] (0,0.2) to (0,-.2);
      \node at (0,0) {$\color{darkblue}\scriptstyle\bullet$};
  \end{tikzpicture}, \quad
  \begin{tikzpicture}[baseline = -0.5mm]
  \draw[-,thick,darkblue] (0,0) to[out=up,in=left] (0.28,0.28) to[out=right,in=up] (0.56,0);
  \draw[-,thick,darkblue] (0,0) to (0,-.2);
 \node at (0,0) {$\color{darkblue}\scriptstyle\bullet$};
 \draw[-,thick,darkblue] (0.56,0) to (0.56,-.2);
  \end{tikzpicture}
   \begin{tikzpicture}[baseline = -1mm]
  \node at (0.,0) {$\text {=}$};
 \node at (0.,-0.6) {$\text {II}$};
\end{tikzpicture}\begin{tikzpicture}[baseline = -0.5mm]
\draw[-,thick,darkblue] (0,0) to[out=up,in=left] (0.28,0.28) to[out=right,in=up] (0.56,0);
 \draw[-,thick,darkblue] (0.56,0) to (0.56,-.2);
 \draw[-,thick,darkblue] (0,0) to (0,-.2);
 \node at (0.56,0) {$\color{darkblue}\scriptstyle\circ$};
  \end{tikzpicture}.
  \end{equation}
So, \begin{tikzpicture}[baseline=-.5mm]
\draw[-,thick,darkblue] (0,-.3) to (0,.3);
      \node at (0,0) {$\color{darkblue}\scriptstyle\bullet$};
      \end{tikzpicture}
is invertible with inverse \begin{tikzpicture}[baseline=-.5mm]
\draw[-,thick,darkblue] (0,-.3) to (0,.3);
      \node at (0,0) {$\color{darkblue}\scriptstyle\circ$};
      \end{tikzpicture}.
Archetypal example is the affine Kauffman category in ~\cite{GRS}.

 For   $\mathcal C\in\{\mathcal C_1, \mathcal C_2\}$,   let $\mathbb D$  be the set of all diagrams  obtained  by tensor product and composition of generating   morphisms in $\mathcal C$ together with  \begin{tikzpicture}[baseline = 10pt, scale=0.5, color=\clr]
                \draw[-,thick] (0,0.5)to[out=up,in=down](0,1.5);
    \end{tikzpicture}.
 Fix any    right tensor ideal $K$ of $\mathcal C$ and consider the quotient category $ \mathcal A=\mathcal C/K$ and its  associated algebra
\begin{equation}\label{z1}A= \bigoplus_{a, b \in \mathbb N}\Hom_{\mathcal A}(  a,  b).\end{equation}
Since  $\mathcal A$ is a quotient category of $\mathcal C$, any diagram in $\mathbb D$ can be interpreted as an element in $A$.
 The monoidal functor $ \tau=?\otimes \begin{tikzpicture}[baseline = -0.5mm]
  \draw[-,thick,darkblue] (0,-0.25) to (0,0.25);
  \end{tikzpicture}:\mathcal C\rightarrow \mathcal C$ stabilizes any right tensor ideal of $\mathcal C$. It
  induces an algebra homomorphism \begin{equation} \label{tau2} \tau_{A}: A\rightarrow A\end{equation}  sending  $f$ to $ f ~\begin{tikzpicture}[baseline = -0.5mm]
  \draw[-,thick,darkblue] (0,-0.25) to (0,0.25);
  \end{tikzpicture}$ for all $f\in \mathbb D$.

\subsection{Assumptions}\label{assump0} In   this section, we keep the following assumptions.
\begin{itemize} \item [(A1)]The $\Bbbk$-algebra $A$ in \eqref{z1}    admits an upper finite weakly triangular decomposition  such that $I=J=\mathbb N$. The partial order $\preceq $
 on $\mathbb N$ is defined such that
$ i\preceq j$   if $i=j+2s$ for some $s\in\mathbb N$.
\item[(A2)] For any $m\in\mathbb N$,
 $ f ~\begin{tikzpicture}[baseline = -0.5mm]
  \draw[-,thick,darkblue] (0,-0.25) to (0,0.25);
  \end{tikzpicture}\in  Q(m+1)$ for all $f\in Q(m)$, where $Q\in \{X, Y, H\}$ (see Definition~\ref{defniogfxhy}).
  \item [(A3)] For any $m\in\mathbb N$,  there is a  $D(m+1)\subseteq H(m+1)$ such that $|H(m+1)|=|D(m+1)||H(m)|$ and  $H(m+1)=\{d\circ \tau_A(d')\mid (d, d')\in D(m+1)\times H(m)\}$.
      \end{itemize}
By (A1), we can use  previous results on upper finite weakly triangular decompositions.
Thanks to (A2) and Lemma~\ref{cellbasis},  $\tau_{A}$   is an algebra monomorphism. Since $A$ comes from a quotient category of  the $\Bbbk$-linear strict  monoidal category $\mathcal C$,    $ d\circ \tau_A(d')$ can be expressed via   the following diagram:
     \begin{equation}\label{subs1}
\begin{tikzpicture}[baseline = 3mm, color=\clr]
                                \draw (-0.1,0.4) rectangle (0.9,0);
                \draw(0.4, 0.2) node{$d'$}; \draw[-,thick] (1.06,0.53)to[out=up,in=down](1.06,0.01);
                \draw (1.2,0.55) rectangle (-0.1,0.85); \draw[-,thick] (0, 0.4)to[out=up,in=down](0,0.55); \draw(0.25, 0.48) node{$ \cdots$};  \draw[-,thick] (0.6,0.4)to[out=up,in=down](0.6,0.55);
                \draw(0.6, 0.7) node{$d$};
    \end{tikzpicture}.
 \end{equation}
\subsection{Endofunctors}Suppose  $A_L=\bigoplus_{m, n\in \mathbb N} 1_{ m} A_L1_{ n} $ and $A_R=\bigoplus_{m, n\in \mathbb N} 1_{ m} A_R1_{ n}$    such that
\begin{equation}\label{weights123} 1_{ m} A_L1_{ n}=1_{ {m+1}}A1_{ n}, \text{ $1_{ m} A_R1_{ n}=1_{ {m}}A1_{ {n+1}} $.}\end{equation}
 Then both  $A_L$ and  $A_R$ are   $(A,A)$-bimodules
 such that the right (resp., left) action of  $A$  on $A_L$ (resp., $A_R$) is given by the usual multiplication, whereas the left (resp., right) action of $A$ on   $A_L$ (resp., $A_R$) is given as follows:
\begin{equation}\label{act123} a\cdot f= \begin{tikzpicture}[baseline = 3mm, color=\clr]
                                \draw (-0.1,0.4) rectangle (1.2,0);
                \draw(0.5, 0.2) node{$f$};
                \draw (0.9,0.55) rectangle (-0.1,0.85); \draw[-,thick] (0, 0.4)to[out=up,in=down](0,0.55); \draw(0.25, 0.48) node{$ \cdots$};   \draw[-,thick] (0.6,0.4)to[out=up,in=down](0.6,0.55);
                \draw(0.5, 0.7) node{$a$}; \draw[-,thick] (1.06,0.4)to[out=up,in=down](1.06,0.85);
    \end{tikzpicture}, \quad  g\cdot a=\begin{tikzpicture}[baseline = 3mm, color=\clr]
                                \draw (-0.1,0.4) rectangle (0.9,0);
                \draw(0.4, 0.2) node{$a$}; \draw[-,thick] (1.06,0.53)to[out=up,in=down](1.06,0.01);
                \draw (1.2,0.55) rectangle (-0.1,0.85); \draw[-,thick] (0, 0.4)to[out=up,in=down](0,0.55); \draw(0.25, 0.48) node{$ \cdots$};  \draw[-,thick] (0.6,0.4)to[out=up,in=down](0.6,0.55);
                \draw(0.6, 0.7) node{$g$};
    \end{tikzpicture},\end{equation}
     for all   $ (f, g, a)\in A_L\times   A_R\times  A $.  Let \begin{equation}\label{EF} E=A_L\otimes _{A}?, \ \   F=A_R\otimes _{A}?.\end{equation}
\begin{Lemma}\label{bimoudisomah}  We have
\begin{itemize}\item[(1)] $E\cong F$ as functors,
\item[(2)] $E$ is self-adjoint and exact.\end{itemize}
\end{Lemma}

\begin{proof} Define $\phi: A_L\rightarrow A_R$ and $\psi: A_R\rightarrow A_L$ such that  $$\phi(f)=\begin{tikzpicture}[baseline = 1mm, color=\clr]
                \draw[-,thick] (0,0.5)to[out=up,in=down](0,0.8); \draw(0.2, 0.6) node{$ \cdots$};  \draw[-,thick] (0.4,0.5)to[out=up,in=down](0.4,.8);  \draw[-,thick,darkblue] (0.5,0.5) to[out=up,in=left] (0.78,0.78) to[out=right,in=up] (1.06,0.5);\draw (-0.1,0.5) rectangle (0.9,0);
                \draw(0.4, 0.25) node{$f$}; \draw[-,thick] (1.06,0.5)to[out=up,in=down](1.06,0.01);
    \end{tikzpicture},~~~\qquad\psi(g)=\begin{tikzpicture}[baseline = 1mm, color=\clr]
                \draw[-,thick] (0,-0.3)to[out=up,in=down](0,0); \draw(0.2, -0.2) node{$ \cdots$};  \draw[-,thick] (0.4,-0.3)to[out=up,in=down](0.4,0);  \draw[-,thick,darkblue] (0.5,0) to[out=down,in=left] (0.78,-0.28) to[out=right,in=down] (1.06,0);\draw (-0.1,0.5) rectangle (0.9,0);
                \draw(0.4, 0.25) node{$g$}; \draw[-,thick] (1.06,0.5)to[out=up,in=down](1.06,0.01);
    \end{tikzpicture},~$$ where $f\in 1_{ {m+1}}A$ and $g\in A 1_{ {m+1}}$.  By \eqref{relation 1}I, we have
      $$\psi\phi(f)=\begin{tikzpicture}[baseline = 1mm, color=\clr]
                \draw[-,thick] (0,0.5)to[out=up,in=down](0,0.8); \draw(0.2, 0.6) node{$ \cdots$};  \draw[-,thick] (0.4,0.5)to[out=up,in=down](0.4,.8);  \draw[-,thick,darkblue] (0.5,0.5) to[out=up,in=left] (0.78,0.78) to[out=right,in=up] (1.06,0.5);\draw (-0.1,0.5) rectangle (0.9,0);
                \draw(0.4, 0.25) node{$f$}; \draw[-,thick] (1.06,0.5)to[out=up,in=down](1.06,0);\draw[-,thick,darkblue] (1.06,0) to[out=down,in=left] (0.28+1.06,-0.28) to[out=right,in=down] (1.06-0.5+1.06,0);\draw[-,thick] (1.06-0.5+1.06,0.8)to[out=up,in=down](1.06-0.5+1.06,0);
    \end{tikzpicture}=f,~~~\qquad\phi\psi(g)=\begin{tikzpicture}[baseline = 1mm, color=\clr]
                \draw[-,thick] (0,-0.3)to[out=up,in=down](0,0); \draw(0.2, -0.2) node{$ \cdots$};  \draw[-,thick] (0.4,-0.3)to[out=up,in=down](0.4,0);  \draw[-,thick,darkblue] (0.5,0) to[out=down,in=left] (0.78,-0.28) to[out=right,in=down] (1.06,0);\draw (-0.1,0.5) rectangle (0.9,0);
                \draw(0.4, 0.25) node{$g$}; \draw[-,thick] (1.06,0.5)to[out=up,in=down](1.06,0.01);\draw[-,thick,darkblue] (1.06,0.5) to[out=up,in=left] (0.28+1.06,0.78) to[out=right,in=up] (1.06-0.5+1.06,0.5);\draw[-,thick] (1.06-0.5+1.06,0.5)to[out=up,in=down](1.06-0.5+1.06,-0.3);
    \end{tikzpicture}=g.~$$
 This proves $\phi^{-1}=\psi$.
Finally, it is easy to see that both  $\phi$ and $\psi$ are $(A,A)$-homomorphisms. So  $A_L\cong A_R$  as   $(A,A)$-bimodules and (1) follows.
The unit $\beta:\text{Id}_{A\text{-mod}}\rightarrow F^2$ is given by  the bimodule homomorphism
$A\overset{\beta'}\rightarrow A_R\otimes_A A_R$ such that
 $$\beta' (f)=( 1_{m}~\lcap)\otimes(f~\begin{tikzpicture}[baseline = 10pt, scale=0.5, color=\clr]
                \draw[-,thick] (0,0.5)to[out=up,in=down](0,1.2);
    \end{tikzpicture}~)$$ for all  $f\in  1_{ m}A$, and
  the counit $\alpha:F ^2\rightarrow \text{Id}_{A\text{-mod}}$ is given by   the bimodule homomorphism
$A_R\otimes _A A_R\overset{\alpha'}\rightarrow A$ such that
$$
\alpha'(f\otimes g)=\begin{tikzpicture}[baseline = 1mm, color=\clr]
                \draw[-,thick] (0,-0.3)to[out=up,in=down](0,0); \draw(0.2, -0.2) node{$ \cdots$};  \draw[-,thick] (0.4,-0.3)to[out=up,in=down](0.4,0);  \draw[-,thick,darkblue] (0.5,0) to[out=down,in=left] (0.78,-0.28) to[out=right,in=down] (1.06,0);\draw (-0.1,0.4) rectangle (0.9,0);
                \draw(0.4, 0.2) node{$g$}; \draw[-,thick] (1.06,0.5)to[out=up,in=down](1.06,0.01);
                \draw (1.2,0.55) rectangle (-0.1,0.85); \draw[-,thick] (0, 0.4)to[out=up,in=down](0,0.55); \draw(0.25, 0.48) node{$ \cdots$};  \draw[-,thick] (0.6,0.4)to[out=up,in=down](0.6,0.55); \draw(0.6, 0.7) node{$f$};
    \end{tikzpicture},$$ for all $ (f, g)\in A 1_{ {n+1}}\times  1_{ n}A 1_{ {m+1}}$. For all admissible diagrams $a\in 1_m A$, $ g\in A_{R}1_n$ and $ f\in 1_nA_{R}$ we have $$\begin{aligned}&(\text{Id}_{A_R}\otimes \alpha')(\beta'\otimes \text{Id}_{A_R})(a\otimes g)=(\text{Id}_{A_R}\otimes \alpha')(( 1_{ m}~\lcap)\otimes(a~\begin{tikzpicture}[baseline = 10pt, scale=0.5, color=\clr]
                \draw[-,thick] (0,0.5)to[out=up,in=down](0,1.2);
    \end{tikzpicture}~)\otimes g)\\&=( 1_{ m}~\lcap)\otimes\begin{tikzpicture}[baseline = 1mm, color=\clr]
                \draw[-,thick] (0,-0.3)to[out=up,in=down](0,0); \draw(0.2, -0.2) node{$ \cdots$};  \draw[-,thick] (0.4,-0.3)to[out=up,in=down](0.4,0);  \draw[-,thick,darkblue] (0.5,0) to[out=down,in=left] (0.78,-0.28) to[out=right,in=down] (1.06,0);\draw (-0.1,0.4) rectangle (0.9,0);
                \draw(0.4, 0.2) node{$g$}; \draw[-,thick] (1.06,0.85)to[out=up,in=down](1.06,0.01);
                \draw (0.9,0.55) rectangle (-0.1,0.85); \draw[-,thick] (0, 0.4)to[out=up,in=down](0,0.55); \draw(0.2, 0.48) node{$ \cdots$};  \draw[-,thick] (0.5,0.4)to[out=up,in=down](0.5,0.55); \draw(0.4, 0.7) node{$a$};
    \end{tikzpicture}
    =( 1_m~\lcap)\circ\begin{tikzpicture}[baseline = 1mm, color=\clr]
                \draw[-,thick] (0,-0.3)to[out=up,in=down](0,0); \draw(0.2, -0.2) node{$ \cdots$};  \draw[-,thick] (0.4,-0.3)to[out=up,in=down](0.4,0);  \draw[-,thick,darkblue] (0.5,0) to[out=down,in=left] (0.78,-0.28) to[out=right,in=down] (1.06,0);\draw (-0.1,0.4) rectangle (0.9,0);
                \draw(0.4, 0.2) node{$g$}; \draw[-,thick] (1.06,0.85)to[out=up,in=down](1.06,0.01);
                \draw (0.9,0.55) rectangle (-0.1,0.85); \draw[-,thick] (0, 0.4)to[out=up,in=down](0,0.55); \draw(0.2, 0.48) node{$ \cdots$};  \draw[-,thick] (0.5,0.4)to[out=up,in=down](0.5,0.55); \draw(0.4, 0.7) node{$a$};
                \draw[-,thick] (1.2,-0.3)to[out=up,in=down](1.2,0.89);
    \end{tikzpicture}\otimes 1_n
    =\begin{tikzpicture}[baseline = 1mm, color=\clr]
                \draw[-,thick] (0,-0.3)to[out=up,in=down](0,0); \draw(0.2, -0.2) node{$ \cdots$};  \draw[-,thick] (0.4,-0.3)to[out=up,in=down](0.4,0);  \draw[-,thick,darkblue] (0.5,0) to[out=down,in=left] (0.78,-0.28) to[out=right,in=down] (1.06,0);\draw (-0.1,0.4) rectangle (0.9,0);
                \draw(0.4, 0.2) node{$g$}; \draw[-,thick] (1.06,0.85)to[out=up,in=down](1.06,0.01);
                \draw (0.9,0.55) rectangle (-0.1,0.85); \draw[-,thick] (0, 0.4)to[out=up,in=down](0,0.55); \draw(0.2, 0.48) node{$ \cdots$};  \draw[-,thick] (0.5,0.4)to[out=up,in=down](0.5,0.55); \draw(0.4, 0.7) node{$a$};
                \draw[-,thick,darkblue] (1.06,0.85) to[out=up,in=left] (0.28+1.06,0.78+0.35) to[out=right,in=up] (1.06-0.5+1.06,0.85);\draw[-,thick] (1.06-0.5+1.06,0.85)to[out=up,in=down](1.06-0.5+1.06,0.05);
    \end{tikzpicture}\otimes 1_n
    \overset{\text{ \eqref{relation 1}I}}
    =(a\cdot g)\otimes 1_n\end{aligned}$$
         and
      $$\begin{aligned}&(\alpha'\otimes \text{Id}_{A_R})(\text{Id}_{A_R}\otimes \beta')(f\otimes a)=(\alpha'\otimes \text{Id}_{A_R})(f\otimes ( 1_{ m}~\lcap)\otimes(a~\begin{tikzpicture}[baseline = 10pt, scale=0.5, color=\clr]
                \draw[-,thick] (0,0.5)to[out=up,in=down](0,1.2);
    \end{tikzpicture}~))
    \\&=\begin{tikzpicture}[baseline = 3mm, color=\clr]
                                \draw (1.25,0.55) rectangle (-0.1,0.85); \draw[-,thick] (0, 0.1)to[out=up,in=down](0,0.55); \draw(0.2, 0.38) node{$ \cdots$};  \draw[-,thick] (0.5,0.1)to[out=up,in=down](0.5,0.55);
                \draw(0.6, 0.7) node{$f$};\draw[-,thick,darkblue] (0.73+0.15,0.3) to[out=down,in=left] (0.73+0.3,0.1) to[out=right,in=down] (0.73+0.45,0.3);\draw[-,thick,darkblue] (0.58,0.3) to[out=up,in=left] (0.73,0.5) to[out=right,in=up] (0.73+0.15,0.3);
                \draw[-,thick] (0.73+0.45,0.3)to[out=up,in=down](0.73+0.45,0.55);\draw[-,thick] (0.58,0.3)to[out=up,in=down](0.58,0.1);
    \end{tikzpicture}\otimes (a~\begin{tikzpicture}[baseline = 10pt, scale=0.5, color=\clr]
                \draw[-,thick] (0,0.5)to[out=up,in=down](0,1.2);
    \end{tikzpicture}~)\overset{\text{\eqref{relation 1}I}}=1_n\otimes f\cdot a.\end{aligned}$$  So,   $(\text{Id}_{A_R}\otimes \alpha')(\beta'\otimes \text{Id}_{A_R}): A\otimes_A A_{R}\rightarrow A_{R}\otimes_A A $ and $(\alpha'\otimes \text{Id}_{A_R})(\text{Id}_{A_R}\otimes \beta'): A_{R}\otimes_A A\rightarrow  A\otimes_A A_{R}$  are identity morphism  of $(A, A)$-bimodules $A_R$ (under the obvious isomorphisms $A\otimes_AA_R\cong A_R$ and $A_R\cong A_R\otimes_A A$),
     and hence $$  \alpha F  \circ   F   \beta =\text{Id}_{ F}=  F  \alpha \circ \beta   F .$$ This shows that $ F$ is self-adjoint and exact. Now, (2) follows from (1), immediately.
 \end{proof}

 Recall the algebra monomorphism  $\tau_{ A}$   in \eqref{tau2}. Then  $\tau_{ A}\mid_{A_{m-1}}:   A_{m-1}\rightarrow A_m$ is
   an  algebra monomorphism. Since   $n-1\preceq m-1 \Leftrightarrow n\preceq m$ and  $\tau_A(1_{ {n}})=1_{ n+1}$ for all $n\in \mathbb N$,   $\tau_A(A^{\npreceq m-1})\subseteq A^{\npreceq m}$.
 By Lemma~\ref{cellbasis} and Proposition~\ref{bara}(2), $\tau_{ A}$ induces  an algebra monomorphism $$\bar\tau_{m-1}:A_{\preceq  {m-1}}\rightarrow A_{\preceq  {m}}.$$
Restricting   $\bar\tau_{m-1}$ to $\bar A_{m-1}$ gives a monomorphism from $\bar A_{m-1}$ to $\bar A_{m}$.
So,   $ \bar A_{m-1}$  can be identified with a subalgebra of $\bar A_m$ via $\bar\tau_{m-1}$ and $ \bar A_{m}$ can be  viewed as  an  $(\bar A_{m},  \bar A_{m-1})$-bimodule and an $(\bar A_{m-1},  \bar A_{m})$-bimodule. By
 (A3) and Proposition~\ref{bara}(1),  $\bar A_m$ is a free right $\bar A_{m-1}$-module with basis $\{\bar h\mid h \in D(m)\}$. Furthermore, we have  two exact functors  \begin{equation}\label{reinfu1}
\text{ind}^m_{m-1}=  \bar A_{m}\otimes _{ \bar A_{m-1} }?, \ \   \
 \text{res}^m_{m-1} = \bar A_{m}\otimes _{\bar A_{m} }?.\end{equation}
Let $\bar {A}^\circ=\bigoplus_{m\in \mathbb N} \bar A_{m}$. Define  $ \bar A_L^\circ=\bigoplus_{m, n\in \mathbb N} 1_m A_L^\circ 1_n$ and $\bar A_R^\circ=\bigoplus_{m, n\in \mathbb N} 1_m A_R^\circ 1_n$ such that
\begin{equation}\label{ws123} \bar 1_{ m} \bar A_L^\circ\bar 1_{ n}=\bar 1_{ {m+1}}\bar A^\circ \bar 1_{ n}, \text{ $\bar 1_{ m} \bar A_R^\circ\bar 1_{ n}=\bar 1_{ {m}}\bar A^\circ\bar 1_{ {n+1}} $.}\end{equation}
 Then both  $ \bar A_L^\circ$ and $\bar A_R^\circ$ are $(\bar A^\circ ,\bar A^\circ)$-bimodules such that
the right (resp., left) action of  $\bar A^\circ$  on $\bar A_L^\circ$ (resp., $\bar A_R^\circ$) is given by the usual multiplication, whereas the left (resp., right) action of $\bar A^\circ$ on   $\bar A_L^\circ$ (resp., $\bar A_R^\circ$) is given by the formulae similar to those in \eqref{act123}. Then
\begin{equation} \label{indes} \bar E=\bigoplus_{m\in\mathbb N} \text{res}^{m+1}_{m}, \text{   $\bar F=\bigoplus_{m\in\mathbb N} \text{ind}^{m+1}_{m} $},
\end{equation}
 where   $\bar E=\bar A_L^\circ\otimes _{\bar A^\circ} ?$ and   $\bar F= \bar A_R^\circ\otimes _{\bar A^\circ} ?$.
For any  admissible $i$, define
\begin{equation}\label{xitil}
   X_i 1_{ m}=1_{ m} X_i= \begin{tikzpicture}[baseline = 12.5, scale=0.35, color=\clr]
       \draw[-,thick](0,1.1)to[out= down,in=up](0,2.5);
       \draw(0.5,1.1) node{$ \cdots$}; \draw(0.5,2.5) node{$ \cdots$};
       \draw[-,thick](1.8,1.1)to (1.8,2.5);
       \draw(3,3)node{\tiny$i$}; \draw(5.5,3)node{\tiny$m$};
        \draw[-,thick] (3,1) to[out=up, in=down] (3,2.5);
         \draw[-,thick](4,1.1)to[out= down,in=up](4,2.5);\draw (3,1.8) node{$ \bullet$};
         \draw(4.5,1.1) node{$ \cdots$}; \draw(4.5,2.5) node{$ \cdots$};
         \draw[-,thick](5.5,1.1)to[out= down,in=up](5.5,2.5);
     \end{tikzpicture},  \end{equation}
where the dot is on the $i$th strand from the left.
  If we know $m$ from the context, we simply denote $ X_i1_{ m}$ by $ X_i$.
  Since $\mathcal A$ is a quotient category of  the strict monoidal category $\mathcal  C$, by interchange law in a strict monoidal category, $X_m$  centralizes the subalgebra $\tau_{A}(A_{m-1})$.
  Obviously, $$\pi_m\tau_A=\bar\tau_{m-1}\pi_{m-1},$$ where   $\pi_m$ is given in Proposition~\ref{bara}(4). Let $\bar X_m=\pi_m(X_m)$. Then    $\bar X_m$ centralizes the subalgebra $\bar\tau_{m-1}(\bar A_{m-1})$. Multiplying  $\bar X_m$ on the right (resp., left) of $\bar A_{m}$  gives an endomorphism of the $(\bar A_{m}, \bar  A_{m-1})$-bimodule (resp., $( \bar A_{m-1}, \bar  A_{m})$-bimodule) $\bar A_{m}$. For any $i\in \Bbbk$, define
  \begin{equation}\label{iinduc}
  i\text{-ind}^m_{m-1}= (\bar A_{m})_{ R, i}\otimes_{\bar A_{m-1} }?,\  \text{   $i\text{-res}^m_{m-1} = ( \bar A_{m})_{ L, i}\otimes _{\bar A_{m} }?$,}\end{equation}
  where $ (\bar A_{m})_{R, i}$ (resp.,  $ (\bar A_{m})_{L, i} $) is  the generalized $i$-eigenspace  of $\bar X_{m}$  with respect to the right (resp., left) action on  $\bar A_{m}$. Then
\begin{equation}\label{bra1} \text{ind}^m_{m-1}=\bigoplus_{i\in\Bbbk}i\text{-ind}^m_{m-1}, \quad \text{res}^m_{m-1}=\bigoplus_{i\in\Bbbk}i\text{-res}^m_{m-1}. \end{equation}
The $X_i1_m$ in \eqref{xitil} induces    $\bar X_L^\circ\in\End_\Bbbk(\bar A^\circ_L)$ and    $ \bar X_R^\circ\in\End_\Bbbk(\bar A_R^\circ)$ such that
  the restriction of $ \bar X_L^\circ $ (resp., $\bar X_R^\circ$)  to  $ \bar 1_{ {m+1}}\bar A^\circ$ (resp., $\bar A^\circ \bar 1_{ {m+1}}$ ) is given by  the left (resp., right)
multiplication  of  $\bar X_{m+1} \in \bar A_{m+1}$.
If $\mathcal C=\mathcal C_2$, then   $\bar X_R^\circ $ (resp.,  $\bar X_L^\circ$) is   invertible and the corresponding inverse is  given by   the right
(resp., left) multiplication  of  $ \bar {X_{ m+1}^{-1}}$  on $ \bar 1_{ {m+1}}\bar A^\circ$ (resp., $\bar A^\circ \bar 1_{ {m+1}}$), where  $$X_{m+1}^{-1}=\begin{tikzpicture}[baseline = 12.5, scale=0.35, color=\clr]
              \draw(5.5,3)node{\tiny$m+1$}; 
        \draw[-,thick] (3,1) to[out=up, in=down] (3,2.5);
         \draw[-,thick](4,1.1)to[out= down,in=up](4,2.5);\draw (5.5,1.8) node{$\circ$};
         \draw(4.5,1.1) node{$ \cdots$}; \draw(4.5,2.5) node{$ \cdots$};
         \draw[-,thick](5.5,1.1)to[out= down,in=up](5.5,2.5);
           \end{tikzpicture}.$$
Since $ \bar X_{m}$ centralizes the subalgebra $ \bar A_{m-1}$  in $ \bar A_m$,  we have the following  result, immediately.
\begin{Lemma}\label{bim123}   Both $\bar X_L^\circ$ and  $ \bar X_R^\circ$
  are $(\bar A^\circ,\bar A^\circ)$-bimodule homomorphisms.\end{Lemma}
  For $i\in  \Bbbk$, let $\bar E_i=\bar A_{L, i}^\circ\otimes _{\bar A^\circ} ?$   and   $\bar F_i= \bar A_{R, i}^\circ\otimes _{\bar A^\circ} ?$, where  $\bar A_{L, i}^\circ$ (resp.,  $\bar A_{R, i}^\circ$)   is the $(\bar A^\circ, \bar A^\circ)$-bimodule that is  the generalized $i$-eigenspace of $\bar X_L^\circ $ (resp., $\bar X_R^\circ $) on $\bar A_L^\circ$ (resp., $\bar A_R^\circ$).
Then \begin{equation}\label{filteration1}\bar E_{i}=\bigoplus _{m\in\mathbb N}i\text{-res}^{m+1}_m, \text{   }\bar F_{i}=\bigoplus _{m\in\mathbb N}i\text{-ind}^{m+1}_m.   \end{equation}

\begin{Assumption}In the remaining part of this section we keep the following assumption:
\begin{itemize}
\item [(A4)]For any $m\in\mathbb N$, let  $Y_1(m)=\tau_A (Y(m-1))$  and $Y_2(m)$ consists of all following diagrams
   \begin{equation}\begin{tikzpicture}[baseline = 1mm, color=\clr]
                \draw[-,thick] (0,-0.3)to[out=up,in=down](0,0); \draw(0.2, -0.2) node{$ \cdots$};  \draw[-,thick] (0.4,-0.3)to[out=up,in=down](0.4,0);  \draw[-,thick,darkblue] (0.5,0) to[out=down,in=left] (0.78,-0.28) to[out=right,in=down] (1.06,0);\draw (-0.1,0.4) rectangle (0.9,0);
                \draw(0.4, 0.2) node{$d$}; \draw[-,thick] (1.06,0.85)to[out=up,in=down](1.06,0.01);
                \draw (0.9,0.55) rectangle (-0.1,0.85); \draw[-,thick] (0, 0.4)to[out=up,in=down](0,0.55); \draw(0.2, 0.48) node{$ \cdots$};  \draw[-,thick] (0.5,0.4)to[out=up,in=down](0.5,0.55); \draw(0.4, 0.7) node{$f$};
    \end{tikzpicture}, \end{equation}
     for all $(f, d)\in Y(m+1)\times D(m+1)$. We assume that $\tilde Y(m):= Y_1(m)\overset{\cdot}\sqcup  Y_2(m)$ is a basis of $A^{-}1_m$.
     \end{itemize}
\end{Assumption}

 \begin{Lemma}\label{intailh1}
There is a short exact sequence of functors from $\bar A^\circ$-fdmod to $A$-lfdmod:
\begin{equation}\label{diejdhd1}
0\rightarrow \Delta\circ \bar E\rightarrow  E\circ \Delta\rightarrow\Delta\circ  \bar F\rightarrow0.
\end{equation}
\end{Lemma}
\begin{proof} At moment, we define $B_1\otimes B_2=\{b_1\otimes b_2\mid b_1\in B_1, b_2\in B_2\}$ for any  sets $B_1, B_2$.
When $B\in \{H(m), Y(m)\}$ and $b\in B$, let   $\bar b$ be the image of $b$ in $A_{\preceq m}$ and set $\bar B=\{\bar b\mid b\in B\}$.

  Thanks to Lemma~\ref{exactj}(1) and \eqref{ws123},   $ A_{\preceq m-1}\bar 1_{m-1}\otimes _{\bar A_{m-1}}\bar A^\circ_L\bar 1_ m$  has  basis $ \bar Y(m-1)\otimes \bar H(m)$ (if $m=0$, this set is $\emptyset$ by convention).  Similarly,  $ A_{\preceq m+1}\bar 1_ {m+1}\otimes _{\bar A_{m+1}}\bar A_R^\circ \bar 1_ m$ has basis $ \bar Y(m+1)\otimes \bar H(m+1)$.
Since $A_L\otimes _{A}A_{\preceq m}\cong A_{\preceq m}$ as vector spaces, $A_L\otimes _{A}A_{\preceq m}\bar 1_m$ has basis $\tilde Y(m)\otimes \bar H(m)$.
To prove \eqref{diejdhd1}, it suffices to show the following is a short exact sequence of $(A,\bar A^\circ)$-bimodules
\begin{equation}\label{shoretmor1}
0\rightarrow A_{\preceq m-1}\bar 1_{m-1}\otimes _{\bar A_{m-1}}\bar A_L^\circ\bar 1_ m\overset {\varphi}\rightarrow A_L\otimes _{A}A_{\preceq m}\bar 1_ m\overset{\psi}\rightarrow A_{\preceq m+1}\bar 1_ {m+1}\otimes _{ \bar A_{m+1}}\bar A_R^\circ\bar 1_ m\rightarrow0
\end{equation}
for any $ m\in \mathbb N $,  where
$$\varphi(\bar f\otimes \bar g)= ( f~\begin{tikzpicture}[baseline = 10pt, scale=0.5, color=\clr]
                \draw[-,thick] (0,0.5)to[out=up,in=down](0,1.2);
    \end{tikzpicture})\otimes \bar g, \ \ \text{    $\psi( f\otimes \bar g)=\bar{\begin{tikzpicture}[baseline = 1mm, color=\clr]
                \draw[-,thick] (0,0.5)to[out=up,in=down](0,0.8); \draw(0.2, 0.6) node{$ \cdots$};  \draw[-,thick] (0.4,0.5)to[out=up,in=down](0.4,.8);  \draw[-,thick,darkblue] (0.5,0.5) to[out=up,in=left] (0.78,0.78) to[out=right,in=up] (1.06,0.5);\draw (-0.1,0.5) rectangle (0.9,0);
                \draw(0.4, 0.25) node{$ f$}; \draw[-,thick] (1.06,0.5)to[out=up,in=down](1.06,0.01);
    \end{tikzpicture}}\otimes (  \bar{g~\begin{tikzpicture}[baseline = 10pt, scale=0.5, color=\clr]
                \draw[-,thick] (0,0.5)to[out=up,in=down](0,1.2);
    \end{tikzpicture}~})$ },$$
         for any admissible basis elements $ f, \bar f$ and $ \bar g$.
     First of all,  it is routine  to check that both  $\varphi $ and $\psi$  are well-defined $(A,\bar A^\circ)$-homomorphisms. Furthermore,  we have $\bar {f~\lcap}=0$   in $A_{\preceq m+1}$
  if $f\in Y(m-1)$. So,  for all  $\bar f\otimes \bar g\in \bar Y(m-1)\otimes \bar H(m)$,
 $$\psi(\varphi(\bar f\otimes \bar g))=\bar{f~\lcap}\otimes(\bar{g~~\begin{tikzpicture}[baseline = 10pt, scale=0.5, color=\clr]
                \draw[-,thick] (0,0.5)to[out=up,in=down](0,1.2);
    \end{tikzpicture}~})=0. $$
 It remains to prove that $\varphi$ gives a bijective map between   $ \bar Y(m-1)\otimes \bar H(m)$ and  $\tilde Y_1$ and
    the restriction of $\psi$ to $\tilde Y_2$ gives a bijective map between $\tilde Y_2$ and   $ \bar Y(m+1)\otimes \bar H(m+1)$, where   $\tilde Y_i=Y_i(m)\otimes \bar H(m)$, and $Y_1(m), Y_2(m)$ are given in (A4). If so, since $\tilde Y(m)\otimes \bar H(m)=\tilde Y_1\overset {.} \sqcup \tilde Y_2$, we have the exactness of  \eqref{shoretmor1}.

   Obviously,  $\varphi$ sends $ \bar Y(m-1)\otimes \bar H(m)$ to $\tilde Y_1$ and the inverse map $\varphi_1$ of  $\varphi$  sends  $\bar f\otimes \bar g\mapsto  \bar {f_1}\otimes \bar g$ for any $ \bar f\otimes \bar g \in \tilde Y_1$,
   where $f_1$ is obtained from $f$ by deleting the rightmost vertical strand.
   For any $f\in Y(m+1)$, $d\in D(m+1)$ and $g_1\in H(m)$, by  \eqref{relation 1}I,  we have
   \begin{equation}\label{jsijdd} \psi (\begin{tikzpicture}[baseline = 1mm, color=\clr]
                \draw[-,thick] (0,-0.3)to[out=up,in=down](0,0); \draw(0.2, -0.2) node{$ \cdots$};  \draw[-,thick] (0.4,-0.3)to[out=up,in=down](0.4,0);  \draw[-,thick,darkblue] (0.5,0) to[out=down,in=left] (0.78,-0.28) to[out=right,in=down] (1.06,0);\draw (-0.1,0.4) rectangle (0.9,0);
                \draw(0.4, 0.2) node{$d$}; \draw[-,thick] (1.06,0.85)to[out=up,in=down](1.06,0.01);
                \draw (0.9,0.55) rectangle (-0.1,0.85); \draw[-,thick] (0, 0.4)to[out=up,in=down](0,0.55); \draw(0.2, 0.48) node{$ \cdots$};  \draw[-,thick] (0.5,0.4)to[out=up,in=down](0.5,0.55); \draw(0.4, 0.7) node{$f$};
    \end{tikzpicture}\otimes \bar{g_1})=\bar{\begin{tikzpicture}[baseline = 1mm, color=\clr]
                \draw[-,thick] (0,-0.3)to[out=up,in=down](0,0); \draw(0.2, -0.2) node{$ \cdots$};  \draw[-,thick] (0.4,-0.3)to[out=up,in=down](0.4,0);  \draw[-,thick,darkblue] (0.5,0) to[out=down,in=left] (0.78,-0.28) to[out=right,in=down] (1.06,0);\draw (-0.1,0.4) rectangle (0.9,0);
                \draw(0.4, 0.2) node{$d$}; \draw[-,thick] (1.06,0.85)to[out=up,in=down](1.06,0.01);
                \draw (0.9,0.55) rectangle (-0.1,0.85); \draw[-,thick] (0, 0.4)to[out=up,in=down](0,0.55); \draw(0.2, 0.48) node{$ \cdots$};  \draw[-,thick] (0.5,0.4)to[out=up,in=down](0.5,0.55); \draw(0.4, 0.7) node{$f$};
                \draw[-,thick,darkblue] (0.5+0.56,0.5+0.35) to[out=up,in=left] (0.78+0.56,0.78+0.35) to[out=right,in=up] (1.06+0.56,0.5+0.35);
                \draw[-,thick] (1.06+0.56,0.5+0.35)to[out=up,in=down](1.06+0.56,0);
    \end{tikzpicture}}\otimes\bar {g_1~\begin{tikzpicture}[baseline = 10pt, scale=0.5, color=\clr]
                \draw[-,thick] (0,0.5)to[out=up,in=down](0,1.2);
    \end{tikzpicture}}=\bar f\otimes \bar{\begin{tikzpicture}[baseline = 3mm, color=\clr]
                \draw (-0.1,0.4) rectangle (0.9,0);
                \draw(0.4, 0.2) node{$ g_1$}; \draw[-,thick] (1.06,0.53)to[out=up,in=down](1.06,0.01);
                \draw (1.2,0.55) rectangle (-0.1,0.85); \draw[-,thick] (0, 0.4)to[out=up,in=down](0,0.55); \draw(0.25, 0.48) node{$ \cdots$};  \draw[-,thick] (0.6,0.4)to[out=up,in=down](0.6,0.55);
                \draw(0.6, 0.7) node{$d$};
    \end{tikzpicture}}=\bar f\otimes \bar g,   \end{equation}
    where $g=d\circ \tau_A(g_1)$. By (A3)
      $\psi$ sends  $\tilde Y_2$ to  $ \bar Y(m+1)\otimes \bar H(m+1)$.

      Suppose $g\in H(m+1)$. Thanks to (A3),
   $g=d\circ \tau_A(g_1)$ for some  $d\in D(m+1)$ and $ g_1\in H(m)$.
 For any   $\bar f\otimes \bar g\in \bar Y(m+1)\otimes \bar H(m+1)$, define
   $$\psi_1 (\bar f\otimes \bar g) = \begin{tikzpicture}[baseline = 1mm, color=\clr]
                \draw[-,thick] (0,-0.3)to[out=up,in=down](0,0); \draw(0.2, -0.2) node{$ \cdots$};  \draw[-,thick] (0.4,-0.3)to[out=up,in=down](0.4,0);  \draw[-,thick,darkblue] (0.5,0) to[out=down,in=left] (0.78,-0.28) to[out=right,in=down] (1.06,0);\draw (-0.1,0.4) rectangle (0.9,0);
                \draw(0.4, 0.2) node{$d$}; \draw[-,thick] (1.06,0.85)to[out=up,in=down](1.06,0.01);
                \draw (0.9,0.55) rectangle (-0.1,0.85); \draw[-,thick] (0, 0.4)to[out=up,in=down](0,0.55); \draw(0.2, 0.48) node{$ \cdots$};  \draw[-,thick] (0.5,0.4)to[out=up,in=down](0.5,0.55); \draw(0.4, 0.7) node{$f$};
    \end{tikzpicture} \otimes  \bar {g_1}.$$
   Thanks to  (A3)-(A4), $\psi_1 (\bar f\otimes \bar g)\in \tilde Y_2$. Moreover, by \eqref{jsijdd} $\psi\psi_1(\bar f\otimes \bar g)=\bar f\otimes \bar g$.
      Using \eqref{relation 1}I  yields
   $$\psi_1\psi(f\otimes \bar g)=\begin{tikzpicture}[baseline = 1mm, color=\clr]
                \draw[-,thick] (0,0.5)to[out=up,in=down](0,0.8); \draw(0.2, 0.6) node{$ \cdots$};  \draw[-,thick] (0.4,0.5)to[out=up,in=down](0.4,.8);  \draw[-,thick,darkblue] (0.5,0.5) to[out=up,in=left] (0.78,0.78) to[out=right,in=up] (1.06,0.5);\draw (-0.1,0.5) rectangle (0.9,0);
                \draw(0.4, 0.25) node{$f$}; \draw[-,thick] (1.06,0.5)to[out=up,in=down](1.06,0);\draw[-,thick,darkblue] (1.06,0) to[out=down,in=left] (0.28+1.06,-0.28) to[out=right,in=down] (1.06-0.5+1.06,0);\draw[-,thick] (1.06-0.5+1.06,0.8)to[out=up,in=down](1.06-0.5+1.06,0);
    \end{tikzpicture}\otimes \bar g=f\otimes \bar g.
   $$
 So,  $\psi_1$ is the inverse map  of  $\psi$. This completes  the proof.
\end{proof}
The $X_i1_m$ in \eqref{xitil} also  induces a linear map $X\in\End_\Bbbk(A_L)$ such that the restriction of $ X$  to $ 1_{ {m+1}}A$ is given by  the left multiplication of $X_{m+1}1_{{m+1}}$.
\begin{Lemma}\label{msikdjde1} The endomorphism $X$ is an $(A,A)$-bimodule homomorphism.  If  $\varphi$ and $\psi$  are given  in \eqref{shoretmor1}, then
\begin{itemize}
\item[(1)] $( X\otimes \text{Id})\circ \varphi=\varphi \circ (\text{Id}\otimes\bar X_L^\circ )$,
 \item [(2)] $(\text{Id}\otimes  (\bar X_R^\circ )^{-1} )\circ \psi=\psi\circ ( X\otimes \text{Id})$ if $\mathcal C=\mathcal C_2$,
\item [(3)] $(\text{Id}\otimes \bar X_R^\circ )\circ \psi=-\psi\circ ( X\otimes \text{Id})$ if $\mathcal C=\mathcal C_1$.
     \end{itemize}
\end{Lemma}
\begin{proof}For any $g\in \Hom_{\mathcal C}( m,  n)$,
\begin{equation} \label{act321} X_{n+1}1_{ {n+1}}\circ (g~\begin{tikzpicture}[baseline = 10pt, scale=0.5, color=\clr]
                \draw[-,thick] (0,0.5)to[out=up,in=down](0,1.2);
    \end{tikzpicture})=g~\begin{tikzpicture}[baseline = 10pt, scale=0.5, color=\clr]
                \draw[-,thick] (0,0.5)to[out=up,in=down](0,1.2);\draw (0,0.8) node{$ \bullet$};
    \end{tikzpicture}=(g~\begin{tikzpicture}[baseline = 10pt, scale=0.5, color=\clr]
                \draw[-,thick] (0,0.5)to[out=up,in=down](0,1.2);
    \end{tikzpicture}) \circ  X_{m+1}1_{ {m+1}}.\end{equation} So,  $X$ is an $(A,A)$-homomorphism.
Suppose $ (f, g)\in Y(m-1)\times H(m)$ and  $m\geq 1$. By  \eqref{act321},  $$( X\otimes \text{Id})\circ \varphi(\bar f\otimes \bar g)=( X\otimes \text{Id})(( f~\begin{tikzpicture}[baseline = 10pt, scale=0.5, color=\clr]
                \draw[-,thick] (0,0.5)to[out=up,in=down](0,1.2);
    \end{tikzpicture})\otimes \bar g)=(f~\begin{tikzpicture}[baseline = 10pt, scale=0.5, color=\clr]
                \draw[-,thick] (0,0.5)to[out=up,in=down](0,1.2);\draw (0,0.8) node{$ \bullet$};
    \end{tikzpicture} )\otimes \bar g=f\otimes  \bar X_{m+1}\bar g =\varphi \circ (\text{Id}\otimes  \bar X_L^\circ )(\bar f\otimes \bar g)
 ,$$ proving (1). If $ (f, g)\in Y(m)\times  H(m)$, then
    $$ \begin{aligned}
    &(\text{Id}\otimes  (\bar X_R^\circ)^{-1} )\circ \psi(f\otimes \bar g)=(\text{Id}\otimes (\bar X_R^\circ)^{-1} )( \bar{\begin{tikzpicture}[baseline = 1mm, color=\clr]
                \draw[-,thick] (0,0.5)to[out=up,in=down](0,0.8); \draw(0.2, 0.6) node{$ \cdots$};  \draw[-,thick] (0.4,0.5)to[out=up,in=down](0.4,.8);  \draw[-,thick,darkblue] (0.5,0.5) to[out=up,in=left] (0.78,0.78) to[out=right,in=up] (1.06,0.5);\draw (-0.1,0.5) rectangle (0.9,0);
                \draw(0.4, 0.25) node{$f$}; \draw[-,thick] (1.06,0.5)to[out=up,in=down](1.06,0.01);
    \end{tikzpicture}}\otimes ( \bar{g~~\begin{tikzpicture}[baseline = 10pt, scale=0.5, color=\clr]
                \draw[-,thick] (0,0.5)to[out=up,in=down](0,1.2);
    \end{tikzpicture}~}))=\bar{\begin{tikzpicture}[baseline = 1mm, color=\clr]
                \draw[-,thick] (0,0.5)to[out=up,in=down](0,0.8); \draw(0.2, 0.6) node{$ \cdots$};  \draw[-,thick] (0.4,0.5)to[out=up,in=down](0.4,.8);  \draw[-,thick,darkblue] (0.5,0.5) to[out=up,in=left] (0.78,0.78) to[out=right,in=up] (1.06,0.5);\draw (-0.1,0.5) rectangle (0.9,0);
                \draw(0.4, 0.25) node{$f$}; \draw[-,thick] (1.06,0.5)to[out=up,in=down](1.06,0.01);
    \end{tikzpicture}}\otimes ( \bar {g~\begin{tikzpicture}[baseline = 10pt, scale=0.5, color=\clr]
                \draw[-,thick] (0,0.5)to[out=up,in=down](0,1.2);\draw (0,0.8) node{$ \circ$};
    \end{tikzpicture}}) \\&
    = \bar{\begin{tikzpicture}[baseline = 1mm, color=\clr]
                \draw[-,thick] (0,0.5)to[out=up,in=down](0,0.8); \draw(0.2, 0.6) node{$ \cdots$};  \draw[-,thick] (0.4,0.5)to[out=up,in=down](0.4,.8);  \draw[-,thick,darkblue] (0.5,0.5) to[out=up,in=left] (0.78,0.78) to[out=right,in=up] (1.06,0.5);\draw (-0.1,0.5) rectangle (0.9,0);
                \draw(0.4, 0.25) node{$f$}; \draw[-,thick] (1.06,0.5)to[out=up,in=down](1.06,0.01); \draw (1.06,0.3) node{$ \circ$};
    \end{tikzpicture}}\otimes ( \bar{g~~\begin{tikzpicture}[baseline = 10pt, scale=0.5, color=\clr]
                \draw[-,thick] (0,0.5)to[out=up,in=down](0,1.2);
    \end{tikzpicture}~})
    \overset{\text{\eqref{relation 3}II}}=\bar{ \begin{tikzpicture}[baseline = 1mm, color=\clr]
                \draw[-,thick] (0,0.5)to[out=up,in=down](0,0.8); \draw(0.2, 0.6) node{$ \cdots$};  \draw[-,thick] (0.4,0.5)to[out=up,in=down](0.4,.8);  \draw[-,thick,darkblue] (0.5,0.5) to[out=up,in=left] (0.78,0.78) to[out=right,in=up] (1.06,0.5);\draw (-0.1,0.5) rectangle (0.9,0);
                \draw(0.4, 0.25) node{$f$}; \draw[-,thick] (1.06,0.5)to[out=up,in=down](1.06,0.01); \draw (0.55,0.6) node{$ \bullet$};
    \end{tikzpicture}}\otimes ( \bar{g~~\begin{tikzpicture}[baseline = 10pt, scale=0.5, color=\clr]
                \draw[-,thick] (0,0.5)to[out=up,in=down](0,1.2);
    \end{tikzpicture}~})
    =\psi\circ ( X\otimes \text{Id}) (f\otimes  \bar g),
    \end{aligned} $$
   proving (2). Replacing  \eqref{relation 3}II by  \eqref{relation 1}II, one can verify (3) by arguments similar to those for (2). We omit details.
\end{proof}

\begin{Defn}\label{sharpind} If $\mathcal C=\mathcal C_1$, we define  $\sharp:\Bbbk\rightarrow \Bbbk$ such that $i^\sharp= -i$ for any $i\in \Bbbk$.  If $\mathcal C=\mathcal C_2$, we define   $\sharp:\Bbbk^\times \rightarrow \Bbbk^\times $ such that
 $i^\sharp=i^{-1}$ for any $i\in \Bbbk^\times$.\end{Defn}
For each $i\in \Bbbk$, let  \begin{equation}\label{ei}  E_i:= A_{L, i}\otimes _{A}?,\end{equation} where $A_{L,i}$ is the generalized $i$-eigenspace of $ X$ on $A_L$.
\begin{Theorem}\label{usuactifuc1} For each $i\in \Bbbk$,
   $ E_i$ is an exact functor. Furthermore, there is a short exact sequence of functors  from $\bar A^\circ$-fdmod to $A$-lfdmod:
 \begin{equation}\label{keyses} 0\rightarrow \Delta\circ \bar E_i\rightarrow  E_i\circ \Delta\rightarrow\Delta\circ  \bar F_{i^\sharp}\rightarrow0.\end{equation}
\end{Theorem}

\begin{proof}Since  $ E_i$ is a summand  of the exact functor $ E$, so is $E_i$. Thanks to  Lemma~\ref{msikdjde1}, the short exact sequence of \eqref{keyses} follows from \eqref{diejdhd1} by
 passing to appropriate generalized eigenspaces.\end{proof}

\subsection{Characters}Fix an $m\in \mathbb N$.  Then     $\{ X_i1_{ m}\mid 1\le i\le m\}$ generates a finite dimensional  commutative subalgebra of  $ A_m$, where
 $ X_i1_{ m}$'s are given in \eqref{xitil}.
 For any $\mathbf i=(i_1,i_2,\ldots, i_m)\in \Bbbk^m$, there is an idempotent $1_{\mathbf i}\in  A_m$ which projects any $ A_m$-module $M$
onto $M_{\mathbf i}$,  the simultaneous generalized eigenspace of $ X_11_{ m} ,\ldots, X_{ m}1_{ m}$ with respect to $\mathbf i$.

\begin{Assumption}\label{res} In the remaining part of this section, we keep the following assumptions.

\begin{itemize}
\item [(A5)] There   is
 a $\Bbbk$-linear  monoidal contravariant functor $\sigma: \mathcal C\rightarrow \mathcal C$ such that   $\sigma^2=\text{Id}$ and $\sigma$ preserves the generating object \begin{tikzpicture}[baseline = 1.5mm]
	\draw[-,thick,darkblue] (0.18,0) to (0.18,.4);
\end{tikzpicture} and  the right ideal $K$. Hence $\sigma$ induces an anti-involution $\sigma_{A}:A\rightarrow A$ with  $\sigma_{A}(1_{ m})=1_ m$ for any $ m\in\mathbb N$.
 \item[(A6)] For any $m\in\mathbb N$,    $\bar A_m$ is a cellular algebra with cellular basis given in \eqref{ceboaa}.  The corresponding  cell module is  $S_m(\lambda)$ (see \eqref{basisofcell}), $\forall \lambda\in \Lambda_m$.
\item[(A7)] For any $i\in \Bbbk$, $m\in \mathbb N$ and any  $\lambda\in \Lambda_m$,
there are   $\mathscr A_{i,\lambda}\subseteq \Lambda_{m+1}$ and   $\mathscr R_{i,\lambda}\subseteq \Lambda_{m-1}$ such that $\bar E_i S_m(\lambda)$ (resp., $\bar F_i S_m(\lambda)$) has a multiplicity-free filtration with sections $S_{m-1}(\mu)$ (resp., $S_{m+1}(\mu)$) for   $ \mu\in\mathscr R_{i,\lambda}$ (resp., $ \mu\in\mathscr A_{i,\lambda}$).
\item [(A8)] $\mu\in\mathscr A_{i,\lambda} $ if and only if $\lambda\in\mathscr R_{i,\mu}$, for any $i\in\Bbbk$ and  $\lambda,\mu\in\Lambda=\bigcup_{m\in \mathbb N} \Lambda_m$.
\item[(A9)]$\bar A_0=\Bbbk\bar 1_0\cong\Bbbk$.
\end{itemize}\end{Assumption}
 Recall that $\bar\Lambda=\bigcup_{m\in\mathbb N}\bar\Lambda_m$, where $\bar \Lambda_m$ parameterizes all the pairwise inequivalent simple  $\bar A_m$-modules in the sense that  the cell module $S_m(\lambda)$ has the simple head $L_m(\lambda)$ for  $\lambda\in \bar\Lambda_m$.    So, $\Lambda_0=\bar \Lambda_0$,   the set containing  a single element,  say  $\emptyset$.  To simplify the notation, we denote $S_m(\lambda)$ by $S(\lambda)$.

 Let  $\Xi$ be  the graph such that the set of  vertices is  $\Lambda$ and  any edge is of form   $\lambda$---$\mu$
whenever  $ \mu\in\mathscr A_{i,\lambda}$ or $ \mu\in\mathscr R_{i,\lambda}$ (equivalently by (A8),  $ \lambda\in\mathscr R_{i,\mu}$ or $ \lambda\in\mathscr A_{i,\mu}$)
 for some $i\in\Bbbk$.

For any $\lambda, \mu\in \Lambda$,
a  path $\gamma:\lambda \rightsquigarrow\mu$ is of length $\ell(\gamma)$ if  there are $\lambda_i\in \Lambda$, $0\le i\le \ell(\gamma)$  such that    $\lambda_0=\lambda$, $\lambda_{\ell(\gamma)}=\mu$ and $\lambda_{j-1}$---$\lambda_{j}$ for $1\leq j\leq \ell(\gamma)$.  Furthermore, the edge  $\lambda_{j-1}$---$\lambda_{j}$ is
  colored by  $i_j\in \Bbbk$ such that $\lambda_j\in\mathscr A_{i_j,\lambda_{j-1} }$ or  $\lambda_j\in\mathscr R_{i_j^\sharp,\lambda_{j-1} }$.
    In this case, we say that   $\gamma$ is of content  $c(\gamma)=( i_1, i_2, \ldots, i_{\ell(\gamma)})$.
When $\ell(\gamma)=0$,  there is a unique path from $\emptyset$ to $\emptyset$ with length $0$ and the  corresponding content  is  $ \emptyset$.

\begin{Defn} For any $V\in A$-lfdmod, define
\begin{equation}
\text{ch}V=\sum_{m\in\mathbb N, \mathbf i\in \Bbbk^m}(\dim 1_{\mathbf i}V)e^{\mathbf i},
\end{equation}
where $1_{\mathbf i}V$ is the   simultaneous generalized eigenspace of $ X_11_{ m} ,\ldots, X_m 1_{ m}$ corresponding to $\mathbf i$.
\end{Defn}

\begin{Prop}\label{charcter1}
For any $\lambda\in\Lambda$,
$\text{ch}\tilde\Delta(\lambda)=\sum_{\gamma}e^{c(\gamma)}$
where the summation ranges over all paths $\gamma: \emptyset \rightsquigarrow \lambda$ and $\tilde\Delta(\lambda)$ is given in \eqref{tildelta}.
\end{Prop}
\begin{proof}It suffices to show that $\dim 1_{\mathbf i}\tilde \Delta(\lambda)=| \{\gamma:\emptyset \rightsquigarrow\lambda\mid c(\gamma)=\mathbf i\}|$ for all $\mathbf  i\in  \Bbbk^m$ (resp., $(\Bbbk^\times)^m$) and all $m\in \mathbb N$ if $\mathcal C=\mathcal C_1$ (resp., $\mathcal C_2$).

Suppose $\lambda\in\Lambda_n$. By Lemma~\ref{exactj}(1),   $\tilde\Delta(\lambda)$ has a basis $\{\bar y\otimes w\mid y\in Y(n), w\in  \mathcal S(\lambda)\}$,
where  $\mathcal S(\lambda)$ is a basis of $S(\lambda)$. So,
\begin{equation}\label{zmmzzm1}
\tilde\Delta(\lambda)=\bigoplus_{d\in\mathbb N}1_{ {n+2d}}\tilde\Delta(\lambda), \ \ \text{  } 1_{ n} \tilde\Delta(\lambda)= S(\lambda).
\end{equation}
We prove the result by induction on $m\in \mathbb N $. Suppose that  $m=0$. If $\lambda\neq\emptyset$, then $1_\emptyset \tilde\Delta(\lambda)=0$. Otherwise,   $1_\emptyset \tilde\Delta(\lambda)=\Bbbk$. So, the result holds for $m=0$.

Suppose that $m>0$.
First, we have
 $1_{ {m-1}}E \tilde \Delta(\lambda)=1_{ m}A\otimes_A \tilde \Delta(\lambda)$,  and an $A_{{m-1}}$-isomorphism
$\phi: 1_{ m}A\otimes_A \tilde \Delta(\lambda)\rightarrow 1_{ m} \tilde \Delta(\lambda)$, sending $b\otimes v$ to  $ bv$
  for all $(b, v)\in 1_{ m}A \times   \tilde \Delta(\lambda)$.
Note that $1_{ {m-1}} E_{i} \tilde \Delta(\lambda)=1_{ {m-1}} (A_L)_i\otimes _A \tilde \Delta(\lambda)= (1_{ m} A)_i\otimes _A \tilde \Delta(\lambda)$, where
$ (1_{ m} A)_i$ is the generalized $i$-eigenspace of $X_m1_{ m}$ with respect to the left action on $1_{ m}A$.
Restricting  $\phi$ to $1_{ {m-1}} E_{i} \tilde \Delta(\lambda) $ gives an isomorphism
between $ 1_{ {m-1}}E_{i} \tilde \Delta(\lambda)$ and the
generalized $i$-eigenspace of $ X_m1_{ m}$ on $1_{ m} \tilde \Delta(\lambda)$ and hence
\begin{equation}\label{cdcedccde1}
\dim1_{\mathbf i}\tilde\Delta(\lambda)=\dim1_{\mathbf i'} E_{i_m}\tilde\Delta(\lambda),
\end{equation}
for  any $\mathbf i\in \Bbbk^m$, where $\mathbf i'=(i_1,\ldots,i_{m-1})$.
Thanks to \eqref{keyses} and (A7)-(A8),  $ E_{i_m}\tilde \Delta(\lambda)$ has a multiplicity-free
$\tilde\Delta$-filtration such that $\tilde\Delta(\mu)$ appears as a section if and only if either
 $ \mu\in\mathscr R_{i_m,\lambda}$ or $ \mu\in\mathscr A_{i_m^\sharp ,\lambda}$.
Now the result follows from \eqref{cdcedccde1} and induction on $m$.
\end{proof}

\begin{Cor}\label{twopath1}
Suppose  $\lambda\in\Lambda, \mu\in\bar \Lambda_m$.  If $[\tilde \Delta(\lambda):L(\mu)]\neq0$, then there are two  paths $\gamma:\emptyset \rightsquigarrow \lambda$
and $\delta: \emptyset\rightsquigarrow\mu$  of length $m$ such that  $c(\gamma)=c(\delta)$.
\end{Cor}
\begin{proof}Thanks to \eqref{kkk1234},  $1_{m}L(\mu)=L_m(\mu)$ for any  $\mu\in \bar\Lambda_m$.  Pick an $ \mathbf i\in \Bbbk^m$ such that $1_{\mathbf i}L(\mu)\neq0$. Since $[\tilde \Delta(\lambda):L(\mu)]\neq0$,  we have $1_{\mathbf i}\tilde\Delta(\lambda)\neq0$.

Thanks to Lemma~\ref{barcom}(1), $[\bar \Delta(\mu): L(\mu)]=1$. Since $\Delta$ is exact and $L_m(\mu)$ is the simple head of $S(\mu)$, $[\tilde \Delta(\mu):L(\mu)]\neq0$, forcing    $1_{\mathbf i}\tilde\Delta(\mu)\neq0$. Now, the result follows from Proposition~\ref{charcter1}, immediately.
\end{proof}

\begin{Defn}\label{deofi0} Let $\mathbb I=\mathbb I_0\cup \mathbb I_0^{\sharp}$, where  $\mathbb I_0=\{i\in \Bbbk\mid \bar E_i\neq0\}$ and
$\mathbb I_0^{\sharp}=\{i^\sharp\mid i\in \mathbb I_0\}$.
\end{Defn}

\begin{Theorem}\label{sjdsdh} Keep the Assumption~\ref{pdual}.
If  $\mathbb I_0\cap\mathbb I_0^\sharp=\emptyset$, then   \begin{itemize}\item [(1)] $P(\mu)=\Delta(\mu)$ for all  $\mu\in \bar\Lambda$,
\item [(2)]  $\Delta: \bar A^\circ\text{-mod}\rightarrow A\text{-mod} $ is an equivalence of categories. \end{itemize}
\end{Theorem}
\begin{proof} Take an arbitrary $\mu\in \bar \Lambda_m$. If  $P(\mu)\neq \Delta(\mu)$, by  Lemma~\ref{isodual1}(2) (the place we have to keep the Assumption~\ref{pdual}) and   Proposition~\ref{dektss1},   $[\bar\Delta(\lambda):L(\mu)]\neq 0$ for some $ \lambda\in \bar\Lambda_n$ and $\lambda\neq \mu$.
Since $\Delta$ is exact and $L_n(\lambda)$ is the simple head of $S(\lambda)$,   there is an epimorphism from $\tilde \Delta(\lambda)$ to $\bar \Delta(\lambda)$. So,  $[\tilde\Delta(\lambda):L(\mu)]\neq 0$.
By Corollary~\ref{twopath1}, there are    two paths   $\gamma:\emptyset \rightsquigarrow \lambda$  and $\delta: \emptyset\rightsquigarrow\mu$  such that
$c(\gamma)=c(\delta)\in \Bbbk^m$.

We claim     $n=m$. Otherwise, $n+2s=m$ for some $s\in \mathbb N\backslash\{0\}$. Therefore, there is an edge, say $\lambda_{j-1}$---$\lambda_j$ such that
$\lambda_j \in \mathscr R_{i,\lambda_{j-1}}$ for some $i\in \Bbbk$. So,  $\mathbb I_0\cap\mathbb I_0^\sharp\neq\emptyset$, a contradiction.
Thanks to  Lemma~\ref{exactj}(1), $1_{ m}\bar\Delta(\lambda)=L_m(\lambda)$. Since $1_m L(\mu)=L_m(\mu)$, it is a composition factor of the simple  $\bar A_m$-module $L_m(\lambda)$, forcing $\lambda=\mu$, a contradiction.  So, $P(\mu)=\Delta(\mu)$ for all  $\mu\in \bar\Lambda$.

  By (1),  the exact functor $\Delta$ sends  projective $\bar A^\circ$-modules $P_m(\lambda)$'s  to  projective $A$-modules $P(\lambda)$'s for any $\lambda\in \bar\Lambda_m, m\in\mathbb N$. Note that
$\{P_m(\lambda)\mid \lambda\in \bar\Lambda_m, m\in\mathbb N \}$ gives a complete set of representatives of all indecomposable objects of  $  \bar A^\circ\text{-pmod} $.
 Thanks to the general  result on locally unital  algebras in  \cite[Corollary~2.5]{BRUNDAN},  $\Delta$ is an equivalence of categories, proving (2).
\end{proof}

\subsection{Categorical actions on $A\text{-mod}^\Delta$} Let $K_0(\bar A^\circ\text{-pmod})$ be the split   Grothendieck group of $\bar  A^{\circ}\text{-pmod}$, and
 let $K_0(A \text{-mod}^\Delta)$ be the Grothendieck group of $A\text{-mod}^\Delta$. Define
\begin{equation}\label{kgroup}[\bar A^\circ\text{-pmod} ]=\mathbb C\otimes_{\mathbb Z}K_0(\bar A^\circ\text{-pmod}), \ \   [ A\text{-mod}^\Delta ]=\mathbb C\otimes_{\mathbb Z}K_0(A\text{-mod}^\Delta).\end{equation}
Let $\mathfrak g$ be a    Kac-Moody Lie algebra (over the complex field $\mathbb C$) with Cartan subalgebra $\mathfrak h$. Its  Chevalley generators are  $\{e_i,f_i\mid i\in \mathbb I_0\}$ subject to
the Serre relations.
We say that there is a  categorical $\mathfrak g$-action on $\bar A^\circ$-pmod
if $\bar E_i$ and $\bar F_i$ are biadjoint with each other and   $[\text{$\bar A^\circ$-pmod}]$ is a $\mathfrak g$-module on which the Chevalley generators $e_i$ and $f_i$ act via  the endomorphisms induced by the  exact functors  $\bar E_i$ and $\bar F_i$  in \eqref{filteration1} for any $i\in\mathbb I_0$.

\begin{Theorem}\label{gsharpiso}  Let $\mathfrak g^\sharp$ be the Lie subalgebra of $\mathfrak g$ generated by $\{\tilde e_i\mid i\in \mathbb I\}$, where $\mathbb I$ is given in Definition~\ref{deofi0}, $\tilde e_i=e_i+f_{i^\sharp}$ and
$e_i=0$ (resp.,  $f_{i^\sharp}=0$)  if $i\notin \mathbb I_0$ (resp.,  $i^\sharp\notin \mathbb I_0$). If there is a  categorical $\mathfrak g$-action on $\bar A^\circ$-pmod, then $[A\text{-mod}^\Delta]$ is a left  $\mathfrak g^\sharp$-module on which   $\tilde e_i$ acts via the exact functor $E_i$ in Theorem~\ref{usuactifuc1}  for any $i\in \mathbb I$.
\end{Theorem}
\begin{proof}We are assuming that $A$ satisfies
 assumptions (A1)-(A9),   all previous results  in this section are available.
The exact functor $\Delta$ gives a linear isomorphism \begin{equation}\label{iso1} \Delta: [\bar A^\circ\text{-pmod} ]\rightarrow [A\text{-mod}^\Delta ]\end{equation}  which sends $[P_m(\lambda)]$ to $[\Delta(\lambda)]$ for all $\lambda\in \bar\Lambda_m$ and $m\in\mathbb N$.
If we define  $$e_i[\Delta(\lambda)]=[\Delta(\bar E_i P_m(\lambda))],  \text{
   $f_i[\Delta(\lambda)]=[\Delta(\bar F_i P_m(\lambda))]$},$$  for all $  \lambda\in\bar \Lambda_m, m\in\mathbb N, i\in\mathbb I_0$,
 then  $ [A\text{-mod}^\Delta ]$ is also a left  $\mathfrak g$-module, and hence  the linear isomorphism   $\Delta$ in \eqref{iso1} is  a $\mathfrak g$-isomorphism. In particular, it is  a $\mathfrak g^\sharp$-isomorphism.  For any $i\in \mathbb I$,
 by Theorem~\ref{usuactifuc1},
 $$ [{ E_i} \Delta(\lambda)]  =[\Delta(\bar E_i P_m(\lambda))]+[\Delta(\bar F_{i^\sharp} P_m(\lambda))] =\tilde e_i([\Delta(\lambda)]).
$$
\end{proof}
\begin{rem}Since $E$ is self-adjoint (see Lemma~\ref{bimoudisomah}), it sends projective module to projective module  and so does $E_i$. This  makes    $[A\text{-pmod}]$ as a $\mathfrak g^\sharp$-submodule of  $ [A\text{-mod}^\Delta]$ (see Proposition~\ref{dektss1}(1)), i.e., there is also a categorical $\mathfrak g^\sharp$-action on $A\text{-pmod}$. Moreover, by Proposition~\ref{dektss1}(2), $[A\text{-pmod}]\cong [A\text{-mod}^\Delta]$ as $\mathfrak g^\sharp$-modules.
\end{rem}

\subsection{Block decomposition} Let  $\mathfrak g$ be the Kac-Moody algebra in  subsection~5.4. Fix simple roots   $\Pi=\{\alpha_i\mid i\in \mathbb I_0\}$. The  weight lattice is
$$P= \{\omega\in \mathfrak h^*\mid \langle h_i,\omega \rangle\in \mathbb Z\text{ for all }i\in \mathbb I_0\},$$
 where $h_i=[e_i,f_i]$.       For any $\lambda, \mu\in P$, write \begin{equation}\label{prec}  \lambda\preceq \mu \ \ \text{ if
  $\bar \lambda=\bar \mu $  and $ \lambda\leq \mu$, } \end{equation}
where $\bar \lambda$ is  the image of $\lambda$ in  $P/Q^\sharp$,
$Q^\sharp=\sum_{i\in\mathbb I_0 \cap \mathbb I_0^\sharp}  \mathbb Z(\alpha_i+\alpha_{i^\sharp}) $ and   $\leq$  is the usual dominance  order on $P$ such that $\lambda\leq \mu$ if $\mu-\lambda\in \mathbb N \Pi$.
Then $\preceq$ is another  partial order  on $P$. In Proposition~\ref{dedji} and Theorem~\ref{hlink},
we keep the following assumption:
 \begin{itemize}\item[(A10)] For any $\mu\in \Lambda_m$ and $m\in\mathbb N$, any two paths  $\gamma,\delta:\emptyset \rightsquigarrow \mu $ of length $m$ have the same contents up to a permutation.
 \end{itemize}
  For any  $\mu\in \Lambda_m$ and any  path $\gamma :\emptyset \rightsquigarrow \mu $ of length $m$ , define  \begin{equation} \label{wt1123}
\text{wt}(\mu)= -\sum_{1\leq j\leq m}\alpha_{i_j}\in P,\end{equation} where $(i_1,\ldots, i_m)=c(\gamma)$.
By (A10), $\text{wt}(\mu)$ in \eqref{wt1123} is independent of  a path.

\begin{Prop}\label{dedji}
Suppose  $\lambda\in \Lambda$ and $\mu\in\bar \Lambda$. If
$[\tilde \Delta(\lambda):L(\mu)] \neq 0$, then   $\text{wt}(\mu) \preceq \text{wt}(\lambda)$.
\end{Prop}

\begin{proof} Suppose $\gamma:\emptyset\rightsquigarrow \lambda$ and  $\delta: \emptyset \rightsquigarrow\mu$ are two paths  such that  $c(\gamma)=c(\delta)=(i_1,\ldots, i_m)$.
We  claim   $\text{wt}(\mu) \preceq \text{wt}(\lambda)$
 if $\lambda\in \Lambda$ and $\mu\in \Lambda_m$.

 In fact, the claim is trivial if $m=0$. In this case,   $\lambda=\mu=\emptyset$. Suppose  $m>0$. Removing the last edge in  both $\gamma$  and $\delta$ yields two shorter paths $\gamma': \emptyset\rightsquigarrow \lambda'$
and $\delta': \emptyset \rightsquigarrow \mu'$ such that $c(\gamma')=c(\delta')$ and $\mu'\in  \Lambda_{m-1}$.
By induction assumption on $m-1$, $\text{wt} (\mu')\preceq \text{wt}(\lambda')$.

We have either $\lambda\in \mathscr A_{i_m,\lambda' }$ or  $\lambda\in \mathscr R_{i^\sharp_m,\lambda' }$.
In the first case,  thanks to ~\eqref{wt1123}, $\text{wt}(\lambda')=\text{wt}(\lambda)+\alpha_{i_m}$. Since $\mu\in  \Lambda_m$, $\text{wt}(\mu')=\text{wt}(\mu)+\alpha_{i_m}$.
So,   $\text{wt}(\mu) \preceq \text{wt}(\lambda)$. In the second case,
 $\lambda\in \mathscr R_{i_m^\sharp,\lambda' }$ (hence $i_m\in \mathbb I_0\cap \mathbb I_0^\sharp$), and
 $\text{wt}(\lambda')=\text{wt}(\lambda)-\alpha_{i^\sharp_m}$ and $\text{wt}(\mu')=\text{wt}(\mu)+\alpha_{i_m}$.
 So,
 $$\text{wt}(\lambda)-\text{wt}(\mu)=\text{wt}(\lambda')-\text{wt}(\mu')+\alpha_{i_m}+\alpha_{i_m^\sharp}\in \mathbb N\Pi,\ \  \overline{\text{wt}(\lambda)}=\overline{\text{wt}(\mu)},$$
 and hence   $\text{wt}(\mu) \preceq \text{wt}(\lambda)$. Now, our result follows immediately from Corollary ~\ref{twopath1} and the above claim.
 \end{proof}
\begin{Theorem}\label{hlink}Keep the Assumption~\ref{pdual}.
Suppose  $\lambda,\mu\in\bar \Lambda$. If  $L(\lambda)$ and $L(\mu)$ are in the  same block,  then $\bar{\text{wt}(\lambda)}=\bar{\text{wt}(\mu)}$.
\end{Theorem}
\begin{proof} Suppose $ [P(\mu):L(\lambda)]\neq0$.
By Corollary~\ref{ijxxexeu},
$P(\mu)$ has  a $\tilde \Delta$-flag and
\begin{equation} \label{b3}[P(\mu):L(\lambda)]=\sum_{\nu\in  \Lambda} [\tilde \Delta(\nu):L(\mu)] [\tilde \Delta(\nu): L(\lambda)]. \end{equation}
So there is at least  $\nu\in\Lambda$
such that $ [\tilde \Delta(\nu):L(\mu)] [\tilde \Delta(\nu): L(\lambda)]\neq 0$.
Thanks to  Proposition~\ref {dedji},  $\text{wt}(\lambda)\preceq \text{wt}(\nu)$ and $\text{wt}(\mu)\preceq \text{wt}(\nu)$, forcing
 $\bar{\text{wt}(\lambda)}=\bar{\text{wt}(\nu)}=\bar{\text{wt}(\mu)}$.
\end{proof}

\begin{rem}Suppose that the  blocks of $\bar A_m$ are given by $\text{wt}$ in the sense that  $L_m(\lambda)\sim L_m(\mu)$ if and only if $\text{wt}(\lambda)=\text{wt}(\mu)$. Then the stratification of the upper finite stratified category $A$-lfdmod in Theorem~\ref{newstratification} can be refined such that
$\rho$ is replaced by the weight function $\text{wt}:\bar \Lambda\rightarrow P$ with the order $\preceq$ on $P$ given in \eqref{prec}. See \cite{RS4} for Brauer category. In the current  case, the result follows from  arguments similar to those for Brauer category  in  \cite{RS4}. Of course, we need   the fact that $\text{wt}(\mu)\prec\text{wt}(\lambda)$ if $[\tilde \Delta(\lambda):L(\mu)]\neq0$ and $\lambda\neq \mu$. One can check this fact by  Proposition~\ref{dedji} and Proposition~ \ref{dektss1}.
\end{rem}

\section{Cyclotomic Brauer categories and  cyclotomic Kauffman categories }

In this section we apply the general theory in sections~2--5 to study  cyclotomic Brauer categories and cyclotomic Kauffman categories. First of all we recall some results on (degenerate) cyclotomic Hecke algebras. Throughout,  $m$ is a fixed positive integer.
\subsection{(Degenerate) cyclotomic Hecke algebras}\label{hecke}
Given  $ \mathbf u=(u_1,u_2,\ldots, u_m)\subset (\Bbbk^\times)^m$, the cyclotomic Hecke algebra $\mathscr H_{m, n}(\mathbf u)$ is the associative  $\Bbbk$-algebra generated by $L_1, T_1, \ldots, T_{n-1}$ subject to the relations:
\begin{itemize}\item[(1)] $(T_i-q)(T_i+q^{-1})=0$, $1\le i\le n-1$ and $q\in \Bbbk^\times$,
\item[(2)] $T_iT_j=T_jT_i$, $1\le i<j-1\le n-2$,
\item [(3)] $T_iT_{i+1}T_i=T_{i+1}T_iT_{i+1} $, $1\le i\le n-2$,
\item [(4)] $T_1L_1T_1L_1=L_1T_1L_1T_1$,
\item [(5)] $L_1T_i=T_iL_1$, $2\le i\le n-1$,
\item [(6)] $(L_1-u_1) \cdots (L_1-u_m)=0$.
\end{itemize}
It is known that   $\mathscr H_{m, n}(\mathbf u)$ is a cellular algebra with certain cellular basis (see \cite[Theorem~3.26]{DJM}).  The corresponding cell modules are denoted by   $S(\lambda)$, $\lambda\in \Lambda_{m, n}$, where  $\Lambda_{m, n}$
is the set of all $m$-tuple of partitions (or $m$-partition) $(\lambda^{(1)}, \lambda^{(2)}, \ldots, \lambda^{(m)})$  of $n$. When all   $u_j$'s  are in the same $q^2$-orbit in the sense that $u_iu_j^{-1}\in q^{2\mathbb Z}$ for all $1\le i< j\le m$, it is proved in \cite{Ari1} that  $S(\lambda)$ has the simple head $D(\lambda)$ if and only if $\lambda\in \bar\Lambda_{m, n}$, where $\bar\Lambda_{m, n}$ is the set of all \textsf{$\mathbf u$-Kleshchev $m$-partitions}  of $n$. If $\mathbf u$ is  a disjoint union of certain $q^2$-orbits, then the above result on the classification of simple modules  is still available. In this case,  one can use Morita equivalence theorem for   cyclotomic Hecke algebra in \cite[Theorem~1.1]{DM}.

Let     $L_i=T_{i-1}L_{i-1}T_{i-1}$ for  $ 2\le i\le n$. Then $L_1, L_2, \ldots, L_n$ are known as the \textsf{Jucys-Murphy elements} of  $\mathscr H_{m, n}(\mathbf u)$.
Thanks to the results on representations of $\mathscr H_{m, n}(\mathbf u)$, all generalized eigenvalues of $L_i$,  $1\le i\le n$,   are  of forms $u_jq^{2k}$,    $1\le j\le m$ and $1-n\le k\le n-1$ (e.g. \cite[\S3]{JM}).  When $\mathbf u$ is a $q^2$-orbit, define  $ \mathbb I_{\mathbf u}=\{  u_1q^{2j}\mid  j\in\mathbb Z\}$.
Let  $ \mathfrak {sl}_{\mathbb  I_\mathbf u}$  be  the Kac-Moody Lie algebra    with respect to the Cartan matrix $(a_{i,j})_{i,j\in \mathbb I_\mathbf u}$, where \begin{equation}\label{cart1} a_{i,j}=\begin{cases}
    2, & \text{if $ij^{-1}=1$,} \\
    -1, & \text{if $ij^{-1}=q^{\pm2}$ and $e\neq 2$,} \\
 -2, & \text{if $ij^{-1}=q^2$ and $e=2$,} \\
    0, & \text{ otherwise,}
  \end{cases}\end{equation}
  and  $e$ is  the quantum characteristic of $q^2$.  In this case, $\mathfrak{sl}_{\mathbb I_\mathbf u}=\mathfrak{sl}_\infty$ if $e=\infty$ and $\mathfrak {sl}_{\mathbb I_\mathbf u}=\hat{ \mathfrak{sl}}_e$ if $e<\infty$.
 Recall that the content of a node $x$  with respect to  an $m$-partition $\lambda$   is $c(x)=u_jq^{2(k-l)}$, if $x$ is at the $l$th row and $k$th column of   the $j$th component of the Young diagram $[\lambda]$ with respect to  $\lambda$~\cite{Ma}.
  Fix  simple roots  $\Pi=\{\alpha_i\mid i\in \mathbb I_\mathbf u\}$.
Let $$\mathcal F_{\mathbf u} = \bigoplus_{\lambda\in \Lambda(m)} \mathbb C v_\lambda,$$
where $\Lambda(m)=\bigcup_{n=0}^\infty \Lambda_{m, n}$. Then $\mathcal F_{\mathbf u}$ is a left  $\mathfrak {sl}_{ \mathbb I_\mathbf u}$-module (e.g. \cite{Ari}) such that for any  $i\in \mathbb I_\mathbf u$, \begin{itemize}
\item[(1)]  $e_iv_\lambda=\sum_\mu v_\mu$, where $\mu$'s  are obtained from $\lambda$
      by removing  a removable node of content $i$.
\item [(2)] $f_iv_\lambda=\sum_\mu v_\mu$, where $\mu$'s  are obtained from $\lambda$
      by adding   an addable node of content $i$.
\item [(3)]The action of the Cartan subalgebra $\mathfrak h$ is defined so that $v_\lambda$ is a weight vector of weight
\begin{equation}\label{omg1} \text{wt}(\lambda):= \omega_\mathbf u-\sum_{x\in[\lambda]}\alpha_{c(x)},\end{equation} where $\omega_\mathbf u=\sum_{j=1}^m\omega_{u_j}$, and $\omega_i$'s (resp., $e_i, f_i$) are fundamental weights (resp., Chevalley generators) for any $i\in \mathbb I_\mathbf u$.
\end{itemize}
 Let $V(\omega_{\mathbf u})$ be  the $\mathfrak {sl}_{ \mathbb I_\mathbf u}$-submodule of  $\mathcal F_{\mathbf u}$ generated by $v_\emptyset$. Then
  $V(\omega_{\mathbf u})$ is the integrable highest weight module with the highest weight $\omega_{\mathbf u}$.
  Let $\mathscr H(\mathbf u)=\bigoplus_{n=0}^\infty \mathscr H_{m,n}(\mathbf u)$. As $ \mathfrak {sl}_{ \mathbb I_\mathbf u}$-modules~\cite{Ari} \begin{equation}\label{isomor}  [\mathscr H(\mathbf u) \text{-pmod}] \cong V(\omega_\mathbf u).\end{equation}   The required isomorphism $\phi$ is given by
\begin{equation} \label{iso321} \phi([Y(\lambda)])= \sum_{\mu\in  \Lambda(m)} [S(\mu): D(\lambda)]v_\mu,\end{equation}
for any $\lambda\in \bar \Lambda(m):=\bigcup_{n=0}^\infty \bar \Lambda_{m, n}$, where  $Y(\lambda)$ is the projective cover of $D(\lambda)$ (e.g. \cite[\S~5.3]{K}) and
the actions of $e_i$ and $f_i$ are given by $i$-restriction functor $\bar E_i$ and $i$-induction functor $\bar F_i$ for all $i\in \mathbb I_\mathbf u$ (see, e.g.  \cite[\S~7.4]{K}).

When $\mathbf u=\mathbf u_1\coprod \ldots \coprod\mathbf u_k$, a disjoint union of certain $q^2$-orbits with $m_j=\sharp \mathbf u_j$,
arrange  $u_i$'s so  that $u_l\in \mathbf u_j $ if $\sum_{i=1}^{j-1}m_i+1\leq l\leq \sum_{i=1}^{j}m_i$. We  have the Cartan matrix $D_j$ with respect to $\mathbb I_{\mathbf u_j}$ in the sense of \eqref{cart1}, and $D=\text{diag} (D_1, D_2, \ldots, D_k)$ where $D$ is the Cartan matrix with respect to $\mathbb I_\mathbf u=\bigcup_{j=1}^k  \mathbb I_{\mathbf u_j}  $. So,
\begin{equation}\label{juded}
\mathfrak {sl}_{ \mathbb I_\mathbf u}\cong \bigoplus _{j=1}^k \mathfrak {sl}_{\mathbb  I_{\mathbf u_j}}.
\end{equation}
Similarly,  for any $j$, $1\leq j\leq k$,  we have the $ \mathfrak {sl}_{\mathbb  I_{\mathbf u_j}}$-module $\mathcal F_{\mathbf u_j}$ and its  $\mathfrak {sl}_{ \mathbb I_{\mathbf u_j}}$-submodule
$V(\omega_{\mathbf u_j})$  generated by $v_\emptyset$. Thanks to  \eqref{juded}, $V(\omega_\mathbf u)$ can be identified with  $V(\omega_{\mathbf u_1})\otimes V(\omega_{\mathbf u_2})\otimes \ldots \otimes V(\omega_{\mathbf u_k})$, and $\mathcal F_{\mathbf u}$ can be identified with  $\mathcal F_{\mathbf u_1}\otimes \mathcal F_{\mathbf u_2}\otimes \ldots \otimes \mathcal F_{\mathbf u_k}$. The following result should be well-known to the expert. We need it when we study  the cyclotomic Kauffman categories later on.

\begin{Lemma}\label{dddddcate}
As $ \mathfrak {sl}_{ \mathbb I_\mathbf u}$-modules,     $ [\mathscr H(\mathbf u) \text{-pmod}] \cong V(\omega_\mathbf u)$ and the
 required isomorphism $\phi$ is given by
$$\phi([Y(\lambda)])= \sum_{\mu\in  \Lambda(m)} [S(\mu): D(\lambda)]v_\mu.$$
Moreover, the Chevalley generators $e_i, f_i$ act  on $ [\mathscr H(\mathbf u) \text{-pmod}]$ via the endomorphisms induced  by the  $i$-restriction and $i$-induction functors $\bar E_i,\bar F_i$ for all $i\in\mathbb I_\mathbf u$.
 \end{Lemma}

\begin{proof} For any $\lambda\in \bar \Lambda(m)$, let $b_\lambda:=\sum_{\mu\in  \Lambda(m)} [S(\mu): D(\lambda)]v_\mu$. Then the  elements $b_\lambda$'s are linearly independent since the rectangular matrix
$([S(\mu):D(\lambda)])$ is unitriangular. This shows that  the linear map
\begin{equation}\label{proj123} \phi_\mathbf u:  [\mathscr H(\mathbf u) \text{-pmod}] \rightarrow \mathcal  F_\mathbf u, ~~ [Y(\lambda)]\mapsto b_\lambda\end{equation}
is injective.
There is a non-degenerate Cartan pairing
$$(\ ,\ ): K_0(\mathscr H(\mathbf u) \text{-pmod} )\times K_0( \mathscr H(\mathbf u) \text{-fdmod}) \rightarrow \mathbb Z$$
such that $ ([Y(\lambda)], [D(\mu)])=\dim\Hom_{\mathscr H(\mathbf u)}(Y(\lambda),D(\mu))=\delta_{\lambda,\mu}$.
It is known that $\bar E_i$ and $\bar F_i$ are biadjoint with each other,  and hence  $[\bar E_i]$ and $[\bar F_i]$
are biadjoint with respect to the above pairing, where $[F]$ is the linear map induced by the exact functor $F$.
There is also a non-degenerate symmetric  bilinear form  $(\ ,\ )$ on $\mathcal F_{\mathbf u}$  so that $\{v_\lambda\mid \lambda\in\Lambda(m)\}$
is an orthonormal basis, and moreover, $e_i$ is biadjoint to $f_i$ with this  form. So, the  $ \mathfrak {sl}_{ \mathbb I_\mathbf u}$-module  $\mathcal F_{\mathbf u}$ can be identified  with $\mathcal F_{\mathbf u}^*$ and  the dual map of $\phi_\mathbf u$ is :
$$\phi_\mathbf u^*: \mathcal F_{\mathbf u} \rightarrow [\mathscr H(\mathbf u) \text{-fdmod} ], ~~ v_\lambda\mapsto   \sum_{\mu\in  \bar\Lambda(m)} [S(\lambda): D(\mu)][D(\mu)].$$
Note that $\phi_\mathbf u^*(v_\lambda)=[S(\lambda)]$. It is clear from the brunching rule of $S(\lambda)$ (e.g. \cite[Corollary~1.12]{AM} or \cite[Proposition~3.7]{JM}) that
$\phi_\mathbf u^*$ intertwines $e_i$, $f_i$ with $[\bar E_i]$, $[\bar F_i]$, respectively.
It is routine to check that $(\phi_\mathbf u(x), y)=(x,\phi_\mathbf u^*(y))$ for any $x\in K_0(\mathscr H(\mathbf u) \text{-pmod} )$ and $y\in \mathcal F_{\mathbf u}$.
We have
$$\begin{aligned}(\phi_\mathbf u([\bar E_iY(\lambda)]), v_\mu)&=([\bar E_iY(\lambda)], \phi_\mathbf u^*(v_\mu))=
([Y(\lambda)], [\bar F_i](\phi_\mathbf u^*(v_\mu) ))\\
=&([Y(\lambda)],  \phi^*_\mathbf u(f_iv_\mu) )=
(\phi_\mathbf u([Y(\lambda)]),   f_iv_\mu )= ( e_i\phi_\mathbf u([Y(\lambda)]), v_\mu ).  \end{aligned} $$
Similar equality holds for $\bar F_i$ and $f_i$. So,  $\phi_\mathbf u$ intertwines $[\bar E_i]$, $[\bar F_i]$ with $e_i$, $f_i$, respectively. This proves that
 $ [\mathscr H(\mathbf u) \text{-pmod}]$ is an  $ \mathfrak {sl}_{ \mathbb I_\mathbf u}$-module on which   $e_i$ and $f_i$ act via restriction functor $\bar E_i$ and induction functor  $\bar F_i$ for all $i\in \mathbb I_\mathbf u$,
 and  $\phi_\mathbf u$ is an $ \mathfrak {sl}_{ \mathbb I_\mathbf u}$-homomorphism.
It remains to prove that the image of $\phi_\mathbf u$ is the submodule $V(\omega_\mathbf u)$.
Thanks to \eqref{isomor},   this is true  if there is only one orbit in $\mathbf u$. So,  the result holds for $\mathscr H(\mathbf u_j)$ , $1\leq j\leq k$, where $\mathscr H(\mathbf u_j)=\bigoplus_{n=0}^\infty \mathscr H_{m_j,n}(\mathbf u_j)$. In general,  we can  partition $\lambda $ into  $k$ parts such that  $ \lambda=(\lambda_1, \lambda_2,\ldots, \lambda_k)$ and  $\lambda_j$ is an $m_j$-partition.  By the Morita equivalence theorem  \cite[Theorem 1.1, Proposition~4.11]{DM}, $[S(\mu): D(\lambda)]= \prod_{j=1}^k[S(\mu_j): D(\lambda_j)]$.
Identifying $\mathcal F_{\mathbf u}$  with  $\mathcal F_{\mathbf u_1}\otimes \mathcal F_{\mathbf u_2}\otimes \ldots \otimes \mathcal F_{\mathbf u_k}$ yields
$ b_\lambda=b_{\lambda_1}\otimes \ldots\otimes  b _{\lambda_k}$. So,
the image of $\phi_\mathbf u$ is  $V(\omega_\mathbf u)$.\end{proof}

\begin{rem}\label{degerecate}
Let $ H_{m, n}(\mathbf u)$ be the degenerate cyclotomic Hecke algebra~\cite{K} over $\Bbbk$ with respect to the parameters $\mathbf u=(u_1,u_2,\ldots,u_m)\in\Bbbk^m$. In this case,   $u_i$ and $u_j$ are in the same orbit if and only if $u_i-u_j\in \mathbb Z$. Similar to the non-degenerate case, we write  $$\mathbf u=\mathbf u_1\coprod \ldots \coprod\mathbf u_k,$$ a disjoint union of certain orbits with $m_j=\sharp \mathbf u_j$,  where  $u_i$'s are  arranged such that $u_l\in \mathbf u_j $ if $\sum_{i=1}^{j-1}m_i+1\leq l\leq \sum_{i=1}^{j}m_i$. Moreover,  $H_{m,n}(\mathbf u)$ is a cellular algebra with cellular basis given in \cite[Theorem~6.3]{AMR}.
Let $  H(\mathbf u)=\bigoplus_{n=0}^\infty   H_{m,n}(\mathbf u)$ and  $  H(\mathbf u_j)=\bigoplus_{n=0}^\infty   H_{m_j,n}(\mathbf u_j)$.
By arguments similar to those   in the proof of Lemma~\ref{dddddcate}, one can see that there is a result   similar to that in Lemma~\ref{dddddcate} for  $H(\mathbf u)$. See remarks
after \cite[Theorem~8.5]{AMR} in which Ariki et.~al stated that there is a Morita equivalence theorem for degenerate cyclotomic Hecke algebra similar to those for cyclotomic Hecke algebra in \cite[Theorem~1.1]{DM}.
In this case, the complete set of inequivalent simple  $H_{m,n}(\mathbf u)$-modules are parameterized by  $\bar\Lambda_{m, n}$, the set of all $\mathbf u$-restricted $m$-partitions of $n$~\cite{K}. Furthermore,
\begin{itemize}
\item  $c(x)= u_j+k-l$ if the box  $x$ is at the $l$th row and $k$th column of   the $j$th component of the Young diagram with respect to  $\lambda$ (e.g.\cite[\S~6]{AMR}),
\item $\mathbb I_{\mathbf u_j}=\{u_i+n\mid n\in \mathbb Z \}$, where $u_i$ is any element in $\mathbf u_j$, and
 $\mathbb I_\mathbf u=\cup_{j=1}^k  \mathbb I_{\mathbf u_j}  $,
\item  $ \mathfrak {sl}_{ \mathbb I_\mathbf u}$ is the Kac-Moody Lie algebra with respect to the Cartan matrix $(a_{i,j})_{i,j\in \mathbb I_\mathbf u}$ such that
$$a_{i,j}=\begin{cases}
    2, & \text{if $i=j$;} \\
    -1, & \text{if $i=j\pm 1$ and $p\neq 2$;} \\
     -2, & \text{if $i=j-1$ and $p= 2$;} \\
    0, & \hbox{ otherwise.}
  \end{cases}
$$
where $p$ is the characteristic of $\Bbbk$.\end{itemize}

\end{rem}

 \subsection {Representations of cyclotomic Brauer categories}\label{cba} Suppose $\mathbf u=(u_1, u_2, \ldots, u_m)\in \Bbbk^m$ and $\text{char}~ \Bbbk=p\neq 2$. The (specialized) cyclotomic  Brauer category $\CB(\mathbf u)$ with respect to the polynomial $f(t)=\Pi_{i=1}^m(t-u_i)$  (denoted by $\CB^f(\omega)$ in \cite[Theorem~C]{RS3}) is introduced in \cite[Definition~1.7]{RS3}. It is a quotient category of the  affine Brauer category $\AB$  (a $ \Bbbk$-linear strict monoidal category). Furthermore,  the affine Brauer category $\AB$ is one of $\mathcal C_1$ since it satisfies \eqref{relation 1}.

Let $A$ be  the algebra  associated to  $\CB(\mathbf u)$.
  When  $\CB(\mathbf u)$ is the  Brauer category, i.e., $m=1$,  thanks to \cite{RS4},  $A$   admits an upper finite triangular decomposition in the sense of \cite{BS}.  Hence $A$ admits  an upper finite weakly triangular decomposition in the sense of Definition~\ref{WT}.

A  dotted $(a, b)$-Brauer diagram is   an $(a, b)$-Brauer diagram given in subsection~2.2 on which there are finite many $\bullet$ on each strand.
A normally ordered dotted $(a, b)$-Brauer diagram is   a dotted $(a, b)$-Brauer diagram on which
 \begin{itemize}\item [(a)] there are at most $ m-1$ dots on the boundary of the lower row\footnote{The original definition in  \cite[Definition~1.3]{RS3} is upper row here. However, using \cite[(3.3)]{RS3} the basis theorem remains true if it is  replaced by the lower row by induction on the number of dots.} if they  appear on a vertical strand,
\item [(b)] there are at most $ m-1$ dots on the leftmost boundary of a cap (resp., rightmost boundary  of a cup)  if they appear on a cap (resp., cup).\end{itemize}
     Suppose $m=4$.  The  right one is a normally ordered dotted  $(5, 5)$-Brauer diagram and the left one is not:
\begin{equation}\label{two AOBC diagrams}
    \begin{tikzpicture}[baseline = 25pt, scale=0.35, color=\clr]
        \draw[-,thick] (2,0) to[out=up,in=down] (0,5);
        \draw[-,thick] (6,0) to[out=up,in=down] (6,5);
        \draw[-,thick] (7,0) to[out=up,in=down] (7,5);
            \draw[-,thick] (0,0) to[out=up,in=left] (2,1.5) to[out=right,in=up] (4,0);
        \draw[-,thick] (2,5) to[out=down,in=left] (3,4) to[out=right,in=down] (4,5);
           \draw(4,0.4)\bdot;
      \draw(0.1,4.5)\bdot;
    \end{tikzpicture}~,
    \qquad\qquad\qquad
    \begin{tikzpicture}[baseline = 25pt, scale=0.35, color=\clr]
        \draw[-,thick] (2,0) to[out=up,in=down] (0,5);
        \draw[-,thick] (6,0) to[out=up,in=down] (6,5);
         \draw[-,thick] (6.7,0) to[out=up,in=down] (6.7,5);
               \draw[-,thick] (0,0) to[out=up,in=left] (2,1.5) to[out=right,in=up] (4,0);
        \draw[-,thick] (2,5) to[out=down,in=left] (3,4) to[out=right,in=down] (4,5);
                              \draw(0,0.3)\bdot;
  \draw(6,0.6)\bdot;\draw(3.95,4.6)\bdot;
    \draw(2,0.6)\bdot;
        \draw (0.5,0.4) node{$3$};
           \end{tikzpicture}~.
\end{equation}
In the above diagrams, $ \bullet~ 3$ represents that   there are  three  $ \bullet$'s on the leftmost boundary on the cap.
Two  normally ordered dotted   $(a,b)$-Brauer diagrams are said to be equivalent if the underlying   Brauer diagrams are equivalent and there are the same number of dots on their corresponding strands.

 Let $ \mathbb{B}_{a,b}$ be the set of all equivalence classes of   normally ordered dotted  $(a,b)$-Brauer diagrams.
Thanks to   \cite[Theorem~C]{RS3},
    $\Hom_{\CB(\mathbf u)} (  a,   b)$ is of maximal dimension  if and only if   the $\mathbf u$-admissible condition holds in the sense of \cite[Definition~1.8]{RS3}\footnote{We do not need explicit description on \cite[Definition~1.8]{RS3} in the current paper.}.
 In this case,  $\Hom_{\CB(\mathbf u)} ( a,    b)$ has  basis given by $ \mathbb{B}_{a,b}$ and two equivalent diagrams represent the same morphism in $\CB(\mathbf u)$. So,  we can identify each equivalence class   with any element in it.  Let $\mathbb{B}=\bigcup_{a,b\in \mathbb N}\mathbb{B}_{a,b}$.

\begin{Prop}\label{CBWA} If $\mathbf u$-admissible condition  holds in the sense of \cite[Definition~1.8]{RS3}, then  $\CB(\mathbf u)$  is an upper finite weakly triangular category.
 \end{Prop}
\begin{proof}  The objects in
 $\CB(\mathbf u)$ are given by  $I=J=\mathbb N$.
So, we have the data in (W1) with $(\mathbb N, \preceq)$ given in (A1).  Obviously, $(\mathbb N, \preceq)$ is upper finite.
 For the data in (W2), we define
\begin{itemize}\item [(a)] $A^-$:  the $\Bbbk$-space with basis consisting  of all elements in $\mathbb B$ on which there are  neither  caps  nor crossings among vertical strands and there are no dots on vertical strands,
\item [(b)] $A^\circ$: the $\Bbbk$-space with basis consisting  of all  elements in $\mathbb B$ on which there are neither  cups nor caps,
 \item [(c)] $A^+$ : the $\Bbbk$-space with basis consisting of all  elements in $\mathbb B$ on which there are neither cups nor crossings among vertical strands and there are no dots on vertical strands.
\end{itemize}
So,  $\{1_a\}=\mathbb B\cap A_{a}^-=\mathbb B\cap A_{a}^+$.
Moreover, if $a\npreceq b$ (i.e., $a\neq b+2s$ for all $s\in\mathbb N$), then  $\mathbb B\cap  A_{a, b}^-=\mathbb B\cap A_{b, a}^+=\emptyset$
and hence (W3) follows.

 For any $(g,h,k)\in (  A^-1_a\cap \mathbb B, A^\circ_a\cap \mathbb B, 1_aA^+\cap \mathbb B)$, the composition $g\circ h\circ k$ is not  normally ordered in general, i.e., the dots on the vertical strands may not be  on the lower line. For example: $$g\circ h \circ k=\begin{tikzpicture}[baseline = 25pt, scale=0.35, color=\clr]
        \draw[-,thick] (2,0) to[out=up,in=down] (0,5);
       \draw[-,thick] (0,0) to[out=up,in=left] (2,1.5) to[out=right,in=up] (4,0);
        \draw[-,thick] (2,5) to[out=down,in=left] (3,4) to[out=right,in=down] (4,5);
       \draw (0,0.3) \bdot;
        \draw  (3.95,4.6) \bdot;
        \draw  (1.2,2.1) \bdot;
           \end{tikzpicture}$$ if $(g, h, k)=\left(\ \ \begin{tikzpicture}[baseline = -0.5mm]
 \draw[-,thick,darkblue] (0,0) to[out=down,in=left] (0.28,-0.28) to[out=right,in=down] (0.56,0);
 \node at (0.56,0) {$\color{darkblue}\scriptstyle\bullet$};
  \draw[-,thick,darkblue] (0,0) to (0,.2);
  \draw[-,thick,darkblue] (-0.1,-0.28) to (-0.1,.2);
   \draw[-,thick,darkblue] (0.56,0) to (0.56,.2);
 \end{tikzpicture},~\begin{tikzpicture}[baseline = -0.5mm]
 \node at (0,0) {$\color{darkblue}\scriptstyle\bullet$};
  \draw[-,thick,darkblue] (0,-0.28) to (0,.2);
 \end{tikzpicture},~\begin{tikzpicture}[baseline = 2.5mm]
 \draw[-,thick,darkblue] (0.28,0) to[out=90,in=-90] (-0.28,.6);
 \draw[-,thick,darkblue] (-0.28,0) to[out=90,in=-90] (0.28,.6);
 \node at (-0.26,0.1) {$\color{darkblue}\scriptstyle\bullet$};
  \draw[-,thick,darkblue] (0.28,.6) to[out=up,in=left] (0.56,0.85) to[out=right,in=up] (0.56+0.28,0.6);
  \draw[-,thick,darkblue] (0.56+0.28,0.6) to (0.56+0.28,0);
\end{tikzpicture}\  \ \right)$.
 By the defining relation for affine Brauer category $\AB$ in \cite{RS3},  \begin{equation}\label{movedot}
                \begin{tikzpicture}[baseline = 7.5pt, scale=0.5, color=\clr]
            \draw[-,thick] (0,0) to[out=up, in=down] (1,2);
            \draw[-,thick] (0,2) to[out=up, in=down] (0,2.2);
            \draw[-,thick] (1,0) to[out=up, in=down] (0,2);
            \draw[-,thick] (1,2) to[out=up, in=down] (1,2.2);
             \draw(0,1.9)\bdot;
        \end{tikzpicture}
        ~-~
        \begin{tikzpicture}[baseline = 7.5pt, scale=0.5, color=\clr]
            \draw[-,thick] (0,0) to[out=up, in=down] (1,2);\draw[-,thick] (0,0) to[out=up, in=down] (0,-0.2);
             \draw[-,thick] (1,0) to[out=up, in=down] (0,2);\draw[-,thick] (1,0) to[out=up, in=down] (1,-0.2);
                        \draw(1,0.1)\bdot;
        \end{tikzpicture}
        ~=~
       \begin{tikzpicture}[baseline = 10pt, scale=0.5, color=\clr]
         \draw[-,thick] (2,2) to[out=down,in=right] (1.5,1.5) to[out=left,in=down] (1,2);
            \draw[-,thick] (2,0) to[out=up, in=right] (1.5,0.5) to[out=left,in=up] (1,0);
        \end{tikzpicture}
        ~-~
        \begin{tikzpicture}[baseline = 7.5pt, scale=0.5, color=\clr]
            \draw[-,thick] (0,0) to[out=up, in=down] (0,2);
            \draw[-,thick] (1,0) to[out=up, in=down] (1,2);
                   \end{tikzpicture}~.
    \end{equation}
 The local relation ~\eqref{movedot} explains that the dots on each strand of $g\circ h\circ k$ can  slid freely up to  a  linear combination of some  elements in $\mathbb B$ with fewer dots~\cite[(3.3)]{RS3}. So, all dots  can be slid such that the resulting diagram is normally ordered. This shows that  the multiplication map from $A^-\otimes_{\mathbb K }  A^\circ \otimes_{\mathbb K}  A^+$ to $A$ sends  any basis element to a unique corresponding  basis element in $\mathbb B$ (up to  a  linear combination of some  elements in $\mathbb B$  with fewer dots). Furthermore, the multiplication map is obviously surjective. So, it is the required  isomorphism and (W4) holds.
 \end{proof}
 We are going to show that Assumption~\ref{pdual} and  (A1)-(A10)  hold for $A$ associated to $\CB(\mathbf u)$  provided that the $\mathbf u$-admissible condition holds. In this case, $\mathcal C=\mathcal C_1=\AB$. If so, then all the results in sections~2-5 can be applied to $\CB(\mathbf u)$.

  The Brauer category has been studied in  \cite{RS4}. So, we assume  $m>1$.
Thanks to Proposition~\ref{CBWA}, (A1) holds. Define (see Definition~\ref{defniogfxhy})
  \begin{equation}\label{dfxyh}
  Y(b,a)=  \mathbb B\cap A^-_{b, a} ,
   ~~H(a)=\mathbb B\cap  A_a^\circ,
   ~~X(a,b)=\mathbb B\cap  A^+_{a, b}.\end{equation}
 Then (A2) holds obviously.
To prove that (A3) holds, we define
\begin{equation}\label{jdijdijxs}
\begin{aligned}
           S_i 1_{a}=1_{a}S_i= \begin{tikzpicture}[baseline = 25pt, scale=0.35, color=\clr]
       \draw[-,thick](0,1.1)to[out= down,in=up](0,3.9);
       \draw(0.5,1.1) node{$ \cdots$}; \draw(0.5,3.9) node{$ \cdots$};
       \draw[-,thick](1.5,1.1)to[out= down,in=up](1.5,3.9);
       \draw[-,thick] (2,1) to[out=up, in=down] (3,4);
       \draw(2,4.5)node{\tiny$i$}; \draw(3.2,4.5)node{\tiny$i+1$};
        \draw[-,thick] (3,1) to[out=up, in=down] (2,4);
         \draw[-,thick](4,1.1)to[out= down,in=up](4,3.9);
         \draw(4.5,1.1) node{$ \cdots$}; \draw(4.5,3.9) node{$ \cdots$};
         \draw[-,thick](5.5,1.1)to[out= down,in=up](5.5,3.9);
         \draw(5.5,4.5)node{\tiny$a$};
           \end{tikzpicture}\in A_a,
           \end{aligned}
           \end{equation}
           for any admissible $i, a\in \mathbb N$. If we know $a$ from the context, we simply denote $ S_i1_{a}$ by $ S_i$. Let
           $$D(a+1)=\{S_{i, a+1}X_{a+1}^j| 1\leq i \leq a+1, 0\leq j\leq m-1\}\subseteq A_{a+1},$$
           where $X_i$ is in \eqref{xitil} and $S_{i, a+1}=S_{i}\circ S_{i+1}\circ \cdots \circ S_{a}$, if $i<a+1$ and $S_{a+1, a+1}=1_{a+1}$.     Then $$D(a+1)\subseteq H(a+1),\ \  |H(a+1)|=|H(a)||D(a+1)|.$$ Given any  $g\in H_{a+1} $ such that the $i$th vertex on the top row is connected with the rightmost  vertex on the bottom row and the corresponding strand has no dots,
    there is some $\tilde g\in H(a) $ such that $S_{i,a+1}^{-1}\circ g= \tilde g~\begin{tikzpicture}[baseline = 10pt, scale=0.5, color=\clr]
                \draw[-,thick] (0,0.5)to[out=up,in=down](0,1.2);
    \end{tikzpicture}$.
    Then $g=S_{i,a+1}\tau_A(\tilde g)$. Since $X_{a+1}\tau_A(d')=\tau_A(d')X_{a+1}$, we have $H(a+1)=\{d\circ \tau_A(d')\mid (d, d')\in D(a+1)\times H(a)\}$. So, (A3) holds.

     Let $Y_1(a)$ be the subset of $Y(a)$ consists of all $d\in Y(a)$ such that the rightmost vertex on the top row of $d$ is on the rightmost vertical strand.  By  \eqref{dfxyh}, $ Y_1(a)=\tau_A(Y(a-1))$ and
  $Y(a)=Y_1(a)\overset {.} \sqcup K_2(a)$, where the rightmost  vertex on the top row  of $d$ is on a cup if  $d\in K_2(a)$.    By \eqref{relation 1} (or  \cite[Lemma~3.1]{RS3}),  \begin{equation}\label{cup-dot}\begin{tikzpicture}[baseline = -0.5mm]
 \draw[-,thick,darkblue] (0,0) to[out=down,in=left] (0.28,-0.28) to[out=right,in=down] (0.56,0);
 \node at (0,0) {$\color{darkblue}\scriptstyle\bullet$};
  \draw[-,thick,darkblue] (0,0) to (0,.2);
   \draw[-,thick,darkblue] (0.56,0) to (0.56,.2);
 \end{tikzpicture}
 ~=~
 - \begin{tikzpicture}[baseline = -0.5mm]
 \draw[-,thick,darkblue] (0,0) to[out=down,in=left] (0.28,-0.28) to[out=right,in=down] (0.56,0);
 \node at (0.56,0) {$\color{darkblue}\scriptstyle\bullet$};
  \draw[-,thick,darkblue] (0,0) to (0,.2);
   \draw[-,thick,darkblue] (0.56,0) to (0.56,.2);
 \end{tikzpicture}~~~.
  \end{equation}
  We have  $Y_2(a)\subseteq K_2(a)$ up to a sign since
$$ (f~\begin{tikzpicture}[baseline = 10pt, scale=0.5, color=\clr]
                \draw[-,thick] (0,0.5)to[out=up,in=down](0,1.2);
    \end{tikzpicture})\circ (S_{i,a+1}~\begin{tikzpicture}[baseline = 10pt, scale=0.5, color=\clr]
                \draw[-,thick] (0,0.5)to[out=up,in=down](0,1.2);
    \end{tikzpicture}) \circ (1_{a}~\lcup)\in Y(a),$$  for any  $f\in Y(a+1)$ where $Y_2(a)$ is in (A3).
  Conversely, suppose   $f\in K_2(a)$. Then   the rightmost  vertex on the top row of $f$ is  connected with  another vertex, say the $i$th vertex on the top row. If   there are $j$  $ \bullet$'s on  the corresponding cup, then there is a $f_1\in Y(a+1)$ such that $$ f=(f_1~\begin{tikzpicture}[baseline = 10pt, scale=0.5, color=\clr]
                \draw[-,thick] (0,0.5)to[out=up,in=down](0,1.2);
    \end{tikzpicture})\circ (S_{i,a+1}~\begin{tikzpicture}[baseline = 10pt, scale=0.5, color=\clr]
                \draw[-,thick] (0,0.5)to[out=up,in=down](0,1.2);
    \end{tikzpicture}) \circ (1_{a}~\begin{tikzpicture}[baseline = -0.5mm]
 \draw[-,thick,darkblue] (0,0) to[out=down,in=left] (0.28,-0.28) to[out=right,in=down] (0.56,0);
 \node at (0.56,0) {$\color{darkblue}\scriptstyle\bullet$};
 \draw (0.66,0) node{$\color{darkblue} j$};
  \draw[-,thick,darkblue] (0,0) to (0,.2);
   \draw[-,thick,darkblue] (0.56,0) to (0.56,.2);
 \end{tikzpicture}).$$
So,  $K_2(a)\subseteq Y_2(a)$ up to a sign by \eqref{cup-dot} and hence $Y_2(a)= K_2(a)$ up to a sign. Since $Y(a)$ is a basis of $A^-1_a$,  (A4) follows.

There is a    monoidal contravariant functor $\sigma: \mathcal \AB\rightarrow \mathcal \AB$
switching  $\lcup$ and $\lcap$ and fixing  the generating object \begin{tikzpicture}[baseline = 1.5mm]
	\draw[-,thick,darkblue] (0.18,0) to (0.18,.4);
\end{tikzpicture} and all other generating   morphisms (see \cite{RS3}).  Moreover, $\sigma^2=\text{Id}$ and  $\sigma$
  induces an anti-involution $\sigma_A$ on $A$ and (A5) holds.
 By \cite[(4.45)]{RS3}, there is an algebra isomorphism $ \phi:  A_{a}   \cong W_{m, a}$,  where  $W_{m, a}$ is the cyclotomic Nazarov-Wenzl
algebra in \cite{AMR}.
 Moreover, it's known that $ W_{m, a}/\phi(N_a)\cong H_{m,a}(\mathbf u)$,  where
  $N_a$ is the two-sided ideal of $A_{a}$ generated by $\begin{tikzpicture}[baseline = 15pt, scale=0.35, color=\clr]
        \draw[-,thick] (0,2.2) to[out=down,in=left] (0.5,1.7) to[out=right,in=down] (1,2.2);
         \draw[-,thick] (0,1) to[out=up,in=left] (0.5,1.5) to[out=right,in=up] (1,1);
         \draw[-,thick](1.5,1.1)to[out= down,in=up](1.5,2.2);
         \draw(2.5,1.1) node{$ \cdots$}; \draw(2.5,2.2) node{$ \cdots$};
         \draw[-,thick](3.5,1.1)to[out= down,in=up](3.5,2.2);
           \end{tikzpicture}~ $ and $H_{m,a}(\mathbf u)$   is   the degenerate cyclotomic  Hecke
algebra~\cite{K}. Note that $\text{Ker} \pi_a|_{A_a}=N_a$, where $\pi_a$ is given in Proposition~\ref{bara}(4). So, the above isomorphism  results in an   induced isomorphism $ \bar \phi: \bar A_a=A_{a}  /N_a\rightarrow H_{m,a}(\mathbf u)$.
 The anti-involution $\sigma_{\bar A_a}$ on $H_{m,a}(\mathbf u)\cong \bar A_a$  coincides with the usual anti-involution of $H_{m,a}(\mathbf u)$ fixing all generators.  By the well-known results for $H_{m,a}(\mathbf u)$ in \cite{Ma} we have (A6)-(A10). In this case, $\Lambda_n=\Lambda_{m,n}$ and $\bar\Lambda_n=\bar\Lambda_{m,n}$ for all $n\in\mathbb N$.
The function  $ \sharp $ in Definition~\ref{sharpind} satisfies   $i^\sharp=-i$ and $\mathbb I_0$ in Definition~\ref{deofi0} is $\mathbb I_\mathbf u$ in Remark~\ref{degerecate}.
For any $\lambda\in\Lambda_{m,a} $, $\mathscr A_{i,\lambda}$ (resp., $\mathscr R_{i,\lambda}$)
is the set of all $\mu\in \Lambda_{m,a+1}$ (resp., $\Lambda_{m,a-1}$) such that $\mu$ is obtained from $\lambda$ by adding (resp., removing) a box with content $i$.

Since  $H_{m,a}(\mathbf u)$ is a cellular algebra, by \cite[Chapter~2, Exercise ~7]{Ma},   we have  \begin{equation}\label{key1234}D(\lambda)^\circledast\cong D(\lambda)\end{equation} for all $\lambda\in\bar\Lambda_{m, a}$.
It is proved in  \cite[Theorem~A.2]{BK3} that  $H_{m,a}(\mathbf u)$ is a symmetric algebra.   By
\cite[Theorem~4.4.4]{GHK}, $\text{soc} Y(\lambda)=\text{hd} Y(\lambda)=D(\lambda)$ for any $\lambda\in \bar\Lambda_{m,a}$. Thanks to  \eqref{key1234},
$Y(\lambda)^\circledast\cong Y(\lambda)$.
So, Assumption~\ref{pdual} holds for $\bar A_{a}$ for all $a\in\mathbb N$.

\begin{Theorem}\label{CNW} Let $A$ be the $\Bbbk$-algebra  associated to the cyclotomic Brauer category $\CB(\mathbf u)$.  Suppose the $\mathbf u$-admissible condition holds.
 \begin{itemize} \item [(1)] The complete set of pairwise inequivalent  simple $A$-modules are indexed by $\bar\Lambda(m)$, i.e., the  set of all $\mathbf u$-restricted $m$-partitions in the sense of \cite{K}.
 \item [(2)] $A$-lfdmod is an upper finite fully  stratified category in the sense of \cite[Definition~3.36]{BS} with respect to the stratification $\rho:\bar\Lambda(m)\rightarrow \mathbb N$ such that $\rho(\lambda)=n$ if $\lambda\in \bar\Lambda_{m,n}$.
 \item [(3)]$A$-lfdmod is an upper finite highest weight  category in the sense of \cite[Definition~3.36]{BS} if   $ p=0$ and $u_i-u_j\notin \mathbb Z 1_{\Bbbk}$ for all   $1\leq i< j\leq m$.

     \item [(4)] $A$ is Morita equivalent to $\bigoplus_{n=0}^\infty  H_{m, n}(\mathbf u)$ if $u_i+u_j\not\in \mathbb Z1_\Bbbk$ for all $1\le i\le j\le m$.
         \item [(5)] $A$ is semisimple over $\Bbbk$  if and only if $ p=0$ and $2u_i\notin \mathbb Z 1_\Bbbk$, $u_i\pm u_j\notin \mathbb Z 1_\Bbbk$ for all different positive integers $i, j\leq m$.\end{itemize}
\end{Theorem}
\begin{proof} (1) follows from Proposition~\ref{CBWA} and Theorem~\ref{striateddndn}(3) together with \cite[Theorem~5.4]{K}. See Remark~\ref{degerecate} when  $\mathbf u$ are  divided into several orbits.
(2) follows immediately from Proposition~\ref{CBWA} and Theorem~\ref{ext1}. (3) follows from Theorem~\ref{newstratification} and \cite[Theorem~6.11]{AMR}, where  \cite[Theorem~6.11]{AMR} gives  a necessary and sufficient condition for degenerate cyclotomic Hecke algebra  being semisimple over $\Bbbk$.
 (4) follows from Theorem~\ref{sjdsdh} since $\mathbb I_0\cap \mathbb I_0^\sharp =\emptyset$ under the assumptions.
 Finally, (5) follows from Proposition~\ref{CBWA}, Theorem~\ref{ss} and \cite[Theorem~7.9]{RuiS1},
  where \cite[Theorem~7.9]{RuiS1} gives  a necessary and sufficient condition for cyclotomic Nazarov-Wenzl algebras being semisimple over $\Bbbk$.
\end{proof}

Suppose $\mathfrak g= \mathfrak {sl}_{ \mathbb I_\mathbf u}$ in Remark~\ref{degerecate} and  $\omega_\mathbf u=\sum_{j=1}^m\omega_{u_j}$ where   $\omega_i$'s   are fundamental weights indexed by $i\in\mathbb I_\mathbf u$. Recall the Lie subalgebra $\mathfrak g^\sharp$ of $\mathfrak g$ in Theorem~\ref{gsharpiso} with $\mathbb I_0=\mathbb I_\mathbf u$.
\begin{Theorem}\label{gsharpisobrauercyc} Let $A$ be the $\Bbbk$-algebra  associated to the cyclotomic Brauer category $\CB(\mathbf u)$ such that   the $\mathbf u$-admissible condition holds. As $\mathfrak g^\sharp$-modules,  $[A\text{-mod}^\Delta]\cong V(\omega_{\mathbf u})$.
The corresponding isomorphism $\phi^\sharp$ satisfies
$\phi^\sharp([\Delta(\lambda)])=\sum_{\mu\in  \Lambda(m)} [S(\mu): D(\lambda)]v_\mu$ for any $\lambda\in\bar \Lambda (m)$.
Moreover, $\tilde e_i $ acts via the exact functor $ E_i$ in Theorem~\ref{usuactifuc1} for any $i\in\mathbb I$.
\end{Theorem}
\begin{proof} By Proposition~\ref{CBWA}, $A$ admits an upper finite weakly triangular decomposition. We have verified the Assumptions (A1)-(A9) before Theorem~\ref{CNW}. So, we are able to use  Theorem~\ref{gsharpiso}. By    the degenerate case of  Lemma~\ref{dddddcate} (see Remark~\ref{degerecate}), we immediately have the result.
In this case, $\bar E_i,\bar F_i$ in Theorem~\ref{gsharpiso} are those in Remark~\ref{degerecate} for the degenerate cyclotomic Hecke algebras with the same notation.
\end{proof}

\subsection{Representations of cyclotomic Kauffman categories}\label{CBMW} Suppose $ q, q-q^{-1}\in \Bbbk^\times$ and $\mathbf u=(u_1, u_2, \ldots, u_m)\in(\Bbbk^\times)^m$.
 The cyclotomic Kauffman category $\CK(\mathbf u)$ with respect to the polynomial $f(t)=\Pi_{i=1}^m(t-u_i)$ (denoted by $\CK^f$ in \cite[Theorem~1.15]{GRS}) is introduced in \cite[Definition~1.11]{GRS}.  Moreover, $\CK(\mathbf u)$
 is a quotient of the affine Kauffman category $\AK$ (a strict monoidal category) which is one of $\mathcal C_2$ in section~5.

Assume $m=1$. Then $\CK(\mathbf u)$ is the  Kauffman category and  $\Hom_{\CK(\mathbf u)} ( a,   b)$ has basis given by all equivalence classes of  reduced totally descending $(a, b)$-tangle diagrams~\cite{VT}. We  depict the reduced totally descending tangle diagrams as follows. Roughly speaking,
an  $(a, b)$-tangle diagram is a diagram obtained from an $(a, b)$-Brauer diagram by replacing its crossings via   either over crossings or under crossings.
 We label  the  endpoints at the upper (resp., lower) line  of an $(a, b)$-tangle diagram by  $\bar 1, \bar 2,\ldots,\bar b$ (resp., $1,2,\ldots, a$) from left to right and assume $i< i+1 < \bar j <\bar {j-1}$ for all admissible $i$ and $j$.
A reduced totally descending $(a, b)$-tangle diagram is an $(a, b)$-tangle diagram on which
\begin{itemize}\item [(1)]  two strands cross each other at most once,
\item [(2)] neither a strand crosses  itself nor there is  a loop,
\item [(3)] the strand $(i, j)$ passes  over the strand  $(k, l)$ whenever  $min\{i, j\}<min\{k, l\}$ and $(i, j)$ crosses $(k, l)$, where $i, j, k, l\in\{\bar 1, \bar 2,\ldots,\bar b, 1,2,\ldots, a\}$.
\end{itemize}
Two reduced totally descending $(a, b)$-tangle diagrams are said to be equivalent if their underlying Brauer diagrams are equivalent.

 Suppose that $m\geq 1$.  A reduced totally descending dotted $(a, b)$-tangle diagram is  a reduced totally descending $(a, b)$-tangle diagram on which there are finitely many  dots (i.e., $\bullet$ or $\circ$) on each strand. A normally ordered reduced totally descending dotted $(a, b)$-tangle diagram is  a reduced totally descending dotted $(a, b)$-tangle diagram with dots ($\bullet$ or $\circ$) on it such that:
 \begin{itemize}

 \item [(1)] Whenever a dot ($\bullet$ or $\circ$) appears on a vertical strand, it is on the boundary of the lower row.
    \item [(2)] Whenever a dot ($\bullet$ or $\circ$) appears on  a cap (resp.,  cup),  it is on   the leftmost boundary of the cap (resp.,   the rightmost boundary  of a cup).
    \item [(3)] $\circ$ and $\bullet$  can not occur on the same strand simultaneously.
    \item [(4)] If there are $i$ (resp.,$j$) $\bullet$'s near each endpoint at upper (resp., lower) row, then $$-i, j\in \left\{\lfloor \frac{m-1}{2}\rfloor, \lfloor\frac{m-1}{2}\rfloor-1,\ldots,-\lfloor \frac{m}{2}\rfloor\right\},$$ where $h$ ``$\bullet$" is the same as $-h$ ``$\circ$" if $h<0$.
    \end{itemize}
    In particular, the identity element $1_a$ in     $\Hom_{\CK(\mathbf u)} (  a,  a)$ is
$\begin{tikzpicture}[baseline = 15pt, scale=0.35, color=\clr]
         \draw[-,thick](1.5,1.1)to[out= down,in=up](1.5,2.2);
         \draw(2.5,1.1) node{$ \cdots$}; \draw(2.5,2.2) node{$ \cdots$};
         \draw[-,thick](3.5,1.1)to[out= down,in=up](3.5,2.2);
           \end{tikzpicture}~ $, where the number of strands is $a$.
 When $m=1$, a normally ordered reduced totally descending dotted tangle diagram is just a  reduced totally descending  tangle diagram.
           We give an example on  normally ordered reduced totally descending dotted tangle diagrams. Suppose $m=5$.  The  right one is a  normally ordered reduced totally descending dotted  $(3, 3)$-tangle diagram and the left one is not:

         \begin{equation}\label{11}
    \begin{tikzpicture}[baseline = 25pt, scale=0.35, color=\clr]
        \draw[-,thick] (2,0) to[out=up,in=down] (0,5);
       \draw[-,line width=4pt,white] (0,0) to[out=up,in=left] (2,1.5) to[out=right,in=up] (4,0);
       \draw[-,thick] (0,0) to[out=up,in=left] (2,1.5) to[out=right,in=up] (4,0);
        \draw[-,thick] (2,5) to[out=down,in=left] (3,4) to[out=right,in=down] (4,5);
       \draw (0,0.3) \wdot;
        \draw  (3.95,4.6) \bdot;
        \draw  (1.2,2.1) \bdot;
           \end{tikzpicture}
    \quad, \qquad
    \begin{tikzpicture}[baseline = 25pt, scale=0.35, color=\clr]
        \draw[-,thick] (2,0) to[out=up,in=down] (0,5);
       \draw[-,line width=4pt,white] (0,0) to[out=up,in=left] (2,1.5) to[out=right,in=up] (4,0);
       \draw[-,thick] (0,0) to[out=up,in=left] (2,1.5) to[out=right,in=up] (4,0);
        \draw[-,thick] (2,5) to[out=down,in=left] (3,4) to[out=right,in=down] (4,5);
       \draw (0,0.3) \wdot;
        \draw  (3.95,4.6) \bdot;\draw  (2,0.4) \bdot;
           \end{tikzpicture}.
\end{equation}
 Two normally ordered reduced totally descending dotted  $(a, b)$-tangle diagrams are said to be equivalent if their underlying Brauer diagrams are equivalent
 and there are the same number of $\bullet$ or $\circ$ on their corresponding strands.

Let $ \mathbb{DB}_{a,b}$ be the set of all equivalence classes of  normally ordered reduced totally descending dotted $(a, b)$-tangle diagrams. Thanks to  \cite[Theorem~1.15]{GRS},  $\Hom_{\CK(\mathbf u)} ( a,   b)$ is of maximal dimension  if and only if   the $\mathbf u$-admissible condition holds in the sense of \cite[Definition~1.13]{GRS}\footnote{We do not need the explicit description of  \cite[Definition~1.13]{GRS}. Furthermore, this condition automatically holds when $m=1$.}.
 In this case,  $\Hom_{\CK(\mathbf u)} ( a,   b)$ has  basis given by $ \mathbb{DB}_{a,b}$  and two equivalent diagrams represent the same morphism in $\CK(\mathbf u)$. So, we can identify each equivalence class   with any element in it.  Let $\mathbb{DB}=\bigcup_{a,b\in \mathbb N}\mathbb{DB}_{a,b}$.

 \begin{Prop}\label{CKWA}   If $\mathbf u$-admissible condition holds  in the sense of  \cite[Definition~1.13]{GRS}, then the cyclotomic Kauffman  category $\CK(\mathbf u)$ is an upper finite weakly  triangular category.
 \end{Prop}
\begin{proof} Let $A$ be the algebra associated to $\CK(\mathbf u)$. The upper finite data in (W1) is given by $(\mathbb N, \preceq)$ in (A1).
For the data in (W2), let
\begin{itemize}\item [(a)]  $A^-$:  the $\Bbbk$-space with basis consisting  of all elements in $\mathbb {DB}$ on which there are  neither  caps  nor crossings among vertical strands and there are no dots on vertical strands.
\item [(b)] $A^\circ :$ the $\Bbbk$-space with basis consisting of all elements in $\mathbb {DB}$ on which there are neither  cups nor caps.
 \item [(c)]  $A^+$ : the $\Bbbk$-space with basis consisting of  all elements in $\mathbb {DB}$ on which there are neither cups  nor crossings among vertical strands and there are no dots on vertical strands.
\end{itemize}
Then one can check  (W3)  by arguments similar to those for (W3) in subsection~\ref{cba}. We omit details here.

  For any   $(g,h,k)\in (A^-1_a\cap  \mathbb{DB}, A^\circ_a\cap  \mathbb{DB}, 1_aA^+\cap  \mathbb{DB})$, the composition $g\circ h\circ k$ is not  normally ordered in general.
 Let $\mathbb N_{a,b}$  be the set of all such   $g\circ h\circ k$'s,
  where   $(g,h,k)\in (A_{b, c}^-\cap  \mathbb{DB}, A_{c}^\circ\cap  \mathbb{DB}, A^+_{c, a}\cap  \mathbb{DB}) $   for all possible  $c$,  and  $c\succeq b, a$.
   Each $g\circ h\circ k$ in $\mathbb N _{a,b}$ satisfies  conditions (2)-(4) as above. However, it does not satisfy condition (1) since  the  dots on the vertical strands may not be on the lower line. We say that a movement is of type I if it is one of the movements as follows:
   $$
\text{Movement I:}\quad \begin{tikzpicture}[baseline = 2.5mm]
	\draw[-,thick,darkblue] (0.28,0) to[out=90,in=-90] (-0.28,.6);
	\draw[-,line width=4pt,white] (-0.28,0) to[out=90,in=-90] (0.28,.6);
	\draw[-,thick,darkblue] (-0.28,0) to[out=90,in=-90] (0.28,.6);
   \node at (-0.26,0.5) {$\color{darkblue}\scriptstyle\bullet$};\node at (-0.15,0.5) {$\color{darkblue}\scriptstyle t$};
\end{tikzpicture}\leftrightsquigarrow \begin{tikzpicture}[baseline = 2.5mm]
	\draw[-,thick,darkblue] (0.28,0) to[out=90,in=-90] (-0.28,.6);
	\draw[-,line width=4pt,white] (-0.28,0) to[out=90,in=-90] (0.28,.6);
	\draw[-,thick,darkblue] (-0.28,0) to[out=90,in=-90] (0.28,.6);
 \node at (0.26,0.1) {$\color{darkblue}\scriptstyle\bullet$};
 \node at (0.15,0.1) {$\color{darkblue}\scriptstyle t$};
\end{tikzpicture},~~\quad
         \begin{tikzpicture}[baseline = 2.5mm]
	\draw[-,thick,darkblue] (-0.28,0) to[out=90,in=-90] (0.28,.6);
	\draw[-,line width=4pt,white] (0.28,0) to[out=90,in=-90] (-0.28,.6);
	\draw[-,thick,darkblue] (0.28,0) to[out=90,in=-90] (-0.28,.6);
   \node at (-0.26,0.5) {$\color{darkblue}\scriptstyle\bullet$};
     \node at (-0.15,0.5) {$\color{darkblue}\scriptstyle t$};
\end{tikzpicture}\leftrightsquigarrow \begin{tikzpicture}[baseline = 2.5mm]
	\draw[-,thick,darkblue] (-0.28,0) to[out=90,in=-90] (0.28,.6);
	\draw[-,line width=4pt,white] (0.28,0) to[out=90,in=-90] (-0.28,.6);
	\draw[-,thick,darkblue] (0.28,0) to[out=90,in=-90] (-0.28,.6);
 \node at (0.26,0.1) {$\color{darkblue}\scriptstyle\bullet$};
 \node at (0.1,0.1) {$\color{darkblue}\scriptstyle\ t$};
\end{tikzpicture},~~\quad
\begin{tikzpicture}[baseline = 2.5mm]
	\draw[-,thick,darkblue] (0.28,0) to[out=90,in=-90] (-0.28,.6);
	\draw[-,line width=4pt,white] (-0.28,0) to[out=90,in=-90] (0.28,.6);
	\draw[-,thick,darkblue] (-0.28,0) to[out=90,in=-90] (0.28,.6);
   \node at (0.26,0.5) {$\color{darkblue}\scriptstyle\bullet$};
   \node at (0.15,0.5) {$\color{darkblue}\scriptstyle t$};
\end{tikzpicture}\leftrightsquigarrow \begin{tikzpicture}[baseline = 2.5mm]
	\draw[-,thick,darkblue] (0.28,0) to[out=90,in=-90] (-0.28,.6);
	\draw[-,line width=4pt,white] (-0.28,0) to[out=90,in=-90] (0.28,.6);
	\draw[-,thick,darkblue] (-0.28,0) to[out=90,in=-90] (0.28,.6);
 \node at (-0.26,0.1) {$\color{darkblue}\scriptstyle\bullet$};
 \node at (-0.15,0.1) {$\color{darkblue}\scriptstyle t$};
\end{tikzpicture},~~\quad
\begin{tikzpicture}[baseline = 2.5mm]
	\draw[-,thick,darkblue] (-0.28,0) to[out=90,in=-90] (0.28,.6);
	\draw[-,line width=4pt,white] (0.28,0) to[out=90,in=-90] (-0.28,.6);
	\draw[-,thick,darkblue] (0.28,0) to[out=90,in=-90] (-0.28,.6);
   \node at (0.26,0.5) {$\color{darkblue}\scriptstyle\bullet$};
    \node at (0.15,0.5) {$\color{darkblue}\scriptstyle t$};
\end{tikzpicture}\leftrightsquigarrow \begin{tikzpicture}[baseline = 2.5mm]
	\draw[-,thick,darkblue] (-0.28,0) to[out=90,in=-90] (0.28,.6);
	\draw[-,line width=4pt,white] (0.28,0) to[out=90,in=-90] (-0.28,.6);
	\draw[-,thick,darkblue] (0.28,0) to[out=90,in=-90] (-0.28,.6);
 \node at (-0.26,0.1) {$\color{darkblue}\scriptstyle\bullet$};
 \node at (-0.15,0.1) {$\color{darkblue}\scriptstyle t$};
\end{tikzpicture}~, ~ t= -1, 1.
$$
Applying the movements of type I on each $g\circ h\circ k$ in $\mathbb N _{a,b}$  we  obtain  the set ${\mathbb {DB}}_{a,b}$.
For example, the left one in \eqref{11} is the composition of
$$\left(\ \ \begin{tikzpicture}[baseline = -0.5mm]
 \draw[-,thick,darkblue] (0,0) to[out=down,in=left] (0.28,-0.28) to[out=right,in=down] (0.56,0);
 \node at (0.56,0) {$\color{darkblue}\scriptstyle\bullet$};
  \draw[-,thick,darkblue] (0,0) to (0,.2);
  \draw[-,thick,darkblue] (-0.1,-0.28) to (-0.1,.2);
   \draw[-,thick,darkblue] (0.56,0) to (0.56,.2);
 \end{tikzpicture},~\begin{tikzpicture}[baseline = -0.5mm]
 \node at (0,0) {$\color{darkblue}\scriptstyle\bullet$};
  \draw[-,thick,darkblue] (0,-0.28) to (0,.2);
 \end{tikzpicture},~\begin{tikzpicture}[baseline = 2.5mm]
 \draw[-,thick,darkblue] (0.28,0) to[out=90,in=-90] (-0.28,.6);
 \draw[-,line width=4pt,white] (-0.28,0) to[out=90,in=-90] (0.28,.6);
 \draw[-,thick,darkblue] (-0.28,0) to[out=90,in=-90] (0.28,.6);
 \node at (-0.26,0.1) {$\color{darkblue}\scriptstyle\circ$};
  \draw[-,thick,darkblue] (0.28,.6) to[out=up,in=left] (0.56,0.85) to[out=right,in=up] (0.56+0.28,0.6);
  \draw[-,thick,darkblue] (0.56+0.28,0.6) to (0.56+0.28,0);
\end{tikzpicture}\  \ \right)$$ and the right one in \eqref{11} is obtained from  the left one by applying a movement in type I as above. Thanks to  \cite[Theorem~1.15]{GRS},   $A_{b, a}$ has basis given by ${\mathbb {DB}}_{a,b}$.  By \cite[Proposition~6.15(1)]{GRS},  any set  obtained from ${\mathbb {DB}}_{a,b}$ by applying movements of type I$'$ is also linear independent over $\mathbb C$, where the movement of type I$'$ is
$ \begin{tikzpicture}[baseline = 2.5mm]
	\draw[-,thick,darkblue] (0.28,0) to[out=90,in=-90] (-0.28,.6);
	\draw[-,line width=4pt,white] (-0.28,0) to[out=90,in=-90] (0.28,.6);
	\draw[-,thick,darkblue] (-0.28,0) to[out=90,in=-90] (0.28,.6);
\end{tikzpicture}\leftrightsquigarrow \begin{tikzpicture}[baseline = 2.5mm]
	\draw[-,thick,darkblue] (-0.28,-.0) to[out=90,in=-90] (0.28,0.6);
	\draw[-,line width=4pt,white] (0.28,-.0) to[out=90,in=-90] (-0.28,0.6);
	\draw[-,thick,darkblue] (0.28,-.0) to[out=90,in=-90] (-0.28,0.6);
\end{tikzpicture}$.
 So,  it is linear independent   over the suitable domain in  which we used
  to define $\CK(\mathbf u)$ in  \cite{GRS} and hence it is a basis of  $A_{b, a}$.  By base change, we obtain  the corresponding result over $\Bbbk$.

   Thanks to the defining relation for $\AK$ in \cite{GRS}, we have \begin{equation}\label{111}
\begin{tikzpicture}[baseline = 2.5mm]
	\draw[-,thick,darkblue] (0.28,0) to[out=90,in=-90] (-0.28,.6);
	\draw[-,line width=4pt,white] (-0.28,0) to[out=90,in=-90] (0.28,.6);
	\draw[-,thick,darkblue] (-0.28,0) to[out=90,in=-90] (0.28,.6);
   \node at (-0.26,0.5) {$\color{darkblue}\scriptstyle\bullet$};
\end{tikzpicture}= \begin{tikzpicture}[baseline = 2.5mm]
	\draw[-,thick,darkblue] (-0.28,0) to[out=90,in=-90] (0.28,.6);
	\draw[-,line width=4pt,white] (0.28,0) to[out=90,in=-90] (-0.28,.6);
	\draw[-,thick,darkblue] (0.28,0) to[out=90,in=-90] (-0.28,.6);
\node at (0.26,0.1) {$\color{darkblue}\scriptstyle\bullet$};
\end{tikzpicture},~~\quad
         \begin{tikzpicture}[baseline = 2.5mm]
	\draw[-,thick,darkblue] (-0.28,0) to[out=90,in=-90] (0.28,.6);
	\draw[-,line width=4pt,white] (0.28,0) to[out=90,in=-90] (-0.28,.6);
	\draw[-,thick,darkblue] (0.28,0) to[out=90,in=-90] (-0.28,.6);
   \node at (-0.26,0.5) {$\color{darkblue}\scriptstyle\circ$};
\end{tikzpicture}= \begin{tikzpicture}[baseline = 2.5mm]
\draw[-,thick,darkblue] (0.28,0) to[out=90,in=-90] (-0.28,.6);
	\draw[-,line width=4pt,white] (-0.28,0) to[out=90,in=-90] (0.28,.6);
	\draw[-,thick,darkblue] (-0.28,0) to[out=90,in=-90] (0.28,.6);
 \node at (0.26,0.1) {$\color{darkblue}\scriptstyle\circ$};
\end{tikzpicture},~~\quad
\begin{tikzpicture}[baseline = 2.5mm]
	\draw[-,thick,darkblue] (0.28,0) to[out=90,in=-90] (-0.28,.6);
	\draw[-,line width=4pt,white] (-0.28,0) to[out=90,in=-90] (0.28,.6);
	\draw[-,thick,darkblue] (-0.28,0) to[out=90,in=-90] (0.28,.6);
   \node at (0.26,0.5) {$\color{darkblue}\scriptstyle\circ$};
\end{tikzpicture}= \begin{tikzpicture}[baseline = 2.5mm]
	\draw[-,thick,darkblue] (-0.28,0) to[out=90,in=-90] (0.28,.6);
	\draw[-,line width=4pt,white] (0.28,0) to[out=90,in=-90] (-0.28,.6);
	\draw[-,thick,darkblue] (0.28,0) to[out=90,in=-90] (-0.28,.6);
 \node at (-0.26,0.1) {$\color{darkblue}\scriptstyle\circ$};
\end{tikzpicture},~~\quad
\begin{tikzpicture}[baseline = 2.5mm]
	\draw[-,thick,darkblue] (-0.28,0) to[out=90,in=-90] (0.28,.6);
	\draw[-,line width=4pt,white] (0.28,0) to[out=90,in=-90] (-0.28,.6);
	\draw[-,thick,darkblue] (0.28,0) to[out=90,in=-90] (-0.28,.6);
   \node at (0.26,0.5) {$\color{darkblue}\scriptstyle\bullet$};
\end{tikzpicture}= \begin{tikzpicture}[baseline = 2.5mm]
	\draw[-,thick,darkblue] (0.28,0) to[out=90,in=-90] (-0.28,.6);
	\draw[-,line width=4pt,white] (-0.28,0) to[out=90,in=-90] (0.28,.6);
	\draw[-,thick,darkblue] (-0.28,0) to[out=90,in=-90] (0.28,.6);
 \node at (-0.26,0.1) {$\color{darkblue}\scriptstyle\bullet$};
\end{tikzpicture}.\end{equation}
By a movement of type II, we mean that we exchange  two diagrams via each equality in  \eqref{111}.
So, any set obtained from ${\mathbb {DB}}_{a,b}$ by applying movements of types I$'$ and II is also a basis of $A_{b, a} $.
Note that  any  movement of  type I can be obtained by applying the movements of  types I$'$ and II.    So,
$\mathbb N_{a,b}$ can also be obtained from $\mathbb {DB}_{a,b}$ via   movements of types I$'$ and II. This proves that
   $\mathbb N_{a,b}$ is   a basis of $A_{b,a}$ and
   (W4) follows.
\end{proof}
We are going to explain that Assumption~\ref{pdual} and (A1)-(A10)  hold for the $\Bbbk$-algebra $A$ associated to $\CK(\mathbf u)$ (in this case, $\mathcal C=\mathcal C_2=\AK$) provided that the $\mathbf u$-admissible condition holds. If so, then all the results in sections~2--5 can be applied to $\CK(\mathbf u)$.

In fact, (A1) follows from Proposition~\ref{CKWA}. (A2) follows if we  define
\begin{equation}\label{dfxyh1}
  Y(b,a)=  \mathbb {DB}\cap A^-_{b, a} ,
   ~~H(a)=\mathbb {DB}\cap  A_a^\circ,
   ~~X(a,b)=\mathbb {DB}\cap  A^+_{a, b}.\end{equation}
To prove (A3), we define $$
\begin{aligned}
           t_i 1_{a}=1_{a}t_i= \begin{tikzpicture}[baseline = -0.5mm, color=\clr]
            \draw[-,thick,darkblue](-1,-0.3)to[out= down,in=up](-1,0.3);
             \draw[-,thick,darkblue](-0.6,-0.3)to[out= down,in=up](-0.6,0.3);
              \draw(-0.8,-0.25) node{$ \cdots$}; \draw(-0.8,0.25) node{$ \cdots$};
              \draw[-,thick,darkblue](1,-0.3)to[out= down,in=up](1,0.3);
             \draw[-,thick,darkblue](0.6,-0.3)to[out= down,in=up](0.6,0.3);
              \draw(0.8,-0.25) node{$  \cdots $}; \draw(0.8,0.25) node{$ \cdots$};
	\draw[-,thick,darkblue] (0.28,-0.3) to[out=90,in=-90] (-0.28,.3);
	\draw[-,line width=4pt,white] (-0.28,-0.3) to[out=90,in=-90] (0.28,.3);
	\draw[-,thick,darkblue] (-0.28,-.3) to[out=90,in=-90] (0.28,.3);
 \draw(1,0.4)node{\tiny$a$}; \draw(0.28,0.4)node{\tiny$i+1$}; \draw(-0.28,.4)node{\tiny$i$};
  \end{tikzpicture}
     \in A_a
           \end{aligned}
          $$
           for any admissible $i, a\in \mathbb N$. Let $$D(a+1)=\left\{t_{i, a+1}X_{a+1}^j| 1\leq i \leq a+1, -\lfloor \frac{m}{2}\rfloor\leq j\leq \lfloor \frac{m-1}{2}\rfloor\right\}\subseteq A_{a+1},$$ where $X_i$ is in \eqref{xitil} and $t_{i, a+1}=t_i \circ t_{i+1}\circ \cdots\circ t_{a}$ if $i<a+1$ and $t_{a+1, a+1}=1$. Then (A3)-(A4) can be proved by arguments similar to those for the cyclotomic Brauer category $\CB(\mathbf u)$ in subsection~\ref{cba}. The only difference is that  we replace  \eqref{cup-dot} by
            \begin{tikzpicture}[baseline = -0.5mm]
 \draw[-,thick,darkblue] (0,0) to[out=down,in=left] (0.28,-0.28) to[out=right,in=down] (0.56,0);
 \node at (0,0) {$\color{darkblue}\scriptstyle\bullet$};
  \draw[-,thick,darkblue] (0,0) to (0,.2);
   \draw[-,thick,darkblue] (0.56,0) to (0.56,.2);
 \end{tikzpicture}
 ~=~
  \begin{tikzpicture}[baseline = -0.5mm]
 \draw[-,thick,darkblue] (0,0) to[out=down,in=left] (0.28,-0.28) to[out=right,in=down] (0.56,0);
 \node at (0.56,0) {$\color{darkblue}\scriptstyle\circ$};
  \draw[-,thick,darkblue] (0,0) to (0,.2);
   \draw[-,thick,darkblue] (0.56,0) to (0.56,.2);
 \end{tikzpicture} in \cite[(1.9)]{GRS} and    \begin{tikzpicture}[baseline = -0.5mm]
 \draw[-,thick,darkblue] (0,0) to[out=down,in=left] (0.28,-0.28) to[out=right,in=down] (0.56,0);
 \node at (0,0) {$\color{darkblue}\scriptstyle\circ$};
  \draw[-,thick,darkblue] (0,0) to (0,.2);
   \draw[-,thick,darkblue] (0.56,0) to (0.56,.2);
 \end{tikzpicture}
 ~=~
 \begin{tikzpicture}[baseline = -0.5mm]
 \draw[-,thick,darkblue] (0,0) to[out=down,in=left] (0.28,-0.28) to[out=right,in=down] (0.56,0);
 \node at (0.56,0) {$\color{darkblue}\scriptstyle\bullet$};
  \draw[-,thick,darkblue] (0,0) to (0,.2);
   \draw[-,thick,darkblue] (0.56,0) to (0.56,.2);
 \end{tikzpicture} in \cite[Lemma~1.5(3)]{GRS} when we prove (A3).

The required  $\Bbbk$-linear  monoidal contravariant functor is $\sigma: \mathcal \AK\rightarrow \mathcal \AK$, which  switches $\lcup$ and $\lcap$ and fixes the generating object \begin{tikzpicture}[baseline = 1.5mm]
	\draw[-,thick,darkblue] (0.18,0) to (0.18,.4);
\end{tikzpicture} and all other  generating   morphisms. So,
 $\sigma^2=\text{Id}$ and $\sigma$
   induces an anti-involution $\sigma_A$ on $A$, and hence  (A5) follows.
 By \cite[Corollary~6.22]{GRS}, there is an algebra isomorphism $ \phi:  A_{a}    \cong \mathscr W_{m, a}$ for any $a\in\mathbb N$, where  $\mathscr W_{m, a}$ is the cyclotomic
 Birman-Wenzl-Murakami algebra  in \cite{Olden}. When $m=1$, it is the
 Birman-Wenzl-Murakami algebra in \cite{BW}.
 Moreover, $ \bar A_a=\mathscr W_{m, a}/\phi(N_a)=\mathscr H_{m,a}(\mathbf u)$,   where
  $N_a$ is the two-sided ideal of $A_a$ generated by $ \begin{tikzpicture}[baseline = 15pt, scale=0.35, color=\clr]
        \draw[-,thick] (0,2.2) to[out=down,in=left] (0.5,1.7) to[out=right,in=down] (1,2.2);
         \draw[-,thick] (0,1) to[out=up,in=left] (0.5,1.5) to[out=right,in=up] (1,1);
         \draw[-,thick](1.5,1.1)to[out= down,in=up](1.5,2.2);
         \draw(2.5,1.1) node{$ \cdots$}; \draw(2.5,2.2) node{$ \cdots$};
         \draw[-,thick](3.5,1.1)to[out= down,in=up](3.5,2.2);
           \end{tikzpicture}~ $ and $\mathscr H_{m,a}(\mathbf u)$   is   the  cyclotomic  Hecke
algebra in subsection~\ref{hecke}. By arguments similar to those for the degenerate cyclotomic Hecke algebras as above,   Assumption~\ref{pdual} and  (A6)-(A10)  follow from corresponding    well-known results on cyclotomic Hecke algebras (e.g., $\mathscr H_{m,a}(\mathbf u)$ is  a symmetric algebra \cite{MM}). In particular, the function  $ \sharp $ in Definition~\ref{sharpind} satisfies   $i^\sharp=i^{-1}$ and  $\mathbb I_0=\mathbb I_\mathbf u$ in Lemma~\ref{dddddcate}.
We omit details here.

\begin{Theorem}\label{CBMW1} Let $A$ be the $\Bbbk$-algebra  associated to the cyclotomic Kauffman category $\CK(\mathbf u)$. Let $e$ be the quantum characteristic of $q^2$.  Suppose   the $\mathbf u$-admissible condition holds.
 \begin{itemize}\item [(1)] The complete set of pairwise inequivalent  simple $A$-modules are indexed by $\bar\Lambda(m)$, i.e., the set of all $\mathbf u$-Kleshchev $m$-partitions in the sense of \cite{Ari1}.
 \item [(2)] $A$-lfdmod is an upper finite fully  stratified category in the sense of \cite[Definition~3.36]{BS} with respect to the stratification $\rho:\bar\Lambda(m)\rightarrow \mathbb N$ such that $\rho(\lambda)=n$ if $\lambda\in \bar\Lambda_{m,n}$.
  \item [(3)]$A$-lfdmod is an upper finite highest weight  category in the sense of \cite[Definition~3.36]{BS} if $ e=\infty $  and  $u_i u_j^{-1} \not\in q^{2\mathbb Z} $ if $1\le i<j\le m$.
     \item [(4)] $A$ is Morita equivalent to $\bigoplus_{n=0}^\infty \mathscr  H_{m, n}(\mathbf u)$ if $u_i u_j\not\in q^{2\mathbb Z}$ for all $1\le i\le j\le m$.
         \item [(5)] $A$ is semisimple over $\Bbbk$  if and only if   $e=\infty $ and $u_iu_j\not\in q^{2\mathbb Z } $, $1\le i,j\le m$ and  $u_i u_j^{-1} \not\in q^{2\mathbb Z} $ if $1\le i<j\le m$.\end{itemize}
\end{Theorem}
\begin{proof}(1) follows from Proposition~\ref{CKWA} and Theorem~\ref{striateddndn}(3) together with Ariki's result on the classification of simple modules for   cyclotomic Hecke algebras in \cite{Ari1}.
(2) follows immediately from Proposition~\ref{CKWA} and Theorem~\ref{ext1}. (3) follows from Theorem~\ref{newstratification} and \cite[Main Theorem]{Ari0} which gives  a necessary and sufficient condition for cyclotomic Hecke algebras  being semisimple over $\Bbbk$.
 (4) follows from Theorem~\ref{sjdsdh} since $\mathbb I_0\cap \mathbb I_0^\sharp=\emptyset$.  Finally, (5) follows from Proposition~\ref{CKWA}, Theorem~\ref{ss} and \cite[Theorem~5.9]{RuiS} and \cite[Theorem~6.5]{RuiS2} in which semisimple criteria on  Birman-Murakami-Wenzl algebras and
 cyclotomic Birman-Murakami-Wenzl algebras are given over an arbitrary field.
 \end{proof}

Suppose that $\mathfrak g=\mathfrak {sl}_{\mathbb I_\mathbf u}$ in Lemma~\ref{dddddcate} and
$\omega_{\mathbf u}=\sum_{i=1}^m\omega_{u_i}$, where   $\omega_j$'s   are fundamental weights indexed by $j\in\mathbb I_\mathbf u$. Recall the Lie subalgebra $\mathfrak g^\sharp$ of $\mathfrak g$ in Theorem~\ref{gsharpiso} with $\mathbb I_0=\mathbb I_\mathbf u$.
The following result follows from  arguments similar to those  in the proof of Theorem~\ref{gsharpisobrauercyc}.
\begin{Theorem}\label{gsharpisoKaucyc}Let $A$ be the $\Bbbk$-algebra  associated to the cyclotomic Kauffman category $\CK(\mathbf u)$ such that  the $\mathbf u$-admissible condition holds.
As $\mathfrak g^\sharp$-modules, $[A\text{-mod}^\Delta]\cong V(\omega_{\mathbf u})$.
The corresponding isomorphism  $\phi^\sharp$ satisfies
$\phi^\sharp([\Delta(\lambda)])=\sum_{\mu\in  \Lambda(m)} [S(\mu): D(\lambda)]v_\mu$ for any $\lambda\in\bar\Lambda(m)$.
Moreover, $\tilde e_i $ acts via the exact functor  $E_i$   in Theorem~\ref{usuactifuc1} for any $i\in\mathbb I$.
\end{Theorem}

\begin{rem}\label{jsjxsss}
Suppose that $\lambda,\mu\in\bar\Lambda(m)$. The multiplicity of $(P(\lambda):\Delta(\mu))$ can be solved via the equations in
Proposition~\ref{xeijdiec1}(2). So,   $\phi^\sharp([P(\lambda)])$ in  Theorem~\ref{gsharpisobrauercyc} (resp., Theorem~\ref{gsharpisoKaucyc}) is   determined by the decomposition numbers
$[1_a\tilde \Delta(\nu): 1_aL(\gamma)]$ (which is equal to $[\tilde S_a(\nu):1_aL(\gamma)]$ by Proposition~\ref{xeijdiec1}(1)) of $A_a$ for all $a\in\mathbb N$.
Recall  $\bar A_a\cong H_{m,a}(\mathbf u)$ (resp., $\mathscr H_{m,a}(\mathbf u)$). Using the well-known cellular basis of $H_{m,a}(\mathbf u) $ in \cite[Theorem~6.3]{AMR} (resp., $\mathscr H_{m,a}(\mathbf u)$ in \cite[Theorem~3.26]{DJM}),  one can check that
the cell module $\tilde S_a(\lambda)$ of $A_a$ in \eqref{cellbasis1}  coincides with the cell modules $C(f,\lambda)$  in \cite{RuiS1} for cyclotomic Nazarov-Wenzl algebras (resp.,  in \cite{RX} for  cyclotomic Birman-Murakami-Wenzl algebras) for all $\lambda\in \bigcup_{k \succeq a}\Lambda_{m,k}$.
If $\Bbbk=\mathbb C$, the decomposition numbers of $H_{m,a}(\mathbf u)$, $\mathscr H_{m,a}(\mathbf u)$, Brauer algebras\cite{CV} and Birman-Murakami-Wenzl algebras \cite{RS2} and  cyclotomic Nazarov-Wenzl  algebras \cite{RS3} are known.
 In the remaining cases, we have no further information.
\end{rem}

\subsection{Blocks  } Let $A$ be the $\Bbbk$-algebra  associated to either the cyclotomic Brauer category $\CB(\mathbf u)$  or the cyclotomic Kauffman category $\CK(\mathbf u)$ such that the $\mathbf u$-admissible conditions hold. Since Assumption~\ref{pdual} and (A1)-(A10) hold for $A$, Theorem~\ref{hlink} gives a partial results on blocks.

\begin{Prop}\label{block} For any $\lambda, \mu\in \bar \Lambda(m)$, $L(\lambda)$ and $L(\mu)$ are in the same
block     if and only if  there is a sequence $\lambda=\lambda_1,\lambda_2,\ldots,\lambda_n=\mu$ such that each pair
  $\lambda_{i}$ and $\lambda_{i+1}$ are cell-link in a weakly   cellular algebra  $A_a$ for some $a\in\mathbb N$.
\end{Prop}
\begin{proof}By Proposition~\ref{blockp}, it suffices to prove only if part of this result.
 Suppose $\lambda\in \bar \Lambda_{m, c}, \mu\in \bar \Lambda_{m, d}$ and
 \begin{equation}\label{kst1}  \Hom_A(P(\lambda),P(\mu))= [P(\mu):L(\lambda)]\neq0.\end{equation}
There are three cases we have to discuss as follows.

First, we assume
  $\min\{c, d\} >0$. Thanks to  \eqref{b3}
there is at least a $\nu\in \Lambda(m)$
such that $ [\tilde \Delta(\nu):L(\mu)] [\tilde \Delta(\nu): L(\lambda)]\neq 0$.
So $ \rho(\nu) \succeq  d,  c$.  Let $ k =\max\{  c,  d\}$. Then $k \neq 0$.
Note that the cell modules $1_k\tilde \Delta(\nu)$'s  coincide with cell modules of $A_k$ studied in \cite{RuiS1} and \cite{RX} (see Remark \ref{jsjxsss}).
 By \cite[Theorem~3.12]{RuiS1} and \cite[Theorem~5.3]{RX}, $\bigcup_{0 \neq  n\succeq k}\bar \Lambda_{m, n}\subseteq\tilde\Lambda_{\succeq k}$, where $\tilde\Lambda_{\succeq  k}\subseteq \Lambda_{\succeq k}$  parameterizes  inequivalent simple $A_{k} $-modules and $\Lambda_{\succeq k}=\bigcup_{n\succeq k}  \Lambda_{m, n}$.  Thanks to  Proposition~\ref{xeijdiec1}(3),  $1_k L(\lambda)\neq 0$ and $1_k L(\mu)\neq 0$.  Hence $[1_k\tilde \Delta(\nu):1_k L(\mu)] [1_k\tilde \Delta(\nu): 1_k L(\lambda)]\neq 0$, and  $\lambda, \nu, \mu$ is the required sequence.

 Secondly, we assume $ c=0$. Then   $\lambda=\emptyset$.
Thanks to Proposition~\ref{dektss1}(1) and \eqref{kst1} , there is at least a $\nu$ such that  $[\Delta (\nu): L(\lambda)](P(\mu):\Delta (\nu))\neq 0$.
Since $1_{0}  L(\lambda)\cong \Bbbk$,  by Lemma~\ref{exactj}(1),  $1_0 \Delta (\nu)\neq 0$ only if $\nu=\emptyset$.  By Proposition~\ref{dektss1}(2), Lemma~\ref{isodual1}(2) and \eqref{kst1},
      $[\bar\Delta (\emptyset): L(\mu)]=(P(\mu):\Delta (\emptyset))\neq 0$  and hence $[\tilde\Delta (\emptyset): L(\mu)]\neq 0$. Therefore
      $      [1_d\tilde\Delta (\emptyset): 1_dL(\mu)]\neq 0$. By Proposition~\ref{xeijdiec1}, $\lambda, \mu$ is the required sequence.

       Finally, we assume  $d=0$, and hence  $\mu=\emptyset$.
By Proposition~\ref{dektss1}, $P(\emptyset)=\Delta(\emptyset)$. Since $P_0(\emptyset)=S_0(\emptyset)=\Bbbk$, We have $\Delta(\emptyset)=\tilde \Delta(\emptyset)$. Therefore,
 $[\tilde\Delta (\emptyset): L(\lambda)]=[P(\emptyset): L(\lambda)]\neq 0$ and
 $[1_c\tilde\Delta (\emptyset): 1_cL(\lambda)]\neq 0$. By Proposition~\ref{xeijdiec1}, $ \lambda,\mu$ is the required sequence.
\end{proof}

 When the  Hecke algebra (resp., (degenerate) cyclotomic Hecke algebra) is semisimple over $\Bbbk$ (with characteristic not two), a necessary and sufficient condition on cell-link for the Birman-Murakami-Wenzl algebra (resp.,  the cyclotomic Nazarov-Wenzl algebra and the cyclotomic Birman-Murakami-Wenzl algebra) is given in
 \cite[Propositions~4.6,~4.10]{RuiS3} (resp.,  \cite[Propositions~5.10,~5.21] {RuiS4}). Using Proposition~\ref{block} yields an explicit combinatorial description on the  blocks of $A$. We omit details since it is a routine work.   In the remaining cases,  we have no further information.

\end{document}